\documentclass{amsart}
\usepackage{amsmath, amssymb,pdfsync}
\usepackage{mathrsfs}
\usepackage{blindtext}
\usepackage[citecolor=blue]{hyperref}

\theoremstyle{plain}
\newtheorem{theorem}[subsection]{Theorem}
\newtheorem{proposition}[subsection]{Proposition}
\newtheorem{lemma}[subsection]{Lemma}
\newtheorem{corollary}[subsection]{Corollary}

\theoremstyle{definition}
\newtheorem{conjecture}[subsection]{Conjecture}
\newtheorem{definition}[subsection]{Definition}

\theoremstyle{remark}

\newtheorem*{remark*}{Remark}
\newtheorem*{example*}{Example}
\numberwithin{equation}{subsection}

\begin{document}

\title[Infinite Reduced Words, Lattice Property And Braid Graph]{Infinite Reduced Words, Lattice Property And Braid Graph of Affine Weyl Groups}
\author{Weijia Wang}
\address{School of Mathematics (Zhuhai)
\\ Sun Yat-sen University \\
Zhuhai, Guangdong, 519082 \\ China}
\email{wangweij5@mail.sysu.edu.cn}

\begin{abstract}
In this paper, we establish a bijection between the
  infinite reduced words of an affine Weyl group and certain biclosed
  sets of its positive system and determine all finitely generated biclosed sets in the positive system of an affine Weyl group. Using these results, we show firstly that
  the biclosed sets in the standard positive system of rank 3 affine Weyl groups when ordered by inclusion form a
  complete algebraic
  ortholattice and secondly that the (generalized) braid graphs of those Coxeter groups are connected, which can be thought of as an infinite version of Tit's solution to the word problem.
\end{abstract}

\maketitle

\section*{Introduction}\label{sectIntro}

Given an arbitrary Coxeter system $(W,S)$ one can attach a root system to it in a canonical way.
A particular total ordering on the set of the positive roots called a \emph{reflection order}
is the notion analogous to the reduced expression of the longest element of a finite Coxeter group.
Such orderings encapsulate important combinatorics of the Iwahori-Hecke algebra, Kazhdan-Lusztig polynomials and $R$-polynomials.
For example, reflection orders can be used to generalize the symmetry of the structure constants of the Iwahori-Hecke algebra associated to an infinite Coxeter group by playing the role of the reduced expressions of the longest element of a finite Coxeter group. See for example \cite{dyerhecke}, \cite{DyerTwistedBruhat}.
Reflection orders can also be used to provide nonrecursive combinatorial formulae for  Kazhdan-Lusztig polynomials and  $R$-polynomails. See \cite{bjornerbrenti} Chapter 5.

However, the structure of the reflection orders of an infinite Coxeter group are in general not well understood. Several questions about the reflection orders and their initial sections are still open. Among them are two long-standing conjectures of Dyer, the first of which relates the initial sections of the reflection orders to the subsets of the positive system that are both closed and coclosed under a certain natural closure operator (these sets will be called \emph{biclosed sets}).  The first conjecture asserts that any maximal chain of biclosed sets arises from the set of all initial sections
of some reflection order and that a set is biclosed if and only if it is an initial section of some reflection order. The second conjecture asserts that the initial sections of all reflection orders when partially ordered by inclusion form a complete lattice.
These two conjectures can be motivated by the analogy between the reflection orders and the basic properties of the reduced expressions by considering the initial sections as elements of possibly ``infinite length''. With this point of view, the second conjecture is a natural generalization to the set of initial sections of the classical result that the weak order of a finite Coxeter group is a lattice  and the first conjecture is analogous to the fact that every element has a reduced expression. The conjectures are true for all finite Coxeter groups (where they reduce to well-known properties of weak order on $W$) and  can be easily verified  for the infinite dihedral case ($\widetilde{A}_1$).

In \cite{DyerReflOrder}, Dyer proves the first conjecture in the case of affine Weyl groups.
Provided that the first conjecture holds, the second conjecture can be restated as follows: the set of all biclosed sets when partially ordered by inclusion forms a complete lattice.
Such a conjecture stimulates many studies. In \cite{DyerWeakOrder}, some evidence in support of the conjectures is given and the conjectural join and meet of two biclosed sets are described in terms of the closure operator. In \cite{labbe}, the author uses a slightly different convex closure operator and proves a variant of the conjecture in rank three cases. In \cite{Hohlw}, the authors characterize
 the existence of the join of a set of elements in $W$ by the corresponding inversion sets
and their position relative to the imaginary cone. In \cite{viard}, the author studies the conjecture using his theory of projective valued digraphs.

In this paper, we study the second conjecture in the case of rank three affine  Weyl groups.
To do this, we first consider the poset of those biclosed sets that are inversion sets of $W$ and the infinite reduced expressions. We show that such a poset is a complete meet semilattice. Then we investigate the infinite reduced words of an affine Weyl group and characterize the biclosed sets which are the inversion sets of infinite reduced words. An immediate consequence of this characterization is that for rank three affine Weyl groups any biclosed set is either an inversion set of some (finite or infinite) word or the complement of an inversion set.
This special property makes the conjecture for rank three affine Weyl groups more accessible.
To be precise, in rank three affine cases the poset of the biclosed sets splits into a lower half (consisting of the inversion sets of finite or infinite words) and a top half (consisting of the complements of the inversion sets of finite or infinite words). The key ingredient to show the conjecture is to prove that in the lower half, a family of biclosed sets is either bounded or its orthogonal family is bounded. We prove this property by some case studies. In this way we show that for rank three affine Weyl groups Dyer's conjecture holds. We also classify the finitely generated biclosed sets in the positive system of an affine root system and this allows us to conclude that the lattices under study are algebraic. The paper is organized as follows.

In section \ref{sect:infbasics}, we consider the set $\overline{W}$ consisting of the infinite reduced words and the elements in the group $W$. The weak order on $W$ is naturally extended to $\overline{W}$. We study such an order and  show that $\overline{W}$ is a complete meet semilattice which is a complete lattice if and only if $W$ is the weak direct product of countably many finite or locally finite Coxeter groups. We also show that $\overline{W}$ has a  Join Orthogonality Property and admits maximal elements provided the rank of $W$ is at most infinite countable.

In section \ref{sec:bijection}, We specialize to the case of affine Weyl groups and establish a ($W-$equivariant) bijection between the infinite reduced words and the certain biclosed sets in the positive system. By using this bijection we give a characterization of affine Weyl groups: a finite rank irreducible infinite Coxeter system is affine if and only if $\overline{W}$ admits finitely many maximal elements.

A consequence of the classification obtained in section \ref{sec:bijection} is that for rank three affine positive systems, the poset of the biclosed sets splits into two halves where the lower half consists of the inversion sets of the (finite or infinite) words. Section \ref{sec:lattice} is devoted to  proving that the lower half of the poset has the following favorable property: a family of its elements is either bounded or its orthogonal family is bounded. By using this property we show the existence of the join of two biclosed sets and therefore prove that the poset of biclosed sets is a lattice. We also prove that in the positive system of an affine root system if the closure of the union of a family of finite biclosed sets is finite then it is biclosed.

In section \ref{sec:compact}, we study those biclosed sets in the positive system of an affine root system that are finitely generated. In rank three cases they are the compact elements of the lattice of biclosed sets. Our results show that those lattices are algebraic.

In section \ref{sec:braid}, we show that the (generalized) braid graphs (defined using reflection orders) of rank three affine Weyl groups are all connected, which generalizes the fact that the reduced expressions of a (finite) word are connected by braid moves. In this infinite setting, a biclosed set of the standard positive system can be thought of as an element of ``infinite length''. The reflection order is a substitute for reduced expression. A braid move is characterized by the reversal of a dihedral string. The proof makes use of the explicit description of the biclosed sets as either inversion sets or their complements in affine rank three cases.

\section{Preliminaries}\label{sectpreli}
A partially ordered set $L$ is a \emph{meet semilattice}  if
any two elements $x,y\in L$ admit a greatest lower bound
(meet), denoted by $x\wedge y$. A partially ordered set $L$ is called a \emph{complete meet
semilattice} if any subset $A\subset L$ admits a greatest
lower bound (meet), denoted by $\bigwedge A.$ A partially ordered set $L$ is a \emph{join semilattice}
if any two elements $x,y\in L$ admit a least upper bound
(join), denoted by $x\vee y$. A partially ordered set $L$ is called a \emph{complete join semilattice}
if any subset $A\subset L$  admits a least upper bound
(join), denoted by $\bigvee A.$ If a partially ordered set  is both a meet semilattice and a
join semilattice, it is called a \emph{lattice}. If a partially ordered set  is both a complete
meet semilattice and a complete join  semilattice, it is called a
\emph{complete lattice}. An \emph{ortholattice} $L$ is a lattice $L$ such that it has a minimum element,
denoted by 0, and a maximum element, denoted by 1 and there exists
an order-reversing involution on $L$, denoted by $x\mapsto x^{oc}$
such that $x\wedge x^{oc}=0$ and $x\vee x^{oc}=1$ for all
$x\in L$.

An element $x$ in a complete lattice $L$ is called \emph{compact} if $x\leq
\bigvee Y$ for some $Y\subset L$ implies that $x\leq \bigvee Y_0$
for some finite $Y_0\subset Y.$ A complete lattice is called
\emph{algebraic} if every element of it is a join of a set of compact
elements.

We record the following easy lemma that will be used frequently in the paper.

\begin{lemma}\label{techsemijoin}
Let $(P,\leq)$ be a complete meet semilattice. Suppose that $A$ is a bounded subset of $P$. Then the join of $A$ exists in $P$.
\end{lemma}

For a totally ordered set $T$, a set $I\subset T$ is called an \emph{initial section} if for all $x\in I, y\in T\backslash I$ we have that $x\leq y$. A set $F\subset T$ is called a \emph{final section} if for all $x\in F, y\in T\backslash F$ we have that $x\geq y.$

Let $(P,\leq)$ be a partially ordered set. A subset $Q$ of $P$ is said to be a \emph{totally ordered subset} (\emph{chain}) of $P$ if the restriction of $\leq$ to $Q$ is a total order. A totally ordered subset $Q$ of $P$ is called a \emph{maximal totally ordered subset} if for any $x\in P\backslash Q$, $Q\cup \{x\}$ is not a totally ordered subset.
A poset $(P,\leq)$ is said to be \emph{chain complete} if any chain of it has a least upper bound and a greatest lower bound.

Let $(W,S)$ be a Coxeter system with root system $\Phi$, see Chapter 5 of
\cite{Hum} for example. Note that in this paper we allow $S$ to be infinite. The set of positive roots and negative roots
are denoted by $\Phi^+$ and $\Phi^-$ respectively. The set of simple
roots is denoted by $\Pi$. An effective way to visualize the infinite root system in lower rank is to consider the projective representation of $W$. The idea is to project all the roots in the standard reflection representation to an affine hyperplane which is transverse to $\Phi^+.$ For the projective representation see Section 2 of \cite{DyerLimitRoot}, Section 2 of \cite{Hohlw2} and Section 2.3 of \cite{labbe}. Given $w\in W$ define the \emph{inversion set} of
$w^{-1}$ to be $\Phi_w:=\{\alpha\in \Phi^+|w^{-1}(\alpha)\in
\Phi^-\}.$ Let $T:=\{wsw^{-1}|w\in W, s\in S\}.$ The set $T$ is called the set of \emph{reflections}.
There exists a canonical bijection between $T$ and $\Phi^+$ and denote by $s_{\alpha}$ the reflection corresponding to a positive root $\alpha$.
Let $W'$ be a subgroup of $W$ generated by a subset of $T$. Then $W'$ is called a \emph{reflection subgroup}.
It can be shown that $(W',\{t\in T\cap W'|l(t't)>l(t), \forall t'\in T\cap W'\backslash \{t\}\})$ is a Coxeter system. For reflection subgroups see Section 3 of  \cite{DyerReflSubgrp}.

The \emph{(right) weak  order} $\leq$ on $W$ is defined as $x\leq y$ if and
only if $\Phi_x\subset \Phi_y$ for all $x,y\in W$. The poset $(W,\leq)$ is a
complete meet semilattice. It is a complete lattice if and only if
$W$ is finite. See Chapter 3 of \cite{bjornerbrenti}.

A set $\Gamma\subset \Phi^+$ is called \emph{closed} (in $\Phi^+$) if for
all $\alpha,\beta\in \Gamma$, $k_1\alpha+k_2\alpha\in
\Phi^+$ with $k_1,k_2\in \mathbb{R}_{\geq 0}$ implies that
$k_1\alpha+k_2\alpha\in \Gamma$. A subset $\Gamma$ is called \emph{coclosed} (in $\Phi^+$) if $\Phi^+\backslash \Gamma$ is closed (in $\Phi^+$). A set $\Gamma\subset \Phi^+$ is called
\emph{biclosed} (in $\Phi^+$) if $\Gamma$ and $\Phi^+\backslash \Gamma$ are
both closed (in $\Phi^+$). Replacing $\Phi^+$ with $\Phi$ in the
above definition, we define the closed sets, coclosed sets and biclosed sets in
$\Phi$. We usually denote by $\overline{X}$ the closure of a set $X$. Such a closure operator is well-defined (since the intersection of closed sets
stays closed).

It is known that the finite biclosed sets in $\Phi^+$ are precisely
the inversion sets  of the elements in the Coxeter group (See \cite{DyerWeakOrder} Lemma 4.1(d)).

Denote by $\mathscr{B}(\Phi^+)$ the set of all biclosed sets in
$\Phi^+$, regarded as a partially ordered set under set inclusion. Two biclosed sets $B_1,B_2$ of this poset are said to be in the same \emph{block} if they differ by finitely many roots.

For a set $\Gamma$ biclosed in $\Phi^+$ and $x\in W$, define
$$x\cdot \Gamma:=(\Phi_x\backslash x(-\Gamma))\cup (x(\Gamma)\backslash
(-\Phi_x)).$$ Then $x\cdot \Gamma$ is again a biclosed set and this
is a group action of $W$ on the set of biclosed sets in $\Phi^+.$
See \cite{DyerWeakOrder} Lemma 4.1(a)(c).

A total order $\prec$ on $\Phi^+$
is called a \emph{reflection order} if for all $\alpha,\beta\in \Phi^+,\alpha\prec \beta$ and $a,b\in \mathbb{R}_{\geq 0}$ such that
$a\alpha+b\beta\in \Phi^+$ we have that $\alpha\prec a\alpha+b\beta\prec \beta.$ Suppose that $(W,S)$ is a dihedral Coxeter system with $S=\{s_{\alpha},s_{\beta}\}$. Then
$$\alpha\prec s_{\alpha}(\beta)\prec s_{\alpha}s_{\beta}(\alpha)\prec \cdots \prec s_{\beta}(\alpha)\prec \beta$$
is a reflection order.

A \emph{locally finite Coxeter system} $(W,S)$ is a Coxeter system such that $W$ is infinite
and each parabolic subgroup $W_J$ of it is finite if $J$ is finite.
Irreducible locally finite Coxeter systems are precisely the ones of
type $A_{\infty}, A_{\infty,\infty}, B_{\infty}, D_{\infty}.$ See
Exercise 4.14 of \cite{Kac}.

Now let $\Phi$ be an irreducible crystallographic root system of a
Weyl group $W$ and be contained in the Euclidean space $V$ as in
\cite{Bourbaki} Chapter VI, \S 1. A subset $\Gamma\subset \Phi$ is called
$\mathbb{Z}-$\emph{closed} if for all $\alpha,\beta\in \Gamma$ such that
$\alpha+\beta\in \Phi$, we have that $\alpha+\beta\in \Gamma.$

Let $\Delta$ be a simple system of $\Phi$ with the corresponding
positive system $\Phi^+$. Let $\gamma$ be the highest root. Suppose that
$\Delta^{'}$ and $\Delta^{''}$ are two subsets of $\Delta$ which are
orthogonal (i.e. $(\alpha,\beta)=0,\forall \alpha\in
\Delta^{'},\beta\in \Delta^{''}$). Define
$$\Phi^+_{\Delta^{'},\Delta^{''}}=(\Phi^+\backslash \Phi_{\Delta'})\cup \Phi_{\Delta''}$$
where $\Phi_{\Delta'}$ is the root subsystem of $\Phi$ generated by $\Delta'.$

In \cite{DyerReflOrder} (Theorem 1.15) and in \cite{biclosedphi} (Theorem 4) all  biclosed sets in $\Phi$ are
classified:

\begin{theorem}\label{theorem:classifiybiclosedweyl}
Let $\Gamma\subset \Phi$. The set $\Gamma$ is biclosed in $\Phi$ if and
only if there exists a simple system $\Delta$ with the corresponding positive
system $\Psi^+$ and  orthogonal subsets
$\Delta^{'},\Delta^{''}\subset \Delta$ such that
$\Gamma=\Psi^+_{\Delta^{'},\Delta^{''}}$.
\end{theorem}

Define a real vector space $V'=V\oplus\mathbb{R}\delta$ and extend
the inner product on $V$ to $V'$ by requiring $(\delta,V')=0.$ For
$\alpha\in \Phi^+$, we define
$\widehat{\alpha}=\{\alpha+n\delta|n\in \mathbb{Z}_{\geq 0}\}\subset
V'$. For $\alpha\in \Phi^-$, we define
$\widehat{\alpha}=\{\alpha+(n+1)\delta|n\in \mathbb{Z}_{\geq
0}\}\subset V'$. For a set $\Gamma\subset \Phi$, define
$\widehat{\Gamma}=\bigcup_{\alpha\in\Gamma}\widehat{\alpha}\subset
V'$.

Denote by $\widetilde{W}$ the (irreducible) affine Weyl group corresponding to $W$.
Then $\widetilde{W}$ has $\widehat{\Phi}\cup -\widehat{\Phi}$ as
root system with $\widehat{\Phi}(=\widetilde{\Phi}^+)$ as the set of positive roots.
The set of
simple roots $\widetilde{\Delta}=\Delta\bigcup \{\delta-\gamma\}$. For this construction, see Chapter 6 of \cite{Kac}  or Section 3.3 of
\cite{GusDyer}.
It is known that $\widetilde{W}=W\ltimes Q$ where $Q$ is the coroot lattice. Denote by $\pi$ the natural epimorphism from $\widetilde{W}$ to $W$.

In \cite{DyerReflOrder} (Proposition 5.11, Corollary 5.12) all  biclosed sets in $\widehat{\Phi}$
are classified:

\begin{theorem}\label{classifybiclosedsetaffine}
(1) Let $\Gamma\subset \widehat{\Phi}$. The set $\Gamma$ is biclosed in
$\widehat{\Phi}$ if and only if $\Gamma=w\cdot \widehat{\Lambda}$
where $\Lambda$ is biclosed in $\Phi$ and $w\in \widetilde{W}$.

(2) The biclosed sets in $\widehat{\Phi}$ that are in the same block as $\widehat{\Psi^+_{\Delta_1,\Delta_2}}$ are those of the form $w\cdot \widehat{\Psi^+_{\Delta_1,\Delta_2}}$ for some  $w\in W'<\widetilde{W}$ where $W'$ is the reflection subgroup generated by $\{s_{\alpha}|\alpha \in \widehat{\Psi_{\Delta_1\cup \Delta_2}}\}$. In particular $w\cdot \widehat{\Psi^+_{\Delta_1,\Delta_2}}=\widehat{\Psi^+_{\Delta_1,\Delta_2}}+(\Phi_{W'}^+\cap w(-\Phi_{W'}^+))$ where $+$ denotes the set symmetric difference and $\Phi_{W'}^+\subset \widehat{\Phi}$ is the positive system of $W'$.
\end{theorem}

It is worth noting that in (2) of the above Theorem $\widehat{\Psi_{\Delta_1}}$ and $\widehat{\Psi_{\Delta_2}}$ are orthogonal.
So, letting $W_i, i=1,2$ be the reflection subgroup generated by $\widehat{\Psi_{\Delta_i}}, i=1,2$, one has $W' = W_1 \times W_2$.

We also need the following results which are proved in \cite{DyerReflOrder} (Lemma 5.10(f), Proposition 5.20, Proposition 5.22).

\begin{theorem}\label{reflectionorderaffine}
(1) Any maximal totally ordered subset of $\mathscr{B}(\widehat{\Phi})$ is the set of all initial sections
of some reflection order of $\widehat{\Phi}$ and any biclosed set in $\widehat{\Phi}$ is an initial section of some reflection order.

(2)  Any maximal totally ordered subset of $\mathscr{B}(\widehat{\Phi})$ is also a maximal totally ordered subset of the poset $(\mathcal{P}(\widehat{\Phi}),\subset)$, the poset of all subsets of $\widehat{\Phi}$ under inclusion.

(3)  Denote by $\pi$  the canonical epimorphism from
$\widetilde{W}$ to $W$. Suppose that $\Psi^+$ is a positive system in $\Phi$ and $M,N$ are orthogonal subsets of its simple system. Let $w\in \widetilde{W}$. Then the biclosed sets $w\cdot \widehat{\Psi^+_{M,N}}$ and $\widehat{\pi(w)(\Psi^+_{M,N})}=\widehat{\Lambda^+_{M',N'}}$ (where $\Lambda^+=\pi(w)(\Psi^+)$, $M'=\pi(w)M$, $N'=\pi(w)N$) are in the same block in the poset $\mathscr{B}(\widehat{\Phi})$.
\end{theorem}

In this paper we denote the disjoint union by $\uplus$ and we denote the cardinality of a countably infinite set by $\aleph_0$.

\section{infinite reduced words of a Coxeter System}
\label{sect:infbasics}

Let $(W,S)$ be a Coxeter system with $W$ infinite.
\begin{definition} An \emph{infinite reduced expression} is a
sequence $$s_1s_2s_3\cdots$$ with $s_i\in S$ such that for all
$k\geq 1$, $s_1s_2\cdots s_k$ is a reduced expression.
\end{definition}

If $s_1s_2s_3\cdots$ is an infinite reduced expression then
immediately we have that
$$\Phi_{s_1}\subset \Phi_{s_1s_2}\subset \Phi_{s_1s_2s_3}\subset\dots.$$

\begin{definition}
Let $s_1s_2s_3\cdots$ be an infinite reduced expression.
Define
$$\Phi_{s_1s_2s_3\cdots}=\bigcup_{i=1}^{\infty}\Phi_{s_1s_2\cdots s_i}$$
and
$$\Phi_{s_1s_2s_3\cdots}'=\Phi^+\backslash \Phi_{s_1s_2s_3\cdots}.$$
\end{definition}
It is clear from the definition that $\Phi_{s_1s_2s_3\cdots}$ is
infinite and biclosed. Given two infinite reduced expressions
$s_1s_2s_3\cdots$ and $r_1r_2r_3\cdots$. Write
$s_1s_2s_3\cdots\thicksim r_1r_2r_3\cdots$ if
$\Phi_{s_1s_2s_3\cdots}=\Phi_{r_1r_2r_3\cdots}$. It is clear that
this is an equivalence relation.

\begin{definition}
Let $L$ be the set of all infinite reduced expressions.
Define $W_l=L/\thicksim$. An element in $W_l$ is called an
\emph{infinite reduced word} of $(W,S)$. Denote $\overline{W}=W_l\uplus
W.$ Given $w=[s_1s_2s_3\cdots]\in W_l$, $s_1s_2s_3\cdots$ is called
a \emph{reduced expression} of $w$.
\end{definition}

If no
confusion is possible, we write $w=s_1s_2s_3\cdots$ if
$w=[s_1s_2s_3\cdots]$.

\begin{definition}
Given $x,y\in \overline{W}$, define $x\leq y$ if and only if $\Phi_x\subset
\Phi_y$. Call this order the \emph{weak order} on $\overline{W}$. Let $x\in
W_l$. An element $w\in W$ is called a \emph{prefix} of $x$ if $w\leq x$.
Say two elements $x,y\in \overline{W}$ are \emph{orthogonal}, denoted by
$x\perp y$ if $\Phi_x\cap \Phi_y=\emptyset.$
\end{definition}

In the above definition we extend the weak order on $W$ to
$\overline{W}$ in an obvious way. Such
an order has been studied
by various authors, for example see \cite{Hohlw}, \cite{Lam1} and \cite{Lam}.

\begin{remark*}
It follows from the definition that we can define the weak order in terms of braid moves.
Let $w_1, w_2$ be two infinite reduced words. Suppose that the reduced expression $s_1s_2\cdots s_k$ (denoted by $\underline{u_k}$) appear as the first $k$ letters in a reduced expression of $w_1$. Then $w_1\leq w_2$ if and only if for any reduced expression of $w_2$ and any $k$, one can apply braid moves  finitely many times to convert such reduced expression to a reduced expression starting with $\underline{u_k}$.
\end{remark*}

\begin{example*}
Consider the Coxeter group $W$ of type $\widetilde{A}_2$, i.e. the group with presentation $\langle s_1,s_2,s_3|s_1^2=s_2^2=s_3^2=(s_1s_2)^3=(s_1s_3)^2=(s_2s_3)^2=e\rangle$.
Let $w_1=s_1s_2s_3s_1s_2s_3s_1s_2s_3\cdots$ and $w_2=s_1s_2s_1s_3s_1s_2s_1s_3s_1s_2s_1s_3\cdots$. Then $w_1\lvertneqq w_2$. To see this, one notes that any (finite) prefix of $w_1$ is also a (finite) prefix of $w_2$. But the prefix $s_1s_2s_1s_3$ of $w_2$ is not a prefix of $w_1.$
\end{example*}

\begin{definition}\label{def:sprod}
Let $s\in S$ and $w$ be an infinite reduced word of $W$. We define
$sw$ in the following way:

Case I: if $s\perp w$, we define $sw$ to be the word obtained by
concatenating $s$ with a reduced expression of $w$.

Case II: if $s\not\perp w$ (i.e. $s\in \Phi_w$), choose one reduced expression
$s_1s_2s_3\cdots$ of $w$ and one can find $u=s_1s_2\cdots s_k$ such
that $l(su)=l(u)-1$. We define $sw$ to be the word obtained by
concatenating one reduced expression of $su$ with the word
$s_{k+1}s_{k+2}\cdots$.
\end{definition}

This multiplication is well-defined
thanks to Lemma \ref{lem:infinitelongbasic} below.

\begin{definition}\label{def:wprod}
Let $u\in W$ and $w$ be an infinite reduced word. Let $s_1s_2\cdots
s_k$ be any (reduced or non-reduced) expression of $u$. Define $uw$
to be
$$s_1(\cdots(s_{k-1}(s_kw))\cdots)$$
\end{definition}

Again Lemma \ref{lem:infinitelongbasic} below ensures that this
multiplication is well-defined.

We summarize the basic properties of infinite reduced words in the
following lemmas. Their proofs reduce easily to the analogous
statements of Coxeter groups.
\begin{lemma}\label{lem:infinitelongbasic}
(a) Let $w_1=s_1s_2s_3\cdots$ and $w_2=r_1r_2r_3\cdots$ be two
infinite reduced words of $W$. Then $w_1\leq w_2$ if and only if any
prefix of $w_1$ is a prefix of $w_2$.

(b) Let $x,y\in \overline{W}$. Then $x\perp y$ if and only if for any
prefix $u$ of $x$ and any prefix $v$ of $y$ we have that
$l(u^{-1}v)=l(u)+l(v)$.

(c) The multiplications in Definition $\ref{def:sprod}$ and
$\ref{def:wprod}$ are all well-defined, i.e. in $\ref{def:sprod}$
case I the product $sw$ is independent of the choice of the reduced
expression of $w$, and in $\ref{def:sprod}$ case II the product $sw$
is independent of the choice of reduced expression of $w$, the
choice of $u$ and the choice of the reduced expression of $su$, and
in $\ref{def:wprod}$ the product $uw$ does not depend on the choice
of the expression of $u$.

(d) Let $x,y\in W$ and $z\in \overline{W}$. Then
$x\cdot\Phi_z=\Phi_{xz}$, $x\cdot\Phi_z'=\Phi_{xz}'$ and
$(xy)z=x(yz).$
\end{lemma}

\begin{proof}
(a) and (b) follow easily from the definition and Proposition 3.1.3 of \cite{bjornerbrenti}.
Now let $s_1s_2s_3\cdots$ and $r_1r_2r_3\cdots$ be two reduced expressions of an infinite reduced word $w$ and suppose that $s\perp w$.
It follows that $\Phi_{ss_1s_2s_3}=\bigcup_{i=1}^{\infty}\Phi_{ss_1s_2\cdots s_i}=\bigcup_{i=1}^{\infty}(\Phi_s\cup s\Phi_{s_1s_2\cdots s_i})=\Phi_s\cup s(\bigcup_{i=1}^{\infty}\Phi_{s_1s_2\cdots s_i})=\Phi_s\cup s(\bigcup_{i=1}^{\infty}\Phi_{r_1r_2\cdots r_i})=\bigcup_{i=1}^{\infty}\Phi_{sr_1r_2\cdots r_i}=\Phi_{sr_1r_2\cdots}$.
 Now instead assume that $s\not\perp w$. Let $u=s_1s_2\cdots s_k$ and $l(su)=k-1$. By the basic properties of the reduced expressions of an element in a Coxeter group, one has $l(sus_{k+1}s_{k+2}\cdots s_j)=j-1$ for any $j\geq k+1$. Therefore concatenating a reduced expression of $su$ (denoted by $\underline{su}$) with $s_{k+1}s_{k+2}\cdots$ gives rise to a valid infinite reduced word. It also follows that $\Phi_{\underline{su}s_{k+1}s_{k+2}\cdots}=\Phi_{su}\cup su\Phi_{s_{k+1}s_{k+2}\cdots}$. Therefore the product is independent of the choice of the reduced expression of $su$. Also note that $\Phi_{su}\cup su\Phi_{s_{k+1}s_{k+2}\cdots}=(s\Phi_u\backslash (-\Phi_s))\cup s(u\Phi_{s_{k+1}s_{k+2}\cdots})=s(\Phi_w)\backslash (-\Phi_s)$. This shows that the product is independent of the choice of the reduced expression of $w$ and the prefix $u$.
 Therefore  the multiplication in Definition $\ref{def:sprod}$ is well defined.

 Now let $s\in S, w\in \overline{W}$. One checks in two cases: $s\perp w$ and $s\not\perp w$ (by using the formula $x\cdot \Gamma:=(\Phi_x\backslash x(-\Gamma))\cup (x(\Gamma)\backslash
(-\Phi_x))$) that $s\cdot \Phi_w=\Phi_{sw}$. This proves that the multiplication in Definition \ref{def:wprod} is well defined. It also follows that $x\cdot \Phi_z=\Phi_{xz}$ and $x(yz)=(xy)z$ in (d). Finally $x\cdot \Phi_z'=\Phi_{xz}'$ follows from the general fact that $x\cdot (\Phi^+\backslash \Gamma)=\Phi^+\backslash (x\cdot \Gamma)$ for $\Gamma$ biclosed in $\Phi^+$.
\end{proof}

Now we study the weak order on $\overline{W}$. First we prove an
easy but useful lemma.

\begin{lemma}\label{lem:boundingrp}
Let $A\subset W$ and $|A|<\infty$. Suppose that $A$ is bounded in
$\overline{W}$. Then $A$ is bounded in $W$ and $\bigvee A$ exists in
$W$.
\end{lemma}

\begin{proof}
Suppose that $A$ is bounded by $s_1s_2s_3\cdots$. Then every element of
$A$ is bounded by some $s_1s_2\cdots s_k\in W$. Since $A$ is finite,
we can take the longest of these bounds and it is a bound in $W$ of
$A$. The second assertion follows from the fact that $W$ is a complete
meet semilattice and Lemma \ref{techsemijoin}.
\end{proof}

\begin{theorem}\label{semilattice}
The poset $(\overline{W},\leq)$ is a complete meet semilattice.
\end{theorem}

\begin{proof}
We first show

\emph{Claim}: Let $w$ be an infinite reduced word. Then the set of
its prefixes is countably infinite.

Let $s_1s_2s_3\cdots$ be a reduced expression of $w$. Then the set
of the prefixes of $w$ is $\bigcup_{i=1}^{\infty}\{u\in W|u\leq
s_1s_2\cdots s_i\}$ and each set in this union is finite. Therefore the claim is true.

We first show that any non-empty subset of $W_l$ has a meet. Take
$\emptyset\neq A\subset W_l$. As a consequence of the claim, the set
$B:=\{w\in W|w\leq a,\forall a\in A\}$ is countable. So we let
$B=\{b_1,b_2,\dots\}$. Then we obtain an ascending chain under weak
order:
$$b_1\leq b_1\vee b_2\leq b_1\vee b_2\vee b_3\leq \dots.$$
Each term in this chain exists in $W$ by Lemma \ref{lem:boundingrp}.
Therefore this chain gives rise to a well-defined infinite reduced
word or an element in $W$. Denote it by $v$. We shall show that $v$ is
the meet of $A$.

First we note that it is immediate that $v\geq b,\forall b\in B$.
Then take some $w\in \overline{W}$ which is a lower bound of $A$. So
any prefix of $w$ is in $B$ and therefore is bounded by $v$. This
shows that $w\leq v$. So $v$ is the meet of $A$.

Since $W$ is a complete meet semilattice, any subset of $W$ admits a
meet in $W$ and it is clear that such meet remains to be the meet in
$\overline{W}$.

Now we consider $A\subset \overline{W}$ with $A\cap W_l\neq
\emptyset$ and $A\cap W\neq \emptyset$. Let $a=\bigwedge
(A\cap W_l)$ and $b=\bigwedge(A\cap W)$. Fix the reduced expressions
$s_1s_2s_3\cdots$ (finite if $a\in W$) for $a$. Then we have a
chain:
$$b\wedge s_1\leq b\wedge s_1s_2\leq b\wedge s_1s_2s_3\leq \dots.$$
This chain defines an infinite reduced word or an element in the
group. Denote it by $y$. One readily sees that all lower bounds of $A$
are less than or equal to $y$.
\end{proof}

The poset $\overline{W}$ and its semilattice property was considered by various authors, for example in \cite{Lam1} Theorem 4.10. We incorporate this  proof for general Coxeter groups and emphasize that this meet semilattice is complete.

It is natural to ask when $\overline{W}$ is a complete lattice. This
question is answered in the following Theorem.

\begin{theorem}
The poset $\overline{W}$ is a complete lattice if and only if $W\simeq
{}^w\Pi_{i=1}^{\infty}W_i$ where each $W_i$ is either a finite or
locally finite irreducible Coxeter group and ${}^w\Pi$ denotes the
weak direct product.
\end{theorem}

\begin{proof}
We first prove the ``only if'' part. A Coxeter group $W$ can always be written as the
weak direct product of a family of irreducible Coxeter groups,
called its components. Since any two elements of $W$ have an upper
bound in $\overline{W}$, they have an upper bound in $W$ by Lemma
\ref{lem:boundingrp}. So each group of that family must either be
locally finite or finite. If the family consists of uncountably many
non-trivial Coxeter groups, then we take one generator from each of
these Coxeter groups. Clearly their join cannot possibly be an
infinite reduced word or an element in $W$.

To prove the converse, we first note that for a finite Coxeter group
$W$ there exists a longest element $w_0(W)$. For a locally finite
Coxeter group $W$, one can find a sequence of parabolic subgroups
$$W_1<W_2<W_3<\dots$$ such that each $W_i$ is a finite
and $\bigcup_{i=1}^{\infty} W_i=W$ by the classification  of
locally finite Coxeter groups. Thus we have a chain
$$w_0(W_1)\leq w_0(W_2)\leq w_0(W_3)\leq \dots.$$
This chain defines an infinite reduced word $w_0(W)$ which is the
unique maximal element in $(\overline{W},\leq)$. Now suppose that
$W\simeq {}^w\Pi_{i=1}^{\infty}W_i$ and that
$w_0(W_i)=s_{i,1}s_{i,2}s_{i,3}\cdots.$ Listing these elements
$$s_{1,1}s_{1,2}s_{1,3}\cdots$$
$$s_{2,1}s_{2,2}s_{2,3}\cdots$$
$$s_{3,1}s_{3,2}s_{3,3}\cdots$$
$$\cdots$$
and then  enumerating $s_{i,j}$'s in a zig-zag way (i.e.
$s_{1,1}s_{1,2}s_{2,1}s_{3,1}s_{2,2}s_{1,3}s_{1,4}s_{2,3}\cdots$)
give rise to a well-defined infinite reduced word which is maximal
under the weak order. The existence of such an element makes the
complete semilattice a complete lattice.
\end{proof}

\begin{remark*}
It can be shown that for a locally finite Coxeter group  any biclosed set in $\Phi^+$  is the inversion set of a (finite or infinite) word. Thus the above theorem proves that the poset of biclosed sets in $\Phi^+$ is a lattice for those types of Coxeter groups.
\end{remark*}

Next we show that $\overline{W}$ has the following favorable
property and we refer to it as Join Orthogonal Property.

\begin{theorem}\label{theorem:jop}
Suppose that $A\subset \overline{W}$ and $A$ admits a join $w\in
\overline{W}.$ Let $v\in \overline{W}$
such that $a\perp v$ for all $a\in A$. Then $w\perp v$.
\end{theorem}

\begin{proof}
Suppose that $w$ is in $W$ (then $A\subset W$). If $v\in W$, the
assertion follows from Lemma 4.2(b) in \cite{DyerWeakOrder}. If
$v\in W_l$, then every prefix of $v$ is orthogonal to $w$ by the
above argument and thus $v\perp w$. Suppose that $w$ is in $W_l$.
Consider $B=\{b\in W|\exists a\in A, b\leq a\}$. Clearly $\bigvee
B=\bigvee A=w$ and the set $B$ is contained in the set of prefixes of $w$
and hence countable by the claim in the proof of Theorem
\ref{semilattice}. Then write $B=\{b_1,b_2,\dots\}$. We also have that
$b_i\perp v$ for all $i$. Therefore any prefix of $w$ must be
smaller than $\bigvee_{i=1}^{n}b_i$ for some $n$ and thereby
orthogonal to $v$.
\end{proof}

The following proposition will be used in proving a characterization
of affine Weyl groups in the next section. Now suppose that $W$ is
countably generated (i.e. $|S|\leq \aleph_0$). As a consequence $W$
is countable.

\begin{proposition}\label{prop:maximalinwbar}
The poset $(\overline{W},\leq)$ contains maximal elements. Every element
$w\in \overline{W}$ is less than or equal to some maximal element.
\end{proposition}

\begin{proof}
To prove the first assertion it suffices to show that the condition
of Zorn's Lemma is satisfied. Suppose that we have an ascending chain in
$\overline{W}$: $\{w_i\}_{i\in I}$. Then the set $B=\{b\in W|b\leq
w_i$ for some $i\}$ is countable since $W$ is countable. Let
$B=\{b_1,b_2,b_3,\dots\}$. For $b_1,b_2,\dots,b_k$, one can find
some $w_j$ which is their upper bound. Hence $b_1,b_2,\dots,b_k$ admit a
join in $W$ by Lemma \ref{lem:boundingrp}. So we have an ascending
chain of elements in $W$:
$$b_1\leq b_1\vee b_2\leq b_1\vee b_2\vee b_2\leq \dots.$$
This defines an element in $\overline{W}$ (denoted by $v$) which will be shown to be
the join of the original chain $\{w_i\}_{i\in I}$. Take any $w_i$ in the chain. Since any finite prefix of $w_i$ is $b_l$ for some $l$, it is less than $v$. Therefore we have that $w_i\leq v$.
Let $u$ be any upper bound of the chain $\{w_i\}_{i\in I}$. Since $\bigvee_{i=1}^k b_i$ is bounded by some $w_j$, $\bigvee_{i=1}^k b_i\leq w_j\leq u$. Therefore $v\leq u$.
Hence we conclude that $v=u.$
The second assertion
follows by the same argument applied to $\{x\in \overline{W}|x\geq
w\}$.
\end{proof}

Another consequence of the  above proof is the following Lemma.

\begin{lemma}\label{chaininwbar}
Let $W$ be countably generated and $\{w_i\}_{i\in I}$ be an ascending chain in $\overline{W}$ under the weak order. Then the join of $\{w_i\}_{i\in I}$ exists in $\overline{W}$ and $\Phi_{\vee_i w_i}=\cup_{i\in I}\Phi_{w_i}$.
\end{lemma}

\begin{proof}
Such a join (denoted by $v$) is constructed in the proof of Proposition \ref{prop:maximalinwbar}. We only need to show that $\Phi_v=\cup_{i\in I}\Phi_{w_i}.$ Take a finite prefix $b$ of $w_i$ for some $i\in I$. By the construction of $v$, $\Phi_b\subset \Phi_{v}$. Therefore $\Phi_{w_i}\subset \Phi_v$ for any $i$. Hence $\Phi_v\supset\cup_{i\in I}\Phi_{w_i}.$ Conversely $v$ is the word defined by the ascending chain:
$$b_1\leq b_1\vee b_2\leq b_1\vee b_2\vee b_2\leq \dots$$
where $\{b_1,b_2,b_3,\dots\}$ is the set of finite prefixes of $\{w_i\}_{i\in I}$.
Hence the inversion set of each term in this ascending  chain is contained in $\cup_{i\in I}\Phi_{w_i}$. Therefore we see the reverse inclusion.
\end{proof}

\section{infinite reduced words And Biclosed Sets In an Affine Root
System}\label{sec:bijection}

The main aim of this section is to obtain a description of $\Phi_w$ for $w\in \overline{W}$ in the affine case.  Related descriptions have been studied by other authors, for example in \cite{Lam1} Section 4. We characterize those biclosed sets arising from infinite reduced words by using the description of affine root systems as introduced in Section \ref{sectpreli} and the action of $W$ on the set of biclosed sets. Our proof is different and  uses  an interesting property of the finite root systems (Proposition \ref{prop:twofunctionequal}) which has not been observed previously to the knowledge of the author and this property will be used again later in the paper (Theorem \ref{thm:classifycompact}).
A proof of the classification without using Proposition \ref{prop:twofunctionequal} is also possible and is given in a subsequent paper \cite{orderpaper}.

Let $\Phi$ be an irreducible crystallographic root system  of a Weyl
group $W$. Let $\Pi$ be a simple system of $\Phi$ and
$L\varsubsetneqq \Pi.$ Denote the positive system corresponding
to $\Pi$ by $\Phi^+$.

We begin by defining the function
$$h_L: \Phi^+\rightarrow \mathbb{N}.$$
Suppose that $\Phi^+\ni\alpha=\sum_{\alpha_i\in \Pi}k_i\alpha_i$. Let
$h_L(\alpha):=\sum_{\alpha_i\not\in L}k_i$. (Note that each $k_i$ is an
integer and that $h_L$ is clearly additive.)

We define another function
$$d_L: \Phi^+\rightarrow \mathbb{N}.$$
If $\alpha\in \Phi^+_{L,\emptyset}$,
$$d_L(\alpha)=\max\{n|\alpha \:\text{can be written as the sum of}\: n \:\text{elements
in}\: \Phi^+_{L,\emptyset}\}.$$ If $\alpha\not\in
\Phi^+_{L,\emptyset}$ (that is $\alpha\in \Phi_L$), then $d_L(\alpha)=0.$

When $L$ is fixed (as we do in the following discussion), we denote the functions $d_L$ and $h_L$ by $d$ and $h$ respectively for convenience.

In order to establish the relation between these two functions, we
need the following lemma which is suggested and proved by Matthew
Dyer. (This has been generalized to sets of $n$ roots in \cite{rootsum} Theorem 3.7.)

\begin{lemma}\label{threerootsum}
Let $\alpha,\beta,\gamma\in \Phi^+$. Suppose that $\alpha+\beta$ and
$\alpha+\beta+\gamma$ are  both in $\Phi^+$. Then either
$\alpha+\gamma$ or $\beta+\gamma$ is in $\Phi^+$.
\end{lemma}

\begin{proof}
Assume that to the contrary, neither $\alpha+\gamma$ nor $\beta+\gamma$
is a (positive) root. Since the root system is crystallographic, we have that $(\alpha,\gamma)\geq 0$ and that $(\beta,\gamma)\geq
0.$ Also since neither $(\alpha+\beta+\gamma)+(-\beta)$ nor
$(\alpha+\beta+\gamma)+(-\alpha)$ is a (positive) root, we
have that $(\alpha+\beta+\gamma,-\beta)\geq 0$ and that
$(\alpha+\beta+\gamma,-\alpha)\geq 0.$ Let
$v^{\vee}=\frac{2v}{(v,v)}$. Then we conclude that
$$(\alpha^{\vee},\gamma)\geq 0,$$
$$(\beta^{\vee},\gamma)\geq 0,$$
$$(\alpha+\beta+\gamma,-\alpha^{\vee})\geq 0,$$
$$(\alpha+\beta+\gamma,-\beta^{\vee})\geq 0.$$
This gives that
$$(\beta,\alpha^{\vee})\leq -2,$$
$$(\alpha,\beta^{\vee})\leq -2.$$
This yields that
$$(\alpha,\beta)^2\geq (\alpha,\alpha)(\beta,\beta).$$
But by Cauchy-Schwarz inequality this forces that $\alpha=\beta$, which
again contradicts the assumption that $\alpha+\beta$ is a root.
\end{proof}

\begin{proposition}\label{prop:twofunctionequal}
Let functions $d,h$ be defined as above. Then for any $\alpha\in \Phi^+$, $d(\alpha)=h(\alpha)$.
\end{proposition}

\begin{proof} (M. Dyer)
If $\alpha\not\in \Phi^+_{L,\emptyset},$ then the both sides are
$0$. So we assume that $\alpha\in \Phi^+_{L,\emptyset}$. (This implies that $\alpha\not\in \Phi_L$.)

Suppose that $\alpha$ can be written as a sum of $n$ elements in
$\Phi^+_{L,\emptyset}.$ When writing such an element as the positive
linear combination of simple roots, there must be at least one
simple root outside $L$ appearing in the sum. Hence we see that
$d(\alpha)\leq h(\alpha)$.

Now we prove that $h(\alpha)\leq d(\alpha)$. For simple roots we
clearly have that $h(\alpha)=d(\alpha)$. Also note that any positive root
$\alpha$ can be written as $\alpha_1+\alpha_2+\cdots+\alpha_k$ where
each $\alpha_i$ is simple and
$\alpha_1+\alpha_2+\cdots+\alpha_i$ is a root. So it suffices to
show that $d(\alpha+\beta)\geq d(\alpha)+d(\beta)$ for
$\alpha,\beta,\alpha+\beta$ all positive roots.

If $d(\alpha)=d(\beta)=0$, the inequality is trivial. If
$d(\alpha)>0$ and $d(\beta)>0$, then the inequality follows from the
definition of the function $d$. So it suffices to prove the case where
$d(\alpha)=0$ and $d(\beta)>0.$ We prove this by  induction on
$d(\beta)$. The case where $d(\beta)=1$ is trivial. Suppose that for
$d(\beta)\leq m$, $d(\alpha)+d(\beta)\leq d(\alpha+\beta)$.

Assume that $d(\beta)=m+1$ and $\beta=\beta_1+\beta_2+\cdots+\beta_{m+1}$
with each $\beta_i\not\in \Phi_L$. We can even assume that  $\beta_1+\beta_2+\cdots+\beta_i$ is a root for all
$i$ by Chapter VI \S 1
Proposition 19 in \cite{Bourbaki}. By Lemma \ref{threerootsum},
either $\alpha+\beta_{m+1}$ or
$\alpha+\beta_1+\beta_2+\cdots+\beta_m$ is a (positive) root. In the
first case
$$\alpha+\beta=(\alpha+\beta_{m+1})+(\beta_1)+(\beta_2)+\cdots+(\beta_{m}).$$
So we have that $d(\alpha+\beta)\geq m+1$. In the second case,
$$\alpha+\beta=(\alpha+\beta_1+\beta_2+\cdots+\beta_m)+\beta_{m+1}.$$
We have that $d(\alpha+\beta_1+\beta_2+\cdots+\beta_m)>0$ and so by the
previous discussion $d(\alpha+\beta)\geq
d(\alpha+\beta_1+\beta_2+\cdots+\beta_m)+1$. But by induction
$d(\alpha+\beta_1+\beta_2+\cdots+\beta_m)\geq m$ and we are done.
\end{proof}

\begin{corollary}\label{cor:additive}
If $\alpha=k_1\beta+k_2\gamma$ with $\alpha, \beta, \gamma\in
\Phi^+$ and $k_1, k_2\in \mathbb{R}_{\geq 0}$, then
$d(\alpha)=k_1d(\beta)+k_2d(\gamma)$.
\end{corollary}

\begin{proof}
By Proposition \ref{prop:twofunctionequal} it suffices to show that
$h(\alpha)=k_1h(\beta)+k_2h(\gamma)$. But this is clear from the
definition of the $h$ function.
\end{proof}

From now on we consider the associated affine root system
$\widetilde{\Phi}$ with the positive system $\widehat{\Phi}$. The
corresponding (irreducible) affine Weyl group is denoted by $\widetilde{W}$.

\begin{definition}
Denote  the set $$\{\alpha+k\delta|\alpha\in
\Phi^+_{L,\emptyset},0\leq k\leq n\}$$ by $(\Phi^+_{L,\emptyset})_n$.
\end{definition}

\begin{lemma}
The identity
$$\overline{(\Phi^+_{L,\emptyset})_n}=\{\alpha+k\delta|\alpha\in \Phi^+_{L,\emptyset},0\leq k\leq nd(\alpha)\}$$
holds.
\end{lemma}

\begin{proof}
Let $0\leq k\leq nd(\alpha)$. Then we can find $0\leq
a_1,a_2,\dots,a_{d(\alpha)}\leq n$ such that
$$a_1+a_2+\cdots+a_{d(\alpha)}=k.$$
By definition of the function $d$, we have that
$\alpha=\alpha_1+\alpha_2+\cdots+\alpha_{d(\alpha)}$ where
$\alpha_i\in \Phi^+_{L,\emptyset}$. Therefore
$$\alpha+k\delta=(\alpha_1+a_1\delta)+(\alpha_2+a_2\delta)+\cdots+(\alpha_{d(\alpha)}+a_{d(\alpha)}\delta).$$
This combined with Chapter VI \S 1 Proposition 19 in \cite{Bourbaki}
shows that $\alpha+k\delta\in \overline{(\Phi^+_{L,\emptyset})_n}.$

Conversely we just need to prove that the set on the right is closed, which is enough
since it lies between the set and its closure. Assume that $\alpha_1,\alpha_2\in \Phi^+_{L,\emptyset}$,
$0\leq k_1\leq nd(\alpha_1), 0\leq k_2\leq nd(\alpha_2)$,
$m_1,m_2\geq 0$ ($k_1, k_2\in \mathbb{Z}$, $m_1, m_2\in
\mathbb{R}$) and $m_1(\alpha_1+k_1\delta)+m_2(\alpha_2+k_2\delta)$
is a root in $\widehat{\Phi}$. Then $m_1\alpha_1+m_2\alpha_2\in
\Phi.$ Since $h(\alpha_i)>0$ for $i=1,2$ and $m_1, m_2$ cannot both
be zero so we have that $h(m_1\alpha_1+m_2\alpha_2)>0$ and therefore
$m_1\alpha_1+m_2\alpha_2\in \Phi^+_{L,\emptyset}$. By Corollary
\ref{cor:additive},
$d(m_1\alpha_1+m_2\alpha_2)=m_1d(\alpha_1)+m_2d(\alpha_2)\geq
m_1\frac{k_1}{n}+m_2\frac{k_2}{n}$. So
$$nd(m_1\alpha_1+m_2\alpha_2)\geq m_1k_1+m_2k_2.$$
This implies that $m_1\alpha_1+m_2\alpha_2+(m_1k_1+m_2k_2)\delta$ is still
in $\{\alpha+k\delta|\alpha\in \Phi^+_{L,\emptyset},0\leq k\leq
nd(\alpha)\}$.
\end{proof}

An immediate consequence of this lemma is

\begin{lemma}\label{lem:finite}
The set $\overline{(\Phi^+_{L,\emptyset})_n}$ is  finite.
\end{lemma}

We also have

\begin{lemma}\label{lem:biclosed}
The set $\overline{(\Phi^+_{L,\emptyset})_n}$ is  biclosed.
\end{lemma}

\begin{proof}
We only need to show that the set is coclosed.

Case I:

Consider $\alpha+k_1\delta, \beta+k_2\delta$ where $\alpha, \beta\in
\Phi^+_{L,\emptyset}$ and $k_1>nd(\alpha), k_2>nd(\beta).$  Suppose that
$t_1\alpha+t_2\beta+(t_1k_1+t_2k_2)\delta, t_1, t_2\geq 0$ is a root
of $\widehat{\Phi}$. So
$d(t_1\alpha+t_2\beta)=t_1d(\alpha)+t_2d(\beta)<t_1\frac{k_1}{n}+t_2\frac{k_2}{n}$.
Therefore
$$t_1k_1+t_2k_2>nd(t_1\alpha+t_2\beta).$$
That implies that $t_1\alpha+t_2\beta+(t_1k_1+t_2k_2)\delta\not\in
\overline{(\Phi^+_{L,\emptyset})_n}.$

Case II:

Consider $\alpha+k_1\delta, \beta+k_2\delta$ where $\alpha,
\beta\not\in \Phi^+_{L,\emptyset}$. Suppose that
$t_1\alpha+t_2\beta+(t_1k_1+t_2k_2)\delta, t_1, t_2\geq 0$ is a root
of $\widehat{\Phi}$. Then $t_1\alpha+t_2\beta\not\in
\Phi^+_{L,\emptyset}$ as $\Phi^+_{L,\emptyset}$ is a biclosed set in
$\Phi.$

Case III:

Consider $\alpha+k_1\delta, \beta+k_2\delta$ where $\alpha\not\in
\Phi^+_{L,\emptyset}$ and $\beta\in \Phi^+_{L,\emptyset}$ and
$k_2>nd(\beta)$. Suppose that $t_1\alpha+t_2\beta+(t_1k_1+t_2k_2)\delta,
t_1, t_2\geq 0$ is a root of $\widehat{\Phi}$. We only need to
examine the case where $t_1\alpha+t_2\beta\in \Phi^+_{L,\emptyset}$.
In this case $d(t_1\alpha+t_2\beta)\leq
t_2d(\beta)<t_2\frac{k_2}{n}$(If $\alpha$ is positive then the first
$\leq$ is $=$). Therefore
$$t_1k_1+t_2k_2\geq t_2k_2>nd(t_1\alpha+t_2\beta).$$
That implies that $t_1\alpha+t_2\beta+(t_1k_1+t_2k_2)\delta\not\in
\overline{(\Phi^+_{L,\emptyset})_n}.$
\end{proof}

\begin{corollary}\label{cor:infinitelongform}
The biclosed set $\widehat{\Phi^+_{L,\emptyset}}$ is of the form $\Phi_w$ where $w$
is an infinite reduced word (in $\overline{\widetilde{W}}$).
\end{corollary}

\begin{proof}
This follows from the above Lemma \ref{lem:finite} and Lemma
\ref{lem:biclosed} and
$$\widehat{\Phi^+_{L,\emptyset}}=\bigcup_{n=0}^{\infty}\overline{(\Phi^+_{L,\emptyset})_n}.$$
\end{proof}

\begin{lemma}\label{lem:biclosedsetthatisnotinfinielongword}
Let $\Psi^+$ be a positive system of $\Phi$, $\Delta$ be its simple
system, $N,M\subset \Delta$ with $N$ orthogonal to $M$ and $M\neq\emptyset$. Then
$\widehat{\Psi_{N,M}^+}$ is not of the form $\Phi_w$ where $w$ is an
infinite reduced word.
\end{lemma}

\begin{proof}
Let $\alpha\in M.$ Then by definition $\alpha+p\delta$ and $-\alpha+q\delta$
are both in $\widehat{\Psi_{N,M}^+}$ for all $p,q\in
\mathbb{Z}_{\geq 1}$.
Suppose that to the contrary $\widehat{\Psi_{N,M}^+}=\Phi_w$ where $w$ is an infinite reduced
word. Then they must both be contained in some finite biclosed set $\Phi_u, u\in
\widetilde{W}$. Since $\Phi_u$ is biclosed so we have that
$2(\alpha+p\delta)+(-\alpha+q\delta)=\alpha+(2p+q)\delta\in \Phi_u, 3(\alpha+p\delta)+2(-\alpha+q\delta)=\alpha+(3p+2q)\delta, \dots$.

This implies that $\Phi_u$ is infinite and this is a contradiction.
\end{proof}

The proof of the above lemma also implies the following useful result that will be repeatedly used later in the paper.

\begin{lemma}\label{closedaffine}
Let $C$ be a closed set in $\widehat{\Phi}$ and $\alpha\in \Phi$. If there exist $\alpha+n\delta, -\alpha+m\delta, m,n\in \mathbb{Z}_{\geq 1}$ both in $C$, then $C$ is infinite.
\end{lemma}

\begin{lemma}\label{tworepresentation}
One has the following equality of sets:

$\{v\cdot \widehat{\Psi^+_{M,N}}|v\in \widetilde{W}, \Psi^+\,\text{is a positive system of }\,\Phi, M,N\, \text{are orthogonal subsets}\newline \text{of the simple system of }\Psi^+\}$
$=\{w\cdot \widehat{\Phi^+_{L,K}}|w\in \widetilde{W}, L, K\subset \Pi, L\,\text{and}\, K \,\text{are othogonal}\}$.
Furthermore, if $v\cdot \widehat{\Psi^+_{M,N}}=w\cdot \widehat{\Phi^+_{L,K}}$, then $|M|=|L|$ and $|N|=|K|.$
\end{lemma}

\begin{proof}
 It is clear that we only need to show that the left hand side is contained in the right hand side.  Furthermore it suffices to show that $\widehat{\Psi^+_{M,N}}$ is contained in the right hand side for  $\Psi^+$  a positive system of  $\Phi$ and $M,N$  two orthogonal subsets of the simple system of $\Psi^+$.

There exists $z\in W$ such that $z\Psi^+=\Phi^+.$ Then $z\Psi^+_{M,N}=\Phi^+_{z(M),z(N)}$ is a biclosed set in $\Phi$. Now by
Theorem \ref{reflectionorderaffine} (2) take $u\in
\widetilde{W}$ such that $\pi(u)=z^{-1}$ (where $\pi$ is the canonical surjection from $\widetilde{W}$ to $W$), we have that
$\widehat{\Psi^+_{M,N}}$ and $u\cdot\widehat{\Phi^+_{z(M),z(N)}}
$ differ only by finitely many roots. Now by Theorem \ref{classifybiclosedsetaffine} (2), there exists $u'\in W'$ where $W'$ is the reflection subgroup generated by $\{s_{\alpha}|\alpha\in \widehat{\Psi_{M\cup N}}\}$ such that $u'\cdot (u\cdot\widehat{\Phi^+_{z(M),z(N)}})=\widehat{\Psi^+_{M,N}}$. So the equality of the two sets is proved.

For the second assertion, one notes that if $v\cdot \widehat{\Psi^+_{M,N}}=w\cdot \widehat{\Phi^+_{L,K}}$ one has that $\pi(w^{-1}v)(\Psi^+)=\Phi^+, \pi(w^{-1}v)(M)=L$ and $\pi(w^{-1}v)(N)=K.$ Therefore  $|M|=|L|$ and $|N|=|K|.$
\end{proof}

\begin{theorem}\label{thm:bijection}
The map
$$u\mapsto \Phi_u$$
gives

(1) a bijection between
$\widetilde{W}_l$ and the following subset of $\mathscr{B}(\widehat{\Phi})$
$$\{w\cdot \widehat{\Phi^+_{L,\emptyset}}|w\in \widetilde{W}, L\subsetneq\Pi\},$$

(2) a bijection between $\overline{\widetilde{W}}$ and the following subset of $\mathscr{B}(\widehat{\Phi})$
$$\{w\cdot \widehat{\Phi^+_{L,\emptyset}}|w\in \widetilde{W}, L\subset \Pi\}.$$
\end{theorem}

\begin{proof}
Suppose that $u\in \overline{\widetilde{W}}$.
In view of Theorem \ref{theorem:classifiybiclosedweyl},
\ref{classifybiclosedsetaffine} and Lemma \ref{tworepresentation}, we need to examine whether a set
$\widehat{\Phi^+_{M,N}},M,N\subset \Pi$ is in the image of the given
map. Corollary \ref{cor:infinitelongform} shows that
$\widehat{\Phi_{L,\emptyset}^+}$ is in the image of the map
$u\mapsto \Phi_u$. By Lemma \ref{lem:infinitelongbasic} (d) $w\cdot \widehat{\Phi_{L,\emptyset}^+}$ is
also in the image. Lemma
\ref{lem:biclosedsetthatisnotinfinielongword} shows that
$\widehat{\Phi_{N,M}^+}$ where $M\neq \emptyset$ is not in the
image. Then neither is $w\cdot \widehat{\Phi_{N,M}^+}$ by Lemma
\ref{lem:infinitelongbasic} (d). Finally we see that $w\cdot
\widehat{\Phi_{\Pi,\emptyset}^+}$ is finite and thus not in the
image if $u\in \widetilde{W}_l$.
\end{proof}

\begin{remark*}
It is worth noting that the case $L = \Pi$ gives the finite biclosed sets.
\end{remark*}

\begin{corollary}\label{cor:biclosedsetofinflongword}
Let $\Psi^+$ be a positive system in $\Phi$ and $\Delta$ be the
simple system of it. Suppose that $M\subset \Delta$ and that $v\in
\widetilde{W}$. Then $v\cdot \widehat{\Psi^+_{M,\emptyset}}$ is
equal to $\Phi_x$ for some $x\in \overline{\widetilde{W}}.$
\end{corollary}

\begin{proof}
The assertion follows immediately from Lemma \ref{tworepresentation}, Theorem \ref{thm:bijection} and
Lemma \ref{lem:infinitelongbasic} (d).
\end{proof}

\begin{example*}
Let $W$ be of type $A_2$. Denote by $\alpha,\beta$ the two simple roots of $\Phi^+$, the standard positive system of $\Phi$. The biclosed set $\widehat{\Phi^+_{\{\beta\},\emptyset}}=\widehat{\{\alpha,\alpha+\beta\}}$ is an inversion set of an infinite reduced word by the above Theorem. In fact it is equal to $\Phi_{s_{\alpha}s_{\beta}s_{\delta-\alpha-\beta}s_{\alpha}s_{\beta}s_{\delta-\alpha-\beta}s_{\alpha}s_{\beta}s_{\delta-\alpha-\beta}\cdots}$
\end{example*}

In the subsequent work \cite{orderpaper}, for arbitrary infinite Coxeter groups a conjectural characterization of biclosed sets from infinite words  by the semilattice property of the associated ``twisted weak order'' is given.

As a consequence of the bijection discussed above, we provide a
characterization of affine Weyl groups. Let $B$ be a subset of
$\widehat{\Phi}$. Define $I_B=\{\alpha\in
\Phi||\widehat{\alpha}\bigcap B|=\aleph_0\}$. Define
$A_B=\{\alpha\in \Phi|\widehat{\alpha}\bigcap B\neq \emptyset\}$.

\begin{lemma}\label{lem:zclosedlemma}
Let $\Gamma\subset \Phi$ be $\mathbb{Z}-$closed. If $\Gamma\cap
-\Gamma=\emptyset$ then $\Gamma$ is contained in a positive system
of $\Phi$.
\end{lemma}

\begin{proof}
This is \cite{Bourbaki} Chapter IV, $\S 1$, Proposition 22.
\end{proof}

\begin{lemma}\label{lem:asetzclosed}
Let $\Gamma$ be a closed subset of $\widetilde{\Phi}^+$. Then
$A_{\Gamma}$ is $\mathbb{Z}-$closed in $\Phi$.
\end{lemma}

\begin{proof}
Let $\alpha,\beta\in A_{\Gamma}$ such that $\alpha+\beta$ is a root. This
implies that $\alpha+m\delta$ and $\beta+n\delta$ are in $\Gamma$ for some
$m,n\in \mathbb{Z}_{\geq 0}$. Then $m+n\in \mathbb{Z}_{\geq 0}$. If
$\alpha+\beta\in \Phi^-$, then at least one  root among $\alpha$ and $\beta$
is in $\Phi^-$, in which case $m+n\in \mathbb{Z}_{>0}$. So
$\alpha+\beta+(m+n)\delta\in\widetilde{\Phi}^+$, and therefore in
$\Gamma$.
\end{proof}

\begin{theorem}\label{maximalthm}
Let $(W,S)$ be a finite rank, irreducible, infinite Coxeter system.
Then $W$ is affine if and only if $\overline{W}$ admits finitely
many maximal elements.
\end{theorem}

\begin{proof}
The existence of such maximal elements is guaranteed by Proposition
\ref{prop:maximalinwbar}. For only if, we show that the maximal
elements in $\overline{W}$ where $W$ is an affine Weyl group correspond to
$\widehat{\Psi^+}$ where $\Psi^+$ is a positive system of $\Phi$
(and therefore there are only finitely many of them). First from
Corollary \ref{cor:biclosedsetofinflongword} every
$\widehat{\Psi^+}$ is indeed the inversion set of an infinite reduced
word.  Lemma
\ref{closedaffine} shows that any biclosed set
properly containing $\widehat{\Psi^+}$ is not the inversion set of an element
in $\overline{W}$ (such biclosed sets must contain some $\alpha+m\delta$ and $-\alpha+n\delta$ for some $\alpha\in \Phi$). Now it suffices to show that any other element
$u$ in $\overline{W}$ is not maximal. By Lemma
\ref{lem:asetzclosed}, $A_{\Phi_u}$ is $\mathbb{Z}$-closed. We claim that $A_{\Phi_u}\cap -A_{\Phi_u}=\emptyset$. To see the claim suppose that $\alpha\in A_{\Phi_u}\cap -A_{\Phi_u}$. So some $\alpha+m\delta$ and $-\alpha+n\delta$ are both in $\Phi_u$. Then  Lemma
\ref{closedaffine} shows that this is not possible. Hence by Lemma \ref{lem:zclosedlemma},
$A_{\Phi_u}\subset \Psi^+$ where $\Psi^+$ is a positive system in
$\Phi$ and $\Phi_u\subset \widehat{\Psi^+}$.

Conversely, by \cite{UniversalRefl}
Theorem 2.7.2, for $W$ finite rank,
irreducible, infinite and non-affine, there exists a universal reflection subgroup
of arbitrarily large rank. Suppose that $\overline{W}$ has $k$ maximal
elements. Choose a rank $k+1$ universal reflection subgroup
generated by $\{t_1,t_2,\dots,t_{k+1}\}$ (as canonical generators).
These reflections correspond to a set of positive roots
$\{\alpha_1,\alpha_2,\dots,\alpha_{k+1}\}$. There must be a maximal
element $w$ such that $\Phi_w$ contains $\alpha_i$ and $\alpha_j,
i\neq j$ by pigeonhole. But $\{k_1\alpha_i+k_2\alpha_j|k_1,k_2\in \mathbb{R}_{\geq 0}\}\cap \Phi$ is infinite and therefore they
cannot be contained in any finite biclosed set. This is a
contradiction.
\end{proof}

\begin{example*}
Let $W$ be of type $A_2$. Then $\overline{\widetilde{W}}$ has 6 maximal elements. They are in bijection with the 6 positive systems of $\Phi$. Denote by $\alpha,\beta$ the two simple roots of $\Phi^+$.
$$\widehat{\{\alpha,\beta,\alpha+\beta\}}=s_{\alpha}s_{\beta}s_{\alpha}s_{\delta-\alpha-\beta}s_{\alpha}s_{\beta}s_{\alpha}s_{\delta-\alpha-\beta}s_{\alpha}s_{\beta}s_{\alpha}s_{\delta-\alpha-\beta}\cdots,$$
$$\widehat{\{-\alpha,-\beta,-\alpha-\beta\}}=s_{\delta-\alpha-\beta}s_{\alpha}s_{\beta}s_{\alpha}s_{\delta-\alpha-\beta}s_{\alpha}s_{\beta}s_{\alpha}s_{\delta-\alpha-\beta}s_{\alpha}s_{\beta}s_{\alpha}\cdots,$$
$$\widehat{\{-\alpha,\beta,-\alpha-\beta\}}=s_{\beta}s_{\delta-\alpha-\beta}s_{\beta}s_{\alpha}s_{\beta}s_{\delta-\alpha-\beta}s_{\beta}s_{\alpha}s_{\beta}s_{\delta-\alpha-\beta}s_{\beta}s_{\alpha}\cdots,$$
$$\widehat{\{\alpha,-\beta,\alpha+\beta\}}=s_{\alpha}s_{\beta}s_{\delta-\alpha-\beta}s_{\beta}s_{\alpha}s_{\beta}s_{\delta-\alpha-\beta}s_{\beta}s_{\alpha}s_{\beta}s_{\delta-\alpha-\beta}s_{\beta}\cdots,$$
$$\widehat{\{-\alpha,\beta,\alpha+\beta\}}=s_{\beta}s_{\alpha}s_{\delta-\alpha-\beta}s_{\alpha}s_{\beta}s_{\alpha}s_{\delta-\alpha-\beta}s_{\alpha}s_{\beta}s_{\alpha}s_{\delta-\alpha-\beta}s_{\alpha}\cdots,$$
$$\widehat{\{\alpha,-\beta,-\alpha-\beta\}}=s_{\alpha}s_{\delta-\alpha-\beta}s_{\alpha}s_{\beta}s_{\alpha}s_{\delta-\alpha-\beta}s_{\alpha}s_{\beta}s_{\alpha}s_{\delta-\alpha-\beta}s_{\alpha}s_{\beta}\cdots.$$
\end{example*}

\section{A Lattice of Biclosed Sets in Rank 3 Affine Root System}\label{sec:lattice}
Let $W$ be a Weyl group and $\Phi$ be its irreducible crystallographic root
system. Let $\Phi^+$ be a positive system of $\Phi$ with $\Pi$ the
corresponding simple system. Denote by $\widetilde{W}$  the corresponding (irreducible)
affine Weyl group. Then $\widetilde{\Phi}:=\widehat{\Phi}\bigcup
-\widehat{\Phi}$ is the root system of $\widetilde{W}$ and
$\widetilde{\Phi}^+=\widehat{\Phi}$.  For $\alpha\in \Phi^+$, define
$\alpha_0=\alpha$. For $\alpha\in \Phi^-$, define
$\alpha_0=\alpha+\delta.$ The following lemma in fact reveals the so called dominance order in an affine root system.

\begin{lemma}\label{lem:startfrombottom}
Let $x\in \widetilde{W}$. If $\alpha\in \Phi$ and
$\alpha_0+k\delta\in \Phi_x$ with $k\geq 0$, then
$\alpha_0,\alpha_0+\delta,\dots, \alpha_0+k\delta$ are all contained in
$\Phi_x$.
\end{lemma}

\begin{proof}
 Suppose that
$x^{-1}(\alpha_0)=\beta+m\delta$ for some $\beta\in \Phi$. One has that
$x^{-1}(\alpha_0+k\delta)=\beta+(k+m)\delta\in \widetilde{\Phi}^-$. If
$\beta\in \Phi^+$ then $k+m\leq -1$ i.e. $m\leq -k-1.$ Then for
$0\leq p\leq k$, $x^{-1}(\alpha_0+p\delta)=\beta+(m+p)\delta$.
Note that $m+p\leq m+k\leq -1$. So $x^{-1}(\alpha_0+p\delta)\in
\widetilde{\Phi}^-.$ If $\beta\in \Phi^-$ then $k+m\leq 0$ i.e.
$m\leq -k.$ Then for $0\leq p\leq k$,
$x^{-1}(\alpha_0+p\delta)=\beta+(m+p)\delta$. Note that $m+p\leq m+k\leq 0$. So
$x^{-1}(\alpha_0+p\delta)\in \widetilde{\Phi}^-.$
\end{proof}

The following theorem answers a question from \cite{DyerWeakOrder} in
the case of irreducible affine Weyl groups (and easily generalized to affine
Coxeter groups). See the remark after Theorem 1.5 of
\cite{DyerWeakOrder}.

\begin{theorem}
Let $X\subset \widetilde{W}$. If $\overline{\bigcup_{x\in X}\Phi_x}$
is finite, then it is  biclosed and the join of $X$ exists.
\end{theorem}

\begin{proof}
We note that it could not happen that $\alpha+k\delta$ and
$-\alpha+m\delta$ are both in the finite set $\overline{\bigcup_{x\in X}\Phi_x}$
thanks to Lemma \ref{closedaffine}. Therefore by Lemma
\ref{lem:asetzclosed} all $\alpha\in \Phi$ with $\alpha+p\delta\in
\overline{\bigcup_{x\in X}\Phi_x}$ for some $p$  form a $\mathbb{Z}-$closed set
and this set has empty intersection with its negative. So by Lemma
\ref{lem:zclosedlemma} they belong to some positive system $\Xi^+$ of
$\Phi$. So we have that $\overline{\bigcup_{x\in
X}\Phi_x}\subset \widehat{\Xi^+}$. Since
$\overline{\bigcup_{x\in X}\Phi_x}$ is finite, it is contained in $\Phi_u$ where  $u$ is a finite
prefix of the infinite reduced word corresponding to
$\widehat{\Xi^+}$. So $X$ has a bound in $\widetilde{W}$ and the
assertion follows from Theorem 1.5(a) of \cite{DyerWeakOrder}.
\end{proof}

\begin{lemma}\label{lem:closedsetktoinfty}
Let $\Gamma$ be a closed subset of $\widetilde{\Phi}^+$ and
$\alpha\in \Phi.$ If $\alpha$ is in $I_{\Gamma}$ and $\alpha+N\delta\in
\Gamma$ then $\alpha+n\delta$ is in $\Gamma$ for all $n\geq N$.
\end{lemma}

\begin{proof}
There must be some $M>N$ such that $\alpha+M\delta\in \Gamma$ since
$\alpha\in I_{\Gamma}$. Then
$$\frac{1}{M-N}(\alpha+M\delta)+\frac{M-N-1}{M-N}(\alpha+N\delta)=\alpha+(N+1)\delta.$$
\end{proof}

\begin{lemma}\label{lem:upisnotinbottom}
Let $x,y\in \overline{\widetilde{W}}$. Then it cannot happen that
$\Phi_x'\subsetneq \Phi_y.$
\end{lemma}

\begin{proof}
By Lemma \ref{lem:biclosedsetthatisnotinfinielongword}, Corollary
\ref{cor:biclosedsetofinflongword} and Lemma
\ref{lem:infinitelongbasic}(d) the biclosed sets $w\cdot
\widehat{\Psi^+_{\emptyset,M}}$ where $w\in \widetilde{W}$ are
precisely those $\Phi_x',x\in \overline{\widetilde{W}}$. Suppose that
$M\neq \emptyset$. Then by Theorem \ref{classifybiclosedsetaffine} (2) and Theorem \ref{reflectionorderaffine} (2), those biclosed sets of the form $w\cdot
\widehat{\Psi^+_{\emptyset,M}}$ must contain some
infinite $\delta$ chain through some $\alpha$ and $-\alpha$ (where
$\alpha\in \Phi$) as $w\cdot \widehat{\Psi^+_{\emptyset,M}}$ differs from $\widehat{\pi(w)(\Psi^+_{\emptyset,M})}$ by finitely many roots. This implies that these sets cannot be possibly
contained in a biclosed set coming from $\Phi_y,y\in
\overline{\widetilde{W}}$ by the proof of
Lemma \ref{lem:biclosedsetthatisnotinfinielongword}. Now suppose that
$M=\emptyset.$ Then $w\cdot
\widehat{\Psi^+_{\emptyset,\emptyset}}=\widehat{\Xi^+}$ for some
positive system $\Xi^+$ in $\Phi$.  Clearly if there is a biclosed
set properly containing $\widehat{\Xi^+}$, again it must have infinite
$\delta$-chains through some $\alpha,-\alpha\in \Phi$, which implies that
it is not of the form $\Phi_y,y\in \overline{\widetilde{W}}$.
\end{proof}

\begin{remark*}
Suppose that $\Phi_x'=\Phi_y$. Again $\Phi_x'=w\cdot \widehat{\Psi^+_{\emptyset, N}}$ for some $w\in \widetilde{W}$.  The same reasoning as in the proof of the above Lemma
implies that $N=\emptyset$. Therefore $\Phi_x'=\widehat{\Xi^+}$ for some positive system $\Xi^+=\pi(w)\Psi^+$ in $\Phi$. Therefore we conclude that if $\Phi_x'=\Phi_y$ then $x,y$ are both maximal in $\overline{\widetilde{W}}$ (by the first paragraph of the proof of Theorem \ref{maximalthm}).
Conversely if $x$ is maximal in $\overline{\widetilde{W}}$, then $\Phi_x=\widehat{\Xi^+}$ for some positive system $\Xi^+$ and consequently $\Phi_x'=\widehat{\Xi^-}$. Take the infinite reduced word $y$ such that $\Phi_y=\widehat{\Xi^-}$. Then we have that
$\Phi_x'=\Phi_y$.
\end{remark*}

From now on, let $\widetilde{W}$ be of type $\widetilde{A}_2$ or
$\widetilde{B}_2=\widetilde{C}_2$ or $\widetilde{G}_2$.

\begin{lemma}\label{lem:orthocaseone}
Suppose that $x,y\in \widetilde{W}$ and that $x\wedge y=e$. Let
$\Gamma=\overline{\Phi_x\cup \Phi_y}$ and assume that
$I_{\Gamma}=A_{\Gamma}=\Psi^+_{\emptyset,\{\alpha\}}$ for some
positive system $\Psi^+$ in $\Phi$ and a simple root $\alpha$ of
$\Psi^+$. Then $\Gamma=\widehat{\Psi^+_{\emptyset,\{\alpha\}}}$.
\end{lemma}

\begin{proof}
The proof will be a case-by-case analysis. But before that we make a
few case-free observations.

(1) The sets $\widehat{\alpha}$ and $\widehat{-\alpha}$ are both contained
in $\Gamma$.

\emph{Proof of (1)}. Note that $I_{\Gamma}=A_{\Gamma}=\Psi^+_{\emptyset,\{\alpha\}}$. We claim that both $\alpha$ and $-\alpha$ are contained in  $A_{\Phi_x\cup \Phi_y}$. If $-\alpha\not\in A_{\Phi_x\cup \Phi_y}$, then $A_{\Phi_x\cup \Phi_y}\subset \Psi^+$ (note that $A_{\Phi_x\cup \Phi_y}\subset A_{\Gamma}$). Consequently $\widehat{-\alpha}\cap \Gamma=\emptyset$ (as one easily sees that $A_{\overline{S}}\subset \overline{A_S}$ for any $S\subset \widehat{\Phi}$), which is a contradiction. The same reasoning works for $\alpha$ in place of $-\alpha$. Therefore $\alpha_0$ and
$(-\alpha)_0$ are in $\Phi_x\bigcup \Phi_y$ by Lemma
\ref{lem:startfrombottom}. By Lemma \ref{lem:closedsetktoinfty}
we get the assertion.

(2) There is at least one simple root $\rho_0\in
\widetilde{\Phi}^+$ such that $\rho\neq \pm\alpha$ and
$\widehat{\rho}$ is contained in $\Gamma.$

\emph{Proof of (2)}. Because $\alpha_0$ and $(-\alpha)_0$ cannot both be simple roots in
$\widetilde{\Phi}^+$ (the simple roots of $\widetilde{\Phi}^+$ are $\gamma,\gamma$ simple in $\Phi^+$ and $-\eta+\delta$ where $\eta$ is the highest root in $\Phi^+$) and we assume that $x$ and $y$ have trivial
meet, there must be $\rho_0\in \Phi_x\bigcup \Phi_y, \rho_0\neq \alpha_0, (-\alpha)_0$ and $\rho_0$ is simple
in $\widetilde{\Phi}^+$. Since $\rho\in A_{\Gamma}=I_{\Gamma}$, this will cause $\widehat{\rho}$ to be
contained in $\Gamma$ by Lemma \ref{lem:closedsetktoinfty}.

(3) If neither  $\alpha_0$ nor $(-\alpha)_0$ is simple in
$\widetilde{\Phi}^+$, there are two simple roots $\rho_0,
\gamma_0\in \widetilde{\Phi}^+$ such that $\widehat{\rho}$ and
$\widehat{\gamma}$ are both contained in $\Gamma.$

\emph{Proof of (3)}. Again because $x$ and $y$ have trivial meet we must have at
least two simple roots of $\widetilde{\Phi}^+$ in $\Phi_x\cup \Phi_y$. Suppose that these two simple roots are $\rho_0,
\gamma_0$ (with $\gamma, \rho\in \Phi$). Then Lemma
\ref{lem:closedsetktoinfty} forces the assertion.

Now we carry out the case study.

$\mathbf{A}_2:$

The two simple roots of $\Phi^+$ are denoted by $\alpha,\beta.$ First we consider the case

$I_{\Gamma}=A_{\Gamma}=\{\alpha,\beta,\alpha+\beta,-\alpha\}$.

By the above fact (1), the sets
$\widehat{\alpha}$ and $\widehat{-\alpha}$ have to be contained in $\Gamma$. Then by the above fact (2), the set $\widehat{\beta}$ also has to be
contained in $\Gamma$. Then $\widehat{\alpha+\beta}$ can clearly be
generated by $\widehat{\alpha}\cup \widehat{\beta}$ and is in $\Gamma$ (since $\alpha+\beta+k\delta=(\alpha)+(\beta+k\delta)$).

Any permutation of the simple roots $\alpha,\beta,-\alpha-\beta+\delta$ induces an
isomorphism of the affine root system $\widetilde{\Phi}$ while preserving the positive system $\widetilde \Phi^+$.   By making use of this
 symmetry of $\widetilde A_2$, it is sufficient to treat the above case.

$\mathbf{B}_2:$

The short simple root is denoted by $\alpha$ and the long simple root
is denoted by $\beta.$

(1)
$I_{\Gamma}=A_{\Gamma}=\{\alpha,\beta,\alpha+\beta,2\alpha+\beta,-\alpha\}$.

By the above fact (1) the sets
$\widehat{\alpha}$ and $\widehat{-\alpha}$ have to be contained in $\Gamma$. By the above fact (2) we also have that $\widehat{\beta}\subset \Gamma$. Then $\widehat{\alpha+\beta}$ and
$\widehat{2\alpha+\beta}$ can clearly be generated by $\widehat{\{\alpha,-\alpha,\beta\}}$ and are thus in
$\Gamma$ (since $\alpha+\beta+k\delta=(\alpha)+(\beta+k\delta)$ and $2\alpha+\beta+k\delta=2(\alpha)+(\beta+k\delta)$).

A permutation of the simple roots $\beta,\delta-2\alpha-\beta$ induces an
isomorphism of the affine root system $\widetilde{\Phi}$ while preserving the positive system $\widetilde \Phi^+$.   By making use of this
 symmetry of $\widetilde B_2$,
the case where $I_{\Gamma}=A_{\Gamma}=\{-2\alpha-\beta,-\alpha,-\beta,-\alpha-\beta,\alpha\}$ can be proved in the same way.

(2)
$I_{\Gamma}=A_{\Gamma}=\{-2\alpha-\beta,\beta,\alpha+\beta,2\alpha+\beta,-\alpha\}$.

By the above fact (1) the sets
$\widehat{2\alpha+\beta}$ and $\widehat{-2\alpha-\beta}$ have to be contained in $\Gamma$. By the above fact (2) the set $\widehat{\beta}$
is also contained in $\Gamma$. Then note that
$\frac{1}{2}(2\alpha+\beta)+\frac{1}{2}\beta=\alpha+\beta,
-\alpha+\delta=\frac{1}{2}(-2\alpha-\beta+\delta)+\frac{1}{2}(\beta+\delta)$.
Therefore $\widehat{\alpha+\beta}$ and $\widehat{-\alpha}$ are in
$\Gamma$ by Lemma \ref{lem:closedsetktoinfty}.

Again a permutation of the simple roots $\beta,\delta-2\alpha-\beta$ induces an
isomorphism of the affine root system $\widetilde{\Phi}$ while preserving the positive system $\widetilde \Phi^+$.   By making use of this
 symmetry of $\widetilde B_2$,
the case where $I_{\Gamma}=A_{\Gamma}=\{-2\alpha-\beta,\beta,-\beta,-\alpha-\beta,-\alpha\}$ can be proved in the same way.

(3)
$I_{\Gamma}=A_{\Gamma}=\{-2\alpha-\beta,\beta,\alpha+\beta,-\alpha-\beta,-\alpha\}$.

By the above fact (1), the sets
$\widehat{\alpha+\beta}$ and $\widehat{-\alpha-\beta}$ are contained in $\Gamma$. Since neither $\alpha+\beta$ nor $\delta-\alpha-\beta$ is simple, by the above fact (3), the sets $\widehat{\beta},
\widehat{-2\alpha-\beta}$ have to be contained in $\Gamma$. Then
note that $-\alpha+\delta=(-\alpha-\beta+\delta)+\beta$. So the set
$\widehat{-\alpha}$ is in $\Gamma$  by Lemma \ref{lem:closedsetktoinfty}.

(4)
$I_{\Gamma}=A_{\Gamma}=\{-2\alpha-\beta,2\alpha+\beta,-\beta,-\alpha-\beta,\alpha\}$.

By the above fact (1) the sets
$\widehat{2\alpha+\beta}$ and $\widehat{-2\alpha-\beta}$
have to be contained in $\Gamma$. By the above fact (2) we also have that $\widehat{\alpha}\subset \Gamma$. Then note that
$-\alpha-\beta+\delta=(-2\alpha-\beta+\delta)+\alpha$ and
$-\beta+\delta=(-\alpha-\beta+\delta)+\alpha$. So
$\widehat{-\alpha-\beta}$ and $\widehat{-\beta}$ are in $\Gamma$  by Lemma \ref{lem:closedsetktoinfty}.

Again a permutation of the simple roots $\beta,\delta-2\alpha-\beta$ induces an
isomorphism of the affine root system $\widetilde{\Phi}$ while preserving the positive system $\widetilde \Phi^+$.   By making use of this
 symmetry of $\widetilde B_2$,
the case where $I_{\Gamma}=A_{\Gamma}=\{\alpha+\beta,2\alpha+\beta,-\beta,\beta,\alpha\}$ can be proved in the same way.

(5)
$I_{\Gamma}=A_{\Gamma}=\{\alpha+\beta,2\alpha+\beta,-\beta,-\alpha-\beta,\alpha\}$.

By the above fact (3) this situation cannot happen as there
will at most be one simple root in $\Gamma.$

$\textbf{G}_2$:

The long simple root is denoted by $\beta$ and the short simple root is
denoted by $\alpha.$

(1)
$I_{\Gamma}=A_{\Gamma}=\{\alpha,\beta,\alpha+\beta,2\alpha+\beta,3\alpha+\beta,3\alpha+2\beta,-\alpha\}$.

By the above fact (1) the sets
$\widehat{\alpha}$ and $\widehat{-\alpha}$ have to be
contained in $\Gamma$. By the above fact (2) we also have that $\widehat{\beta}\subset \Gamma$. Then the sets $\widehat{\alpha+\beta}$,
$\widehat{2\alpha+\beta}$, $\widehat{3\alpha+\beta}$ and
$\widehat{3\alpha+2\beta}$ can clearly be generated and are all in
$\Gamma$ thanks to Lemma \ref{lem:closedsetktoinfty}.

(2)
$I_{\Gamma}=A_{\Gamma}=\{-3\alpha-\beta,\beta,\alpha+\beta,2\alpha+\beta,3\alpha+\beta,3\alpha+2\beta,-\alpha\}$.

By the above fact (3) this situation cannot happen as there
will at most be one simple root in $\Gamma.$

(3)
$I_{\Gamma}=A_{\Gamma}=\{-3\alpha-\beta,\beta,\alpha+\beta,2\alpha+\beta,-2\alpha-\beta,3\alpha+2\beta,-\alpha\}$.

By the above fact (3) this situation cannot happen as there
will at most be one simple root in $\Gamma.$

(4)
$I_{\Gamma}=A_{\Gamma}=\{-3\alpha-\beta,\beta,\alpha+\beta,-3\alpha-2\beta,-2\alpha-\beta,3\alpha+2\beta,-\alpha\}$.

By the above fact (1) the sets
$\widehat{3\alpha+2\beta}$ and $\widehat{-3\alpha-2\beta}$ are contained in $\Gamma$. By the above fact (2) the set $\widehat{\beta}$
also has to be contained in $\Gamma$. Then note that
$\alpha+\beta=\frac{1}{3}(3\alpha+2\beta)+\frac{1}{3}\beta$,
$-\alpha+\delta=\frac{1}{3}(-3\alpha-2\beta+\delta)+\frac{2}{3}(\beta+\delta)$,
$-3\alpha-\beta+\delta=(-3\alpha-2\beta+\delta)+\beta$ and
$-2\alpha-\beta+\delta=(-3\alpha-2\beta+\delta)+(\alpha+\beta).$
Hence
$\widehat{\alpha+\beta},\widehat{-\alpha},\widehat{-3\alpha-\beta}$
and $\widehat{-2\alpha-\beta}$ are all in $\Gamma$ thanks to Lemma \ref{lem:closedsetktoinfty}.

(5)
$I_{\Gamma}=A_{\Gamma}=\{-3\alpha-\beta,\beta,\alpha+\beta,-3\alpha-2\beta,-2\alpha-\beta,-\alpha-\beta,-\alpha\}$.

By the above fact (1) the sets
$\widehat{\alpha+\beta}$ and $\widehat{-\alpha-\beta}$ are contained in $\Gamma$. Since neither $\alpha+\beta$ nor $\delta-\alpha-\beta$ is simple, by the above fact (3)
the sets $\widehat{\beta}$ and $\widehat{-3\alpha-2\beta}$ are also contained in $\Gamma$. Then
note that
$-\alpha+\delta=\frac{1}{3}(-3\alpha-2\beta+\delta)+\frac{2}{3}(\beta+\delta)$,
$-3\alpha-\beta+\delta=(-3\alpha-2\beta+\delta)+\beta$ and
$-2\alpha-\beta+\delta=(-3\alpha-2\beta+\delta)+(\alpha+\beta).$
Hence $\widehat{-\alpha},\widehat{-3\alpha-\beta}$ and
$\widehat{-2\alpha-\beta}$ are all in $\Gamma$ thanks to Lemma \ref{lem:closedsetktoinfty}.

(6)
$I_{\Gamma}=A_{\Gamma}=\{-3\alpha-\beta,\beta,-\beta,-3\alpha-2\beta,-2\alpha-\beta,-\alpha-\beta,-\alpha\}$.

By the above fact (1) the sets $\widehat{\beta}$ and $\widehat{-\beta}$ have to be contained in $\Gamma$.
By the above fact (2) we also have that
$\widehat{-3\alpha-2\beta}\subset \Gamma.$
 Then note that
$-\alpha+\delta=\frac{1}{3}(-3\alpha-2\beta+\delta)+\frac{2}{3}(\beta+\delta)$,
$-3\alpha-\beta+\delta=(-3\alpha-2\beta+\delta)+\beta$,
$-2\alpha-\beta+\delta=\frac{2}{3}(-3\alpha-2\beta+\delta)+\frac{1}{3}(\beta+\delta)$
and
$\frac{1}{3}(-3\alpha-2\beta+2\delta)+\frac{1}{3}(-\beta+\delta)=-\alpha-\beta+\delta$.
Hence $\widehat{-\alpha}, \widehat{-3\alpha-\beta},
\widehat{-2\alpha-\beta}$ and $\widehat{-\alpha-\beta}$ are all in
$\Gamma$ thanks to Lemma \ref{lem:closedsetktoinfty}.

(7)
$I_{\Gamma}=A_{\Gamma}=\{-3\alpha-\beta,\alpha,-\beta,-3\alpha-2\beta,-2\alpha-\beta,-\alpha-\beta,-\alpha\}$.

By the above fact (1) the sets $\widehat{\alpha}$ and $\widehat{-\alpha}$ have to be contained in $\Gamma$.
By the above fact (2) we also have that
$\widehat{-3\alpha-2\beta}\subset \Gamma.$
 Then note that
$-\alpha-\beta+\delta=\frac{1}{2}(-3\alpha-2\beta+\delta)+\frac{1}{2}(\alpha+\delta)$,
$-\beta+\delta=(-\alpha-\beta+\delta)+\alpha,
-2\alpha-\beta+\delta=\frac{1}{2}(-3\alpha-2\beta+\delta)+\frac{1}{2}(-\alpha+\delta)$.
Hence $\widehat{-\alpha-\beta}, \widehat{-\beta}$ and
$\widehat{-2\alpha-\beta}$ are all in $\Gamma$ thanks to Lemma \ref{lem:closedsetktoinfty}. Finally we note that
in this case the two elements $x,y$ are of the form $s_{\alpha}\cdots$ and
$s_{\delta-3\alpha-2\beta}\cdots$. Because $\Gamma$ is infinite they
could not be $s_{\alpha}$ and $s_{\delta-3\alpha-2\beta}$ (note that $s_{\alpha}$ and $s_{\delta-3\alpha-2\beta}$ commute to each other). Because
the two elements have trivial meet we cannot have
$s_{\alpha}s_{\delta-3\alpha-2\beta}\cdots$ or
$s_{\delta-3\alpha-2\beta}s_{\alpha}$. Finally if one of the elements is
$s_{\alpha}s_{\beta}\cdots$ then we will have that $3\alpha+\beta\in
A_{\Gamma}$, which is a contradiction. So we must have one element being
equal to $s_{\delta-3\alpha-2\beta}s_{\beta}\cdots$. Therefore we have that
$-3\alpha-\beta+\delta\in \Phi_x\cup \Phi_y$. So we conclude that
$\widehat{-3\alpha-\beta}\subset \Gamma$ thanks to Lemma \ref{lem:closedsetktoinfty}.

(8)
$I_{\Gamma}=A_{\Gamma}=\{-3\alpha-\beta,\alpha,-\beta,-3\alpha-2\beta,-2\alpha-\beta,-\alpha-\beta,3\alpha+\beta\}$.

By the above fact (1) the sets
$\widehat{-3\alpha-\beta}$ and $\widehat{3\alpha+\beta}$ are contained in $\Gamma$. Since neither $-3\alpha-\beta+
\delta$ nor $3\alpha+\beta$ is simple, by the above fact (3) we also have that $\widehat{\alpha}\subset \Gamma$ and $\widehat{-3\alpha-2\beta}\subset \Gamma.$
 Then note that
$-2\alpha-\beta+\delta=(-3\alpha-\beta+\delta)+\alpha$,
$-\alpha-\beta+\delta=(-2\alpha-\beta+\delta)+\alpha,$ and
$-\beta+\delta=(-\alpha-\beta+\delta)+\alpha$. Hence
$\widehat{-2\alpha-\beta}, \widehat{-\alpha-\beta}$ and
$\widehat{-\beta}$ are all in $\Gamma$  thanks to Lemma \ref{lem:closedsetktoinfty}.

(9)
$I_{\Gamma}=A_{\Gamma}=\{2\alpha+\beta,\alpha,-\beta,-3\alpha-2\beta,-2\alpha-\beta,-\alpha-\beta,3\alpha+\beta\}$.

By the above fact (1) the sets
$\widehat{-2\alpha-\beta}$ and $\widehat{2\alpha+\beta}$ are contained in $\Gamma$. Since neither $-2\alpha-\beta+\delta$ nor $2\alpha+\beta$ is simple, by the above fact (3) the sets $\widehat{\alpha}$ and $\widehat{-3\alpha-2\beta}$
have to be in $\Gamma.$ Then note that
$3\alpha+\beta=(2\alpha+\beta)+\beta$,
$-\alpha-\beta+\delta=(-2\alpha-\beta+\delta)+\alpha,$ and
$-\beta+\delta=(-\alpha-\beta+\delta)+\alpha$. Hence
$\widehat{3\alpha+\beta}, \widehat{-\alpha-\beta}$ and
$\widehat{-\beta}$ are all in $\Gamma$ thanks to Lemma \ref{lem:closedsetktoinfty}.

(10)
$I_{\Gamma}=A_{\Gamma}=\{2\alpha+\beta,\alpha,-\beta,-3\alpha-2\beta,3\alpha+2\beta,-\alpha-\beta,3\alpha+\beta\}$.

By the above fact (1) the sets
$\widehat{-3\alpha-2\beta}$ and $\widehat{3\alpha+2\beta}$ are contained in $\Gamma$.  By the above fact (2) the set
$\widehat{\alpha}$ is also contained in $\Gamma.$ Then note that
$2\alpha+\beta=\frac{1}{2}(3\alpha+2\beta)+\frac{1}{2}\alpha$,
$3\alpha+\beta=(2\alpha+\beta)+\alpha,
-\alpha-\beta+\delta=\frac{1}{2}(-3\alpha-2\beta+\delta)+\frac{1}{2}(\alpha+\delta)$
and $-\beta+\delta=(-\alpha-\beta+\delta)+\alpha$. Therefore
$\widehat{2\alpha+\beta}, \widehat{3\alpha+\beta},
\widehat{-\alpha-\beta}$ and $\widehat{-\beta}$ are all in $\Gamma$ thanks to Lemma \ref{lem:closedsetktoinfty}.

(11)
$I_{\Gamma}=A_{\Gamma}=\{2\alpha+\beta,\alpha,-\beta,\alpha+\beta,3\alpha+2\beta,-\alpha-\beta,3\alpha+\beta\}$.

By the above fact (3) this situation cannot happen as there
will at most be one simple root in $\Gamma.$

(12)
$I_{\Gamma}=A_{\Gamma}=\{2\alpha+\beta,\alpha,-\beta,\alpha+\beta,3\alpha+2\beta,\beta,3\alpha+\beta\}$.

By the above fact (1) the sets
$\widehat{-\beta}$ and $\widehat{\beta}$ have to be
contained in $\Gamma$. By the above fact (2) we also have that $\widehat{\alpha}\subset \Gamma.$ Then $\widehat{\alpha+\beta}$,
$\widehat{2\alpha+\beta}$, $\widehat{3\alpha+\beta}$ and
$\widehat{3\alpha+2\beta}$ can clearly be generated and are all in
$\Gamma$ thanks to Lemma \ref{lem:closedsetktoinfty}.
\end{proof}

\begin{lemma}\label{lem:orthocasetwo}
Suppose that $x,y\in \widetilde{W}$ and that $x\wedge y=e$. Let
$\Gamma=\overline{\Phi_x\bigcup \Phi_y}$ and assume that
$I_{\Gamma}=A_{\Gamma}=\Phi$.  Then
$\Gamma\supset\widehat{\Psi^+_{\emptyset,\{\alpha\}}}$ for some
positive system $\Psi^+$ and a simple root $\alpha$ of $\Psi^+$.
\end{lemma}

\begin{proof}
Again we begin with a few general observations. We
then make a case-by-case analysis.

(1) Since $x$ and
$y$ have trivial meet, there are two simple roots $\rho_0,
\gamma_0\in \Phi_x\bigcup \Phi_y$ where $\rho, \gamma\in \Phi$.

(2) There is no closed half-plane in $\mathbb{R}^2$ with 0 in its boundary containing all $\eta\in \Phi$ such that $\eta_0\in \Phi_x\bigcup \Phi_y$.

\emph{Proof of (2).} Assume that to the contrary such closed half-plane exists. Let this closed half-plane be defined by $\{v|(u,v)\geq 0\}$ where $u$ is a vector in $\mathbb{R}^2$.  Suppose that $wu, w\in W$ is contained in the closure of the fundamental chamber. Then $\Phi\cap \{v|(wu,v)\geq 0\}$ is either $\Phi^+$ (if no root is perpendicular to $wu$) or $\Phi^+_{\emptyset,\{\alpha\}}$ (if there exists a pair of roots ($\pm\alpha$)  perpendicular to $wu$) where $\alpha$ is a simple root of $\Phi^+$. Then $\Phi\cap \{v|(u,v)\geq 0\}$ is either $w^{-1}\Phi^+$ or $(w^{-1}\Phi^+)_{\emptyset,\{w^{-1}(\alpha)\}}$. So $A_{\Phi_x\cup \Phi_y}$ is contained in either $w^{-1}\Phi^+$ or $(w^{-1}\Phi^+)_{\emptyset,\{w^{-1}(\alpha)\}}$ by Lemma \ref{lem:startfrombottom}. Then $A_{\overline{\Phi_x\cup \Phi_y}}$ is also contained in either $w^{-1}\Phi^+$ or $(w^{-1}\Phi^+)_{\emptyset,\{w^{-1}(\alpha)\}}$ (since it is easy to see that $A_{\overline{S}}\subset \overline{A_S}$ for any $S\subset \widehat{\Phi}$) and hence cannot be the whole $\Phi$. This is a contradiction.

Combining the fact (1), (2) and Lemma \ref{lem:closedsetktoinfty} we have that

 (3) There exist $\widehat{\rho}, \widehat{\gamma}$ and $\widehat{\eta}, \rho, \gamma, \eta\in \Phi$ all completely in $\Gamma$ with the properties

  (a) $\rho_0, \gamma_0$ simple;

  (b) there is no closed half-plane in $\mathbb{R}^2$ with 0 in its boundary containing all $\rho, \gamma$ and $\eta$.

Now we carry out a case-by-case analysis.

$\textbf{A}_2:$

The two simple roots are denoted by $\alpha$ and $\beta.$
By the above fact (3), at least two of the three sets $\widehat{\alpha},\widehat{\beta},\widehat{-\alpha-\beta}$ are contained in $\Gamma.$
 Any permutation of the simple roots $\alpha,\beta,-\alpha-\beta+\delta$ induces an
isomorphism of the affine root system $\widehat{\Phi}$ while preserving the positive system $\widetilde \Phi^+$.   By making use of this
 symmetry of $\widetilde A_2$,  it is enough to prove this in the case where $\widehat{\alpha}$ and $\widehat{\beta}$ are in $\Gamma$.
 Since $\widehat{\alpha},\widehat{\beta}\subset \Gamma$, it follows that $\widehat{\alpha+\beta}\subset \Gamma$ (since $\alpha+\beta+k\delta=(\alpha)+(\beta+k\delta)$).
 The above fact (3) also forces that at least one of the three sets $\widehat{-\alpha},\widehat{-\beta}$, $\widehat{-\alpha-\beta}$ is contained in $\Gamma$.
If any of $\widehat{-\alpha}$ and $\widehat{-\beta}$ is in $\Gamma$,
then either $\Gamma\supset \widehat{\Phi^+_{\emptyset,\{\alpha\}}}$ or $\Gamma\supset \widehat{\Phi^+_{\emptyset,\{\beta\}}}$. If $\widehat{-\alpha-\beta}\in \Gamma$, then  the whole
$\widetilde{\Phi}^+$ is in $\Gamma$ (as all three simple roots of $\widehat{\Phi}$ are in $\Gamma$). So we see that the assertion holds
in this case.

$\textbf{B}_2:$

Suppose that $\beta$ is the long simple root and that $\alpha$ is the short
simple root. By the above fact (3), at least two of the three sets $\widehat{\alpha},\widehat{\beta}$, $\widehat{-2\alpha-\beta}$ are contained in $\Gamma.$

Case (I). Suppose that $\widehat{\beta}$ and $\widehat{\alpha}$ are both
contained in $\Gamma.$  As a consequence, $\widehat{\alpha+\beta}$ and $\widehat{2\alpha+\beta}$ are both contained in $\Gamma$ (since $\alpha+\beta+k\delta=(\alpha)+(\beta+k\delta), 2\alpha+\beta+k\delta=2(\alpha)+(\beta+k\delta)$).
Again by the above fact (3), one of the sets $\widehat{-\alpha},\widehat{-\beta},\widehat{-\alpha-\beta}$, $\widehat{-2\alpha-\beta}$ has to be contained in $\Gamma.$
If any of $\widehat{-\beta}$ and $\widehat{-\alpha}$ are contained in $\Gamma$, we are
done as $\widehat{\Phi^+_{\emptyset,\{\alpha\}}}$ or $\widehat{\Phi^+_{\emptyset,\{\beta\}}}$ is contained in $\Gamma$. Now
suppose that $-\beta-2\alpha+\delta\in \Gamma$, then
$\widetilde{\Phi}^+\subset \Gamma$ (as all three simple roots are in $\Gamma$). Finally suppose that
$\widehat{-\beta-\alpha}\subset \Gamma$. Since
$-\beta+k\delta=-\beta-\alpha+k\delta+\alpha\in \Gamma,$ we have that $\widehat{-\beta}\subset \Gamma$. So we see that the
assertion holds in this case.

Case (II). Suppose that $\widehat{-\beta-2\alpha}$ and $\widehat{\beta}$
are both contained in $\Gamma.$ As a consequence $\widehat{-\alpha}\subset \Gamma$ (as $-\alpha+k\delta=\frac{1}{2}(-\beta-2\alpha+k\delta)+\frac{1}{2}(\beta+k\delta)$).
By the above fact (3) at least one of the sets $\widehat{-\beta},\widehat{\alpha}$, $\widehat{2\alpha+\beta}$ has to be contained in $\Gamma$ (otherwise only $\widehat{\beta},\widehat{\alpha+\beta},\widehat{-\alpha},\widehat{-2\alpha-\beta}$ and $\widehat{-\alpha-\beta}$ may be completely contained in $\Gamma$ and the roots $\beta,\alpha+\beta,-\alpha,-2\alpha-\beta,-\alpha-\beta$ are contained in a closed half-plane with the origin in its boundary).

If $\widehat{-\beta}\subset \Gamma,$ then $\widehat{-\alpha-\beta}\subset \Gamma$ since $-\alpha-\beta+k\delta=\frac{1}{2}(-\beta+k\delta)+\frac{1}{2}(-2\alpha-\beta+k\delta)$.
Then $\widehat{\Psi^+_{\emptyset,\{\beta\}}}\subset \Gamma$ where $\Psi^+=\{\beta,-\alpha,-2\alpha-\beta,-\alpha-\beta\}$.

If $\widehat{2\alpha+\beta}\subset \Gamma$, then $\widehat{\alpha+\beta}\subset \Gamma$ since $\alpha+\beta+k\delta=\frac{1}{2}(2\alpha+\beta+2k\delta)+\frac{1}{2}\beta$.
Then $\widehat{\Psi^+_{\emptyset,\{2\alpha+\beta\}}}\subset \Gamma$ where $\Psi^+=\{\alpha+\beta,\beta,-\alpha,2\alpha+\beta\}.$

If $\widehat{\alpha}\in \Gamma$, then the
whole $\widetilde{\Phi}^+$ is in $\Gamma$ (as all three simple roots are in $\Gamma$). So we see that the assertion
holds in this case.

Case (III). Suppose that $\widehat{-\beta-2\alpha}$ and
$\widehat{\alpha}$ are both
contained in $\Gamma.$ A permutation of the simple roots $\beta,\delta-2\alpha-\beta$ induces an
isomorphism of the affine root system $\widetilde{\Phi}$ while preserving the positive system $\widetilde \Phi^+$.   By making use of this
 symmetry of $\widetilde B_2$,
  this case can be proved by using the same reasoning as in case (I).

$\textbf{G}_2$

The long simple root is denoted by $\beta$ and the short simple root is
denoted by $\alpha.$ By the above fact (3), at least two of the three sets $\widehat{\alpha},\widehat{\beta}$ and $\widehat{-3\alpha-2\beta}$ are contained in $\Gamma.$

Case (I). Suppose that $\widehat{\alpha}$ and $\widehat{\beta}$ are both
in $\Gamma.$ Then clearly $\widehat{\Phi^+}$ is contained in $\Gamma$ where $\Phi^+=\{\alpha,\beta,\alpha+\beta,2\alpha+\beta,3\alpha+\beta,3\alpha+2\beta\}$ is the standard positive system.
Again by the above fact (3), one of the sets $\widehat{-\alpha-\beta},\widehat{-2\alpha-\beta},\widehat{-3\alpha-\beta},\widehat{-3\alpha-2\beta},\widehat{-\beta}$ has to be contained in $\Gamma$ (since otherwise only $\widehat{\alpha},\widehat{\beta},\widehat{\alpha+\beta},\widehat{2\alpha+\beta},\widehat{3\alpha+\beta},\widehat{3\alpha+2\beta},\newline \widehat{-\alpha}$ may be completely contained in $\Gamma$, but the roots $\alpha,\beta,\alpha+\beta,2\alpha+\beta,3\alpha+\beta,3\alpha+2\beta,-\alpha$ are contained in a closed half-plane with the origin in its boundary).

  If $\widehat{-3\alpha-2\beta}\subset
\Gamma$, then the whole $\widetilde{\Phi}^+$ is in $\Gamma$ (as all three simple roots are in $\Gamma$).

Now
suppose that $\widehat{-3\alpha-\beta}\subset \Gamma$. Then
$-\beta+k\delta=(-3\alpha-\beta+k\delta)+3\alpha\in \Gamma.$ Therefore $\widehat{-\beta}\subset \Gamma$ and then $\widehat{\Phi^+_{\emptyset,\{\beta\}}}\subset \Gamma.$

 Now suppose that
$\widehat{-2\alpha-\beta}\subset \Gamma$. Then
$-\beta+k\delta=(-2\alpha-\beta+k\delta)+2\alpha\in \Gamma.$ Therefore $\widehat{-\beta}\subset \Gamma$  and then $\widehat{\Phi^+_{\emptyset,\{\beta\}}}\subset \Gamma.$

 Now suppose that
$\widehat{-\alpha-\beta}\subset \Gamma$. Then
$-\beta+k\delta=(-\alpha-\beta+k\delta)+\alpha\in \Gamma.$ Therefore $\widehat{-\beta}\subset \Gamma$  and then $\widehat{\Phi^+_{\emptyset,\{\beta\}}}\subset \Gamma.$

Now suppose that $\widehat{-\beta}\subset \Gamma$. Then $\Gamma\supset \widehat{\Phi^+_{\emptyset,\{\beta\}}}$.

Case (II). Suppose that $\widehat{-3\alpha-2\beta}$ and $\widehat{\beta}$
are both in $\Gamma.$
Note that $-\alpha+k\delta=\frac{1}{3}(-3\alpha-2\beta+3k\delta)+\frac{2}{3}\beta$, $-3\alpha-\beta+k\delta=(-3\alpha-2\beta+k\delta)+\beta$ and $-2\alpha-\beta+k\delta=\frac{2}{3}(-3\alpha-2\beta+k\delta)+\frac{1}{3}(\beta+k\delta)$. So the sets $\widehat{-\alpha},\widehat{-3\alpha-\beta}$ and $\widehat{-2\alpha-\beta}$ are all contained in $\Gamma.$

By the above fact (3) one of the sets $\widehat{\alpha},\widehat{3\alpha+\beta},\widehat{2\alpha+\beta},\widehat{-\beta},\widehat{-\alpha-\beta}$ has to be contained in $\Gamma$ (as otherwise only $\widehat{3\alpha+2\beta},\widehat{-3\alpha-2\beta},\widehat{\alpha+\beta},\widehat{\beta},\widehat{-\alpha},\widehat{-3\alpha-\beta},\widehat{-2\alpha-\beta}$ may be completely contained in $\Gamma$ but the roots $3\alpha+2\beta,-3\alpha-2\beta, \alpha+\beta, \beta, -\alpha, -3\alpha-\beta, -2\alpha-\beta$ are contained in a closed half-plane with the origin in its boundary).

Suppose that $\widehat{-\alpha-\beta}$ is contained in $\Gamma.$ The above fact (3) forces that one of the sets $\widehat{\alpha},\widehat{3\alpha+\beta},\widehat{2\alpha+\beta},\widehat{3\alpha+2\beta},\widehat{-\beta}$ must be contained in $\Gamma$ (as otherwise only  $\widehat{\alpha+\beta},\widehat{-\alpha-\beta},\widehat{\beta},\widehat{-\alpha},\widehat{-3\alpha-2\beta},\widehat{-3\alpha-\beta},\widehat{-2\alpha-\beta}$ may be completely contained in $\Gamma$ but the roots $\alpha+\beta, -\alpha-\beta, \beta, -\alpha, -3\alpha-2\beta, -3\alpha-\beta, -2\alpha-\beta$ are contained in a closed half-plane with the origin in its boundary). If $\widehat{\{\alpha,-\alpha-\beta\}}\subset \Gamma$, then we have that $\Gamma=\widehat{\Phi}$ as all three simple roots are in $\Gamma$. If $\widehat{\{3\alpha+\beta,-\alpha-\beta\}}\subset \Gamma$, then $\alpha+\beta+k\delta=\frac{1}{3}(3\alpha+\beta+3k\delta)+\frac{2}{3}\beta$. Consequently $\Gamma\supset \widehat{\Psi^+_{\emptyset,\{\alpha+\beta\}}}$ where $\Psi^+=\{\alpha+\beta,\beta,-\alpha,-3\alpha-\beta,-2\alpha-\beta,-3\alpha-2\beta\}$. If $\widehat{\{2\alpha+\beta,-\alpha-\beta\}}\subset \Gamma$, $\alpha+\beta+k\delta=\frac{1}{2}(2\alpha+\beta+2k\delta)+\frac{1}{2}\beta$. Consequently $\Gamma\supset \widehat{\Psi^+_{\emptyset,\{\alpha+\beta\}}}$ where $\Psi^+=\{\alpha+\beta,\beta,-\alpha,-3\alpha-\beta,-2\alpha-\beta,-3\alpha-2\beta\}$.
If $\widehat{\{3\alpha+2\beta,-\alpha-\beta\}}\subset \Gamma$, $\alpha+\beta+k\delta=\frac{1}{3}(3\alpha+2\beta+3k\delta)+\frac{1}{3}\beta$. Consequently $\Gamma\supset \widehat{\Psi^+_{\emptyset,\{\alpha+\beta\}}}$ where $\Psi^+=\{\alpha+\beta,\beta,-\alpha,-3\alpha-\beta,-2\alpha-\beta,-3\alpha-2\beta\}$.
If $\widehat{\{-\beta,-\alpha-\beta\}}\subset \Gamma$, then $\Gamma\supset \widehat{\Psi^+_{\emptyset,\{\beta\}}}$ where $\Psi^+=\{-\alpha-\beta,\beta,-\alpha,-3\alpha-\beta,-2\alpha-\beta,-3\alpha-2\beta\}$.

Suppose that $\widehat{\alpha}$ is contained in $\Gamma$. Then $\Gamma=\widehat{\Phi}$ as all three simple roots are in $\Gamma.$

Suppose that $\widehat{3\alpha+\beta}\subset \Gamma.$ Then $\alpha+\beta+k\delta=\frac13(3\alpha+\beta+3k\delta)+\frac23\beta$. Hence $\widehat{\alpha+\beta}\subset \Gamma.$
We also have that $3\alpha+2\beta+k\delta=(3\alpha+\beta+k\delta)+
\beta$. So $\widehat{3\alpha+2\beta}\subset \Gamma$
Consequently $\Gamma\supset \widehat{\Psi^+_{\emptyset, \{3\alpha+2\beta\}}}$ where $\Psi^+=\{3\alpha+2\beta, \alpha+\beta, \beta, -\alpha, -3\alpha-\beta, -2\alpha-\beta\}$.

Suppose that $\widehat{2\alpha+\beta}\subset \Gamma$. Then $\alpha+\beta+k\delta=\frac12(2\alpha+\beta+2k\delta)+\frac12\beta, 3\alpha+2\beta+k\delta=(\alpha+\beta+k\delta)+(2\alpha+\beta)$. Hence $\widehat{\alpha+\beta}\subset \Gamma$ and $\widehat{3\alpha+2\beta}\subset \Gamma.$ Consequently $\Gamma\supset \widehat{\Psi^+_{\emptyset, \{3\alpha+2\beta\}}}$ where $\Psi^+=\{3\alpha+2\beta, \alpha+\beta, \beta, -\alpha, -3\alpha-\beta, -2\alpha-\beta\}$.

Suppose that $\widehat{-\beta}\subset \Gamma.$ Then $-\alpha-\beta+k\delta=\frac{1}{3}(-3\alpha-2\beta+k\delta)+\frac13(-\beta+2k\delta)$.
Consequently $\widehat{\Psi^+_{\emptyset,\{\beta\}}}\subset \Gamma$ where $\Psi^+=\{\beta,-\alpha,-3\alpha-\beta,-2\alpha-\beta,-3\alpha-2\beta,-\alpha-\beta\}$.

Case (III). Suppose that $\widehat{-3\alpha-2\beta}$ and
$\widehat{\alpha}$ are both in $\Gamma.$ Note that $-\beta+k\delta=\frac12(-3\alpha-2\beta+2k\delta)+\frac32\alpha$ and $-\alpha-\beta+k\delta=\frac12(-3\alpha-2\beta+2k\delta)+\frac12\alpha$. Therefore $\widehat{-\alpha-\beta}$ and $\widehat{-\beta}$ are both contained $\Gamma.$

Then the above fact (3) forces that one of the sets $\widehat{\alpha+\beta},\widehat{\beta},\widehat{3\alpha+2\beta}, \widehat{2\alpha+\beta},\widehat{3\alpha+\beta}$ must be contained in $\Gamma$ (since otherwise only $\widehat{\alpha}, \widehat{-\beta},\widehat{-\alpha-\beta}, \widehat{-3\alpha-2\beta}, \widehat{-3\alpha-\beta}, \newline \widehat{-2\alpha-\beta},\widehat{-\alpha}$ may be completely contained in $\Gamma$ but the roots $\alpha,-\alpha,-\beta, -\alpha-\beta, -3\alpha-2\beta, -3\alpha-\beta, -2\alpha-\beta$ are contained in a closed half-plane with the origin in its boundary).

Suppose that the set $\widehat{\beta}$ is contained in $\Gamma$. Then $\Gamma=\widehat{\Phi}$ since all three simple roots are in $\Gamma$.

Suppose that the set $\widehat{\alpha+\beta}$ is contained in $\Gamma$. Note that $3\alpha+2\beta+k\delta=2(\alpha+\beta)+\alpha+k\delta$, $2\alpha+\beta+k\delta=(\alpha+\beta)+(\alpha+k\delta)$ and $3\alpha+\beta=(\alpha+\beta+k\delta)+2\alpha$. Hence $\widehat{3\alpha+2\beta}, \widehat{3\alpha+\beta}$ and $\widehat{2\alpha+\beta}$ are all contained in $\Gamma$. Then $\Gamma\supset \widehat{\Psi^+_{\emptyset,\{3\alpha+2\beta\}}}$ where $\Psi^+=\{3\alpha+2\beta, 2\alpha+\beta, 3\alpha+\beta, \alpha, -\beta, -\alpha-\beta\}$.

Suppose that the set $\widehat{3\alpha+2\beta}$ is contained in $\Gamma$. Note that $3\alpha+\beta+k\delta=\frac12(3\alpha+2\beta+2k\delta)+\frac{3}{2}\alpha$ and $2\alpha+\beta+k\delta=\frac12(3\alpha+2\beta+2k\delta)+\frac{1}{2}\alpha$. Hence $\widehat{3\alpha+2\beta}, \widehat{3\alpha+\beta}$ and $\widehat{2\alpha+\beta}$ are all contained in $\Gamma$. Then $\Gamma\supset \widehat{\Psi^+_{\emptyset,\{3\alpha+2\beta\}}}$ where $\Psi^+=\{3\alpha+2\beta, 2\alpha+\beta, 3\alpha+\beta, \alpha, -\beta, -\alpha-\beta\}$.

Suppose that the set $\widehat{2\alpha+\beta}$ is contained in $\Gamma$. Note that $3\alpha+\beta+k\delta=(2\alpha+\beta+k\delta)+\alpha$ and thus $\widehat{3\alpha+\beta}\subset \Gamma.$ The above fact (3) again forces that one of the sets $\widehat{\alpha+\beta},\widehat{\beta},\widehat{-\alpha},\widehat{-3\alpha-\beta},\widehat{-2\alpha-\beta}$ has to be contained in $\Gamma$ (since otherwise only  $\widehat{3\alpha+2\beta}, \widehat{2\alpha+\beta}, \widehat{3\alpha+\beta}, \widehat{\alpha}, \widehat{-\beta}, \widehat{-\alpha-\beta}, \widehat{-3\alpha-2\beta}$ may be completely contained in $\Gamma$ but the roots $3\alpha+2\beta, 2\alpha+\beta, 3\alpha+\beta, \alpha, -\beta, -\alpha-\beta, -3\alpha-2\beta$ are contained in a closed half-plane with the origin in its boundary). If $\widehat{\beta}$ or $\widehat{\alpha+\beta}$ is contained in $\Gamma$, then we are reduced to the case treated above. Now assume that $\widehat{\{2\alpha+\beta,-\alpha\}}$ is contained in $\Gamma$. Note that $-2\alpha-\beta+k\delta=\frac{1}{2}(-3\alpha-2\beta+k\delta)+\frac{1}{2}(-\alpha+k\delta)$.
Therefore $\Gamma\supset \widehat{\Psi^+_{\emptyset,\{2\alpha+\beta\}}}$ where $\Psi^+=\{2\alpha+\beta, 3\alpha+\beta, \alpha,-\alpha-\beta, -\beta, -3\alpha-2\beta\}$.
Assume that $\widehat{\{2\alpha+\beta,-2\alpha-\beta\}}\subset \Gamma$. Then $\Gamma\supset \widehat{\Psi^+_{\emptyset,\{2\alpha+\beta\}}}$ where $\Psi^+=\{2\alpha+\beta, 3\alpha+\beta, \alpha,-\alpha-\beta, -\beta, -3\alpha-2\beta\}$. Assume that $\widehat{\{2\alpha+\beta,-3\alpha-\beta\}}\subset \Gamma$. Note that $-2\alpha-\beta+k\delta=\frac13(-3\alpha-2\beta+2k\delta)+\frac13(-3\alpha-\beta+k\delta)$.  Then $\Gamma\supset \widehat{\Psi^+_{\emptyset,\{2\alpha+\beta\}}}$ where $\Psi^+=\{2\alpha+\beta, 3\alpha+\beta, \alpha,-\alpha-\beta, -\beta, -3\alpha-2\beta\}$.

Suppose that the set $\widehat{3\alpha+\beta}$ is contained in $\Gamma$. The above fact (3) again forces that one of the sets $\widehat{\alpha+\beta},\widehat{\beta},\widehat{-\alpha},\widehat{2\alpha+\beta},\widehat{3\alpha+2\beta}$ has to be contained in $\Gamma$ (since otherwise only  $\widehat{3\alpha+\beta}, \widehat{-3\alpha-\beta}, \widehat{-3\alpha-2\beta}, \widehat{\alpha}, \widehat{-\beta}, \widehat{-\alpha-\beta}, \widehat{-2\alpha-\beta}$ may be completely contained in $\Gamma$ but $3\alpha+\beta, -3\alpha-\beta, -3\alpha-2\beta, \alpha, -\beta, -\alpha-\beta, -2\alpha-\beta$ are contained in a closed half-plane with the origin in its boundary). Based on the above discussion we only need to consider the case where $\widehat{\{-\alpha,3\alpha+\beta\}}\subset \Gamma$.

Now if $\beta$ is contained in $\Gamma$, then $\Gamma=\widehat{\Phi}$ since all three simple roots are in $\Gamma$.

If $\beta\not\in \Gamma$, then we must have $x=s_{\alpha}\cdots$ and $y=s_{\delta-3\alpha-2\beta}\cdots$. Since $\overline{\Phi_x\cup \Phi_y}$ is infinite these two elements cannot be $s_{\alpha}$ and $s_{\delta-3\alpha-2\beta}$. Note that $s_{\alpha}$ and $s_{\delta-3\alpha-2\beta}$ commute and $x\wedge y=e.$ Then we must have $s_{\alpha}s_{\beta}\cdots$ or $s_{\delta-3\alpha-2\beta}s_{\beta}\cdots$. If one of the element is $s_{\delta-3\alpha-2\beta}s_{\beta}\cdots$, then $-3\alpha-\beta+\delta\in \Gamma$. Consequently $\widehat{-3\alpha-\beta}$ is contained in $\Gamma$ thanks to Lemma \ref{lem:closedsetktoinfty}.
Note that $-2\alpha-\beta+k\delta=\frac13(-3\alpha-2\beta+2k\delta)+\frac13(-3\alpha-\beta+k\delta)$. Therefore $\widehat{-2\alpha-\beta}\subset \Gamma$.
So we have that $\Gamma\supset \widehat{\Psi^+_{\emptyset,\{3\alpha+\beta\}}}$ where $\Psi^+=\{3\alpha+\beta,\alpha,-\beta,-\alpha-\beta,-3\alpha-2\beta,-2\alpha-\beta\}$.
Finally assume that one of the element is $s_{\delta-3\alpha-2\beta}$. Note that $A_{\Phi_x\cup \Phi_y}$ cannot be contained in $\Xi^+_{\emptyset,\{3\alpha+\beta\}}$ where $\Xi^+=\{3\alpha+\beta, \alpha, -\beta, -\alpha-\beta, -3\alpha-2\beta, -2\alpha-\beta\}$ because otherwise $A_{\overline{\Phi_x\cup \Phi_y}}\subset \overline{A_{\Phi_x\cup \Phi_y}}=\Xi^+_{\emptyset,\{3\alpha+\beta\}}$. Therefore at least one of roots  $-\alpha, \beta, \alpha+\beta, 3\alpha+2\beta, 2\alpha+\beta$ must be contained in $A_{\Phi_x\cup \Phi_y}$. Since one of the element, say $x$, is $s_{-3\alpha-2\beta+\delta}$, this implies that at least one of the roots $-\alpha, \beta, \alpha+\beta, 3\alpha+2\beta, 2\alpha+\beta$ must be contained in $A_{\Phi_y}$. But the fact that $y=s_{\alpha}\cdots$ forces that $-\alpha\not\in A_{\Phi_y}$ thanks to Lemma \ref{closedaffine}. So at least one of the roots $\beta, \alpha+\beta, 3\alpha+2\beta, 2\alpha+\beta$ is contained in $A_{\Phi_y}$. This implies that at least one of the sets $\widehat{\beta},\widehat{\alpha+\beta},\widehat{3\alpha+\beta},\widehat{2\alpha+\beta}$ is contained in $\Gamma$ thanks to Lemma \ref{lem:closedsetktoinfty}. Therefore we are reduced to the cases that have been treated above.
\end{proof}

\begin{lemma}\label{othogoalboundedlemma}
Let $u,v\in \widetilde{W}$ be such that $u,v$ do not have an upper bound in
$\widetilde{W}$ and $u\wedge v=e$.  Then there exists a positive system $\Psi^+$ such that $\widehat{\Psi^+}\supset B$ for any biclosed set $B$ with $B\cap (\Phi_u\cup \Phi_v)=\emptyset.$
\end{lemma}

\begin{proof}
By Lemma \ref{lem:asetzclosed} $A:=A_{\overline{\Phi_u\cup
\Phi_v}}$ is $\mathbb{Z}-$closed. If $A\cap -A=\emptyset$, by
Lemma \ref{lem:zclosedlemma}, $A\subset \Omega^+$ where $\Omega^+$
is a positive system in $\Phi$. This implies that $\overline{\Phi_u\cup
\Phi_v}\subset \widehat{\Omega^+}$. So $u$ and $v$ are both less
than or equal to an infinite reduced word by Theorem \ref{thm:bijection}, and
consequently less than or equal to a word in $\widetilde{W}$, contradicting
the assumption. This shows that there exists $\alpha\in \Phi$ such
that $\alpha$ and $-\alpha$ are both in $A$. By the closedness of
$\overline{\Phi_u\cup \Phi_v}$ and the calculation:
$$2(\alpha+m\delta)+(-\alpha+n\delta)=\alpha+(2m+n)\delta,$$
$$3(\alpha+m\delta)+2(-\alpha+n\delta)=\alpha+(3m+2n)\delta,$$
$$\dots,$$
 an infinite $\delta$ chain
through $\alpha$ and an infinite $\delta$ chain through $-\alpha$
are both in $\overline{\Phi_u\cup \Phi_v}$, i.e.
$\alpha,-\alpha\in I_{\overline{\Phi_u\cup \Phi_v}}$.

Now we claim that there must be at least one root $\beta$ such that $\beta\neq \alpha$ and
$\beta\neq -\alpha$ and $\beta\in A_{\Phi_u\cup \Phi_v}$. If not, $A_{\Phi_u\cup \Phi_v}\subset \{\alpha,-\alpha\}$.
The containment cannot be proper since the closure of a subset of $\widehat{\alpha}$ (resp. $\widehat{-\alpha}$) is again contained in $\widehat{\alpha}$ (resp. $\widehat{-\alpha}$). Hence $A_{\Phi_u\cup \Phi_v}=\{\alpha,-\alpha\}$.
We note that $A_{\Phi_u}$ (resp. $A_{\Phi_v}$) cannot be $\{-\alpha,\alpha\}$ because of Lemma \ref{closedaffine}.
Therefore we conclude that one of $A_{\Phi_u}$ and $A_{\Phi_v}$ is $\alpha$ and the other is $-\alpha$. But that is not possible because either none of the roots in $\widehat{-\alpha}$ is simple in $\widehat{\Phi}$ or none of the roots in $\widehat{\alpha}$ is simple in $\widehat{\Phi}$. Therefore there exists $\beta\in A_{\Phi_u\cup \Phi_v}, \beta\neq \pm\alpha$. By Lemma \ref{lem:startfrombottom}, $\beta_0\in \Phi_u\cup \Phi_v.$

Then $\alpha,-\alpha,\beta$ determines a closed half-plane $H$ with 0 in its boundary in $V=\mathbb{R}^2$ (i.e. take a vector $x$ in $\mathbb{R}^2$ such that $x\perp \alpha$ and $(x,\beta)>0$ and the closed half-plane is $\{y|(x,y)\geq 0\}$).
We shall show that $\Phi\cap H\subset I_{\overline{\Phi_u\cup
\Phi_v}}$. First we show that there exists a root $\gamma$ other than
$\alpha,-\alpha,\beta$ in this closed half-plane  such that $\gamma+p\delta\in \overline{\Phi_u\cup \Phi_v}$ for  sufficiently large $p$.

To do this first we assume that $\beta\in \Phi^+$. In this case $\beta\in \Phi_u\cup \Phi_v.$ Then there must exist some $\gamma$ other than
$\alpha,-\alpha,\beta$ in this closed half-plane (this is clear by
inspecting the graphs of rank 2 irreducible crystallographic root
systems).
Without loss of generality we assume that
$\gamma=m\alpha+n\beta$ where $m,n\in \mathbb{Q}_{>0}$. For
sufficiently large $k$, $\alpha+k\delta\in \overline{\Phi_u\cup
\Phi_v}$ thanks to Lemma \ref{lem:closedsetktoinfty}.  Then there are infinitely many $k$ such that
$\alpha+k\delta\in \overline{\Phi_u\cup \Phi_v}$ and $km\in
\mathbb{Z}_{>0}$. Therefore for such $k$, $\gamma+mk\delta=m(\alpha+k\delta)+n\beta.$ This shows that $\gamma\in I_{\overline{\Phi_u\cup
\Phi_v}}.$ By Lemma \ref{lem:closedsetktoinfty}, $\gamma+p\delta\in \overline{\Phi_u\cup \Phi_v}$ for sufficiently large $p$.

Now we assume that $\beta\in \Phi^-.$ Then $\beta+\delta\in \Phi_u\cup \Phi_v.$ We first treat the case where $(\beta,\alpha)\neq 0$. Then either $(\beta,\alpha)<0$ or $(\beta,-\alpha)<0$. Therefore either $\beta+\alpha$ is a root or $\beta-\alpha$ is a root by Chapter VI \S 1.3 Theorem 1 of \cite{Bourbaki}. Without loss of generality we assume that the former is the case. Let $\gamma=\beta+\alpha.$ For
sufficiently large $k$, $\alpha+k\delta\in \overline{\Phi_u\cup
\Phi_v}$ thanks to Lemma \ref{lem:closedsetktoinfty}. Then $\gamma+(k+1)\delta=\beta+\delta+(\alpha+k\delta)$. So we are done with this case. Now we treat the case where $(\beta,\alpha)=0$.
We show that we can always find some $\gamma$ such that $\gamma=\alpha'+\beta$ or $\gamma=\frac12\alpha'+\frac12\beta$ where $\alpha'$ can be $\alpha$ or $-\alpha$.
In type $A_2$ such a situation never happens. Suppose that $\Phi$ is of type $B_2.$
Denote the short simple root by $\alpha_1$ and the long simple root by $\alpha_2.$
If $\{\alpha,\beta,-\alpha\}=\{\alpha_1+\alpha_2,-\alpha_1,-\alpha_1-\alpha_2\}$, let $\gamma=\alpha_2=-\alpha_1+(\alpha_1+\alpha_2)$. If $\{\alpha,\beta,-\alpha\}=\{\alpha_2,-2\alpha_1-\alpha_2,-\alpha_2\}$, let $\gamma=-\alpha_1=\frac{1}{2}(-2\alpha_1-\alpha_2)+\frac12\alpha_2.$
If $\{\alpha,\beta,-\alpha\}=\{\alpha_1,-\alpha_1-\alpha_2,-\alpha_1\}$,
let $\gamma=-\alpha_2=-\alpha_1-\alpha_2+\alpha_1$. If $\{\alpha,\beta,-\alpha\}=\{2\alpha_1+\alpha_2,-\alpha_2,-2\alpha_1-\alpha_2\}$, let
$\gamma=\alpha_1=\frac12(-\alpha_2)+\frac12(2\alpha_1+\alpha_2)$. Suppose that $\Phi$ is of type $G_2$.
Denote the short simple root by $\alpha_1$ and the long simple root by $\alpha_2.$ If $\{\alpha,\beta,-\alpha\}=\{3\alpha_1+2\alpha_2, -\alpha_1,-3\alpha_1-2\alpha_2\}$, $\gamma$ can be chosen to be $\alpha_1+\alpha_2=\frac12(-\alpha_1)+\frac12(3\alpha_1+2\alpha_2)$. If $\{\alpha,\beta,-\alpha\}=\{\alpha_1+\alpha_2, -3\alpha_1-\alpha_2,-\alpha_1-\alpha_2\}$, $\gamma$ can be chosen to be $-\alpha_1=\frac12(-3\alpha_1-\alpha_2)+\frac12(\alpha_1+\alpha_2)$. If $\{\alpha,\beta,-\alpha\}=\{\alpha_2, -2\alpha_1-\alpha_2,-\alpha_2\}$, $\gamma$ can be chosen to be $-\alpha_1=\frac12(-2\alpha_1-\alpha_2)+\frac12(\alpha_2)$.  If $\{\alpha,\beta,-\alpha\}=\{\alpha_1, -3\alpha_1-2\alpha_2,-\alpha_1\}$, $\gamma$ can be chosen to be $-\alpha_1-\alpha_2=\frac12(-3\alpha_1-2\alpha_2)+\frac12(\alpha_1)$.
If $\{\alpha,\beta,-\alpha\}=\{3\alpha_1+\alpha_2, -\alpha_1-\alpha_2,-3\alpha_1-\alpha_2\}$, $\gamma$ can be chosen to be $\alpha_1=\frac12(3\alpha_1+\alpha_2)+\frac12(-\alpha_1-\alpha_2)$.
If $\{\alpha,\beta,-\alpha\}=\{2\alpha_1+\alpha_2, -\alpha_2,-2\alpha_1-\alpha_2\}$, $\gamma$ can be chosen to be $\alpha_1=\frac12(2\alpha_1+\alpha_2)+\frac12(-\alpha_2)$. Now assume that $\gamma=\alpha'+\beta$. Since for sufficiently large $k$, $\alpha'+k\delta\in \overline{\Phi_u\cup
\Phi_v}$ thanks to Lemma \ref{lem:closedsetktoinfty}. One has that $\gamma+(k+1)\delta=\beta+\delta+(\alpha'+k\delta)$. Assume that
$\gamma=\frac12\alpha'+\frac12\beta$. Again for sufficiently large $k$, $\alpha'+k\delta\in \overline{\Phi_u\cup
\Phi_v}$. Take $k$ to be odd. Then $\gamma+\frac{k+1}{2}\delta=\frac12(\beta+\delta)+\frac12(\alpha'+k\delta)$.
By Lemma \ref{lem:closedsetktoinfty}, $\gamma+p\delta\in \overline{\Phi_u\cup \Phi_v}$ for sufficiently large $p$.

Therefore we have proved the existence of such $\gamma.$
In any case such $\gamma$ equals $m\alpha+n\beta$ or $m(-\alpha)+n\beta$ where $m,n\in \mathbb{Q}_{>0}$. Without loss of generality assume that the former is the case.
So $\beta=\frac{1}{n}\gamma+\frac{m}{n}(-\alpha)$. There are infinitely many $k,k'$ such that
$-\alpha+k\delta, \gamma+k'\delta\in \overline{\Phi_u\cup \Phi_v}$ and $k\frac{m}{n},\frac{k'}{n}\in
\mathbb{Z}_{>0}$. Therefore for such $k,k'$, $\beta+(k\frac{m}{n}+\frac{k'}{n})\delta=\frac{1}{n}(\gamma+k'\delta)+\frac{m}{n}(-\alpha+k\delta).$ This shows that $\beta\in I_{\overline{\Phi_u\cup
\Phi_v}}.$ By Lemma \ref{lem:closedsetktoinfty}, $\beta+p\delta\in \overline{\Phi_u\cup \Phi_v}$ for sufficiently large $p$.
Now applying the same
reasoning to other roots (if they exist) in this closed half plane, we see that for all roots $\gamma$ contained in this closed half-plane $\gamma\in I_{\overline{\Phi_u\cup
\Phi_v}}$.

Now there are two possibilities: $A=\Phi\cap H$ and $A\varsupsetneqq (\Phi\cap H)$. If $A=\Phi\cap H$, then the previous paragraph shows that
$A=A_{\overline{\Phi_u\cup \Phi_v}}=I_{\overline{\Phi_u\cup \Phi_v}}=\Phi\cap H$. If $A\varsupsetneqq (\Phi\cap H)$, applying the the same reasoning as in the previous paragraph to both $H$ and the closed half plane with 0 in its boundary determined by $\alpha,-\alpha$ and a root in $A\backslash(\Phi\cap H)$ shows that $A=A_{\overline{\Phi_u\cup \Phi_v}}=I_{\overline{\Phi_u\cup \Phi_v}}=\Phi.$

Now we treat these two cases separately:

Case (I). Assume that $A=A_{\overline{\Phi_u\cup \Phi_v}}=I_{\overline{\Phi_u\cup \Phi_v}}=\Phi\cap H$.
One notes that $\Phi\cap H=\Psi^+_{\emptyset,\{\alpha\}}$ for some positive system $\Psi^+$ and a simple root $\alpha$ in $\Psi^+$. (To see this, suppose that $H=\{y|(x,y)\geq 0\}$ and that $wx$ is in the closure of fundamental chamber for some $w\in W$. Then $\Psi^+=w^{-1}\Phi^+$ where $\Phi^+$ is the standard positive system.)
Therefore $\overline{\Phi_u\cup
\Phi_v}=\widehat{\Psi^+_{\emptyset,\{\alpha\}}}$ by  Lemma \ref{lem:orthocaseone}.
By assumption $B\cap \Phi_u=\emptyset$ and $B\cap
\Phi_v=\emptyset$. Then $B\cap \overline{\Phi_u\cup
\Phi_v}=\emptyset$ since $B$ is coclosed. So we conclude that $B\subset
\widehat{\Psi^+_{\emptyset,\{\alpha\}}}^c=\widehat{(-\Psi^+)_{\{-\alpha\},\emptyset}}$.

Case (II). Assume that $A=A_{\overline{\Phi_u\cup \Phi_v}}=I_{\overline{\Phi_u\cup \Phi_v}}=\Phi.$ Then
Lemma \ref{lem:orthocasetwo} asserts that $\overline{\Phi_u\cup
\Phi_y}\supset \widehat{\Psi^+_{\emptyset,\{\alpha\}}}$ for some
positive system $\Psi^+$ in $\Phi$ and some simple root $\alpha$ of
$\Psi^+$. By assumption $B\cap \Phi_u=\emptyset$ and $B\cap
\Phi_v=\emptyset$. Then $B\cap \overline{\Phi_u\cup
\Phi_v}=\emptyset$ since $B$ is coclosed. So we conclude that $B\subset
\widehat{\Psi^+_{\emptyset,\{\alpha\}}}^c=\widehat{(-\Psi^+)_{\{-\alpha\},\emptyset}}$.
\end{proof}

\begin{corollary}\label{cor:orthofinal}
Suppose that $u,v\in \overline{\widetilde{W}}$ do not admit a join in
$\overline{\widetilde{W}}$. The set $\{x\in \overline{\widetilde{W}}|x\perp u, x\perp v\}$
admits a join in $\overline{\widetilde{W}}$.
\end{corollary}

\begin{proof}
First we assume that $u,v\in \widetilde{W}$. If $u,v$ have trivial meet
then the assertion follows from Lemma \ref{othogoalboundedlemma},
Corollary \ref{cor:biclosedsetofinflongword} and the fact that
$\overline{\widetilde{W}}$ is a complete meet semilattice. Now
suppose that $u\wedge v=w\neq e$. Then by the previous discussion the
set $\{x|x\perp w^{-1}u, x\perp w^{-1}v\}$ admits a join $y$ in
$\overline{\widetilde{W}}$. So $w^{-1}\leq y.$ Then we claim that
$wy$ is the join of $\{x|x\perp u,x\perp v\}$. We show
that for $p\in \{x|x\perp u,x\perp v\}$, $w^{-1}p\in \{x|x\perp
w^{-1}u, x\perp w^{-1}v\}$ and $l(w^{-1}p)=l(w^{-1})+l(p)$:
$$l(p^{-1}w)+l(w^{-1}u)\geq l(p^{-1}ww^{-1}u)=l(p^{-1}u)=l(p^{-1})+l(u)$$
$$=l(p^{-1})+l(ww^{-1}u)=l(p^{-1})+l(w)+l(w^{-1}u)\geq l(p^{-1}w)+l(w^{-1}u)$$
(and the same for $v$ in place of $u$).

 So
$w^{-1}p\leq y$. Since $w^{-1}\leq w^{-1}p$ and $w^{-1}\leq y$, we
can multiply by $w$ on both sides of the inequality $w^{-1}p\leq y$ and
have $p\leq wy$ (See Proposition 3.12(vi) of \cite{bjornerbrenti}).
By Theorem \ref{theorem:jop}, $y\perp w^{-1}u$ and
$y\perp w^{-1}v$. Therefore $wy\perp u$ and $wy\perp v$:
$$l(y^{-1}w^{-1})+l(u)\geq l(y^{-1}w^{-1}u)=l(y^{-1})+l(w^{-1}u)$$$$=l(y^{-1})-l(w)+l(u)=l(y^{-1}w^{-1})+l(u)$$
(and the same for $v$ in place of $u$).

 This shows that
$wy$ is the join. Now if at least one of $u$ and $v$ is
an infinite word, then some prefix $x$ of $u$ and some prefix $y$ of
$v$ do not have bound in $\overline{\widetilde{W}}$. Also one notes
that $\{p|p\perp u,p\perp v\}\subset \{p|p\perp x,p\perp y\}$ and
thus the former is bounded in $\overline{\widetilde{W}}$. Then the existence of join follows from the fact $\overline{\widetilde{W}}$ is a complete meet semilattice.
\end{proof}

Now we have made all the preparations for proving that
$\mathscr{B}(\widetilde{\Phi}^+)$ is a lattice.

\begin{proposition}\label{prop:lattice}
Let $x,y\in \overline{\widetilde{W}}$,

(1) Suppose that $x\vee y$ exists in $\overline{\widetilde{W}}$. Then
the join of $\Phi_x$ and $\Phi_y$ exists in
$\mathscr{B}(\widetilde{\Phi}^+)$ and is given by $\Phi_{x\vee
y}$.

(2) Suppose that $x$ and $y$ have no upper bound in
$\overline{\widetilde{W}}$. Let $M=\{w|w\in
\overline{\widetilde{W}}, w\perp x, w\perp y\}$. Then the join of
$M$ exists in $\overline{\widetilde{W}}$ by Corollary \ref{cor:orthofinal} (denoted by $z$). Then the
join of $\Phi_x$ and $\Phi_y$ is given by $\Phi_z'$ in
$\mathscr{B}(\widetilde{\Phi}^+)$.

(3) The join of $\Phi_x'$ and $\Phi_y'$ exists in
$\mathscr{B}(\widetilde{\Phi}^+)$ and is given by $\Phi_{x\wedge
y}'$.

(4) The join of $\Phi_x$ and $\Phi_y'$ exists in
$\mathscr{B}(\widetilde{\Phi}^+)$ and is given by $\Phi_{z}'$ where
$z=\bigvee\{w|w\in \overline{\widetilde{W}}, w\perp x, w\leq y\}$.
\end{proposition}

\begin{proof}
One first notes that the biclosed sets  in this (rank 3) case are
either $\Phi_x$ or $\Phi_x'$ where $x\in \overline{\widetilde{W}}$
by Theorem \ref{classifybiclosedsetaffine}, Corollary \ref{cor:biclosedsetofinflongword} and Lemma
\ref{lem:infinitelongbasic}(d).

(1) If for some $w\in \overline{\widetilde{W}}$,
$\Phi_x,\Phi_y\subset \Phi_{w}'$, by Join Orthogonality Property
(Theorem \ref{theorem:jop}) $\Phi_w'\supset \Phi_{x\vee y}$.

(2) If for some $w\in \overline{\widetilde{W}}$,
$\Phi_x,\Phi_y\subset \Phi_{w}'$, then this is equivalent to $w\perp
x, w\perp y$. $w\leq z$. So $\Phi_w'\supset \Phi_z'.$

(3) By Lemma \ref{lem:upisnotinbottom}, no $\Phi_z$ can properly
contain $\Phi_x'$ and $\Phi_y'$. If for some $w\in
\overline{\widetilde{W}}$, $\Phi_x',\Phi_y'\subset \Phi_{w}'$, then
$x\geq w$ and $y\geq w$ and thus $x\wedge y\geq w$. Therefore
$\Phi_{x\wedge y}'\subset \Phi_w'.$

(4) Note that $e\in \{w|w\in \overline{\widetilde{W}}, w\perp x, w\leq y\}$
and therefore $\{w|w\in \overline{\widetilde{W}}, w\perp x, w\leq
y\}$ is not empty. The existence of $z$ is guaranteed by the fact
that $\overline{\widetilde{W}}$ is a complete meet semilattice and
the set $\{w|w\in \overline{\widetilde{W}}, w\perp x, w\leq y\}$ is
bounded by $y$. By Join Orthogonality Property (Theorem
\ref{theorem:jop}), $z\perp x$. Hence $\Phi_x\subset \Phi_z'.$ It is
clear that $z\leq y$. Hence $\Phi_z\subset \Phi_y$ and thus
$\Phi_z'\supset \Phi_y'$. If there is $\Phi_u'$ such that
$\Phi_u'\supset \Phi_x$ and $\Phi_u'\supset \Phi_y'$, then it
follows that $u\in \{w|w\in \overline{\widetilde{W}}, w\perp x, w\leq
y\}$. So $u\leq z$ and $\Phi_u'\supset \Phi_z'.$ If $\Phi_w\supset
\Phi_x$ and $\Phi_w\supset \Phi_y'$, then by Lemma
\ref{lem:upisnotinbottom} $\Phi_w=\Phi_y'$ and thus $\Phi_y'\supset
\Phi_x.$ Therefore $x\perp y$. In this case $z=\bigvee\{w|w\in
\overline{\widetilde{W}}, w\perp x, w\leq y\}=y$.
$\Phi_w=\Phi_y'=\Phi_z'$.
\end{proof}

\begin{lemma}\label{chainunionintersection}
Let $\{B_i\}_{i\in I}$ be a chain of biclosed sets in the set of positive roots. Then their union and intersection are both biclosed.
\end{lemma}

\begin{proof}
Let $B=\cup_{i\in I}B_i.$ Take $\alpha,\beta\in B$. Then $\alpha\in B_j,\beta\in B_k, j,k\in I$.
Without loss of generality assume that $B_j\subset B_k$. Then $\alpha,\beta\in B_k$. Since $B_k$ is closed, any root  of the form $k_1\alpha+k_2\beta, k_1, k_2>0$ must be contained in $B_k$.
Let $\alpha,\beta$ be two positive roots not contained $B.$ If a root  $\gamma=k_1\alpha+k_2\beta, k_1, k_2>0$ is contained in $B$, it must be contained in some $B_l, l\in I$. However $\alpha,\beta\not\in B_l$. This contradicts the coclosedness of $B_l$. Therefore $B$ is biclosed. Similarly one can prove the assertion for $\cap_iB_i.$
\end{proof}

\begin{theorem}\label{mainthmlattice}
The poset $\mathscr{B}(\widetilde{\Phi}^+)$ is a complete ortholattice.
\end{theorem}

\begin{proof}
Proposition \ref{prop:lattice} guarantees that
$\mathscr{B}(\widetilde{\Phi}^+)$ is a lattice. By defining
$B^{oc}=\widetilde{\Phi}^+\backslash B$ one easily sees that
$\mathscr{B}(\widetilde{\Phi}^+)$ is an ortholattice. Now let
$A\subset \mathscr{B}(\widetilde{\Phi}^+)$. Let $C$ be the set of
biclosed sets bounded by all elements in $A$. Since the union of a
chain of biclosed sets is a biclosed set by Lemma \ref{chainunionintersection}, the condition of Zorn's
Lemma is satisfied. The set $C$ has at least one maximal element. If the
maximal element is not unique, take two of them and their join again
is bounded by all elements in $A$, which is a contradiction. This
shows that the maximal element is unique and $A$ admits a meet. Hence
$\mathscr{B}(\widetilde{\Phi}^+)$ is a complete meet semilattice
with a unique maximal element(given by $\widetilde{\Phi}^+$). Therefore it is
a complete lattice.
\end{proof}

\begin{remark*}
In \cite{labbe} Section 2.4 the author replaced the closure used in this paper with convex closure and proved the lattice property for biconvex sets in the positive system of rank three Coxeter groups by studying the distribution of the roots in the space and  using the projective representation.
\end{remark*}

\begin{remark*}
Call a biclosed set $\Gamma$ in $\widetilde{\Phi}$
a \emph{quasi-positive system} (or a \emph{2 closure hemispace}) if $\Gamma\uplus
-\Gamma=\widetilde{\Phi}$. One might ask whether
$\mathscr{B}(\Gamma)$ (all the biclosed sets in $\Gamma$ under
inclusion) is a complete lattice. This is conjectured in the first version of \cite{DyerWeakOrder} (Conjecture 2.5(a)). We show that in general this is not true as
illustrated by the following example.

Consider the root system of $\widetilde{A}_2$. Let $\alpha,\beta$ be
two simple roots of a chosen positive system of $\Phi$, the root
system of the corresponding Weyl group.

The set $\Psi=\{-\alpha+l\delta, -\beta+m\delta,
-\alpha-\beta+n\delta|l,m,n\in \mathbb{Z}\}$ is a quasi-positive
system.

The following four sets are all biclosed in $\Psi$: $B_1=\{-\alpha+a\delta, -\beta+b\delta, -\alpha-\beta+c\delta|a,c\in
\mathbb{Z}, b\in \mathbb{Z}_{\leq 0}\}, B_2=\{-\alpha+d\delta,
-\beta+e\delta, -\alpha-\beta+f\delta|e,f\in \mathbb{Z}, d\in
\mathbb{Z}_{\geq 0}\}, B_3=\{-\alpha+g\delta|g\in \mathbb{Z}_{\geq
0}\}$  and $B_4=\{-\beta+h\delta|h\in \mathbb{Z}_{\leq 0}\}$.

The sets $B_3$ and $B_4$ are both contained in $B_1\cap B_2$. If $B_1$ and
$B_2$ admit a meet, say $C$, we have that $\overline{B_3\cup
B_4}\subset C\subset B_1\cap B_2$. But the first and third term
coincide, which is $\{-\alpha+i\delta, -\beta+j\delta,
-\alpha-\beta+k\delta|i\in \mathbb{Z}_{\geq 0}, j\in
\mathbb{Z}_{\leq 0}, k\in \mathbb{Z}\}$. This forces $C$ to be this
set. However this set is not biclosed in $\Psi$.
\end{remark*}

\section{Compact Biclosed Sets And Algebraic Lattice}\label{sec:compact}

For the first part of this section $\widetilde{W}$ is a general (irreducible)
affine Weyl group without restriction on the rank.

A \emph{compact biclosed set} in $\widetilde{\Phi}^+$ is meant to be a
compact element in $\mathscr{B}(\widetilde{\Phi}^+)$ provided that
$\mathscr{B}(\widetilde{\Phi}^+)$ is a lattice. A \emph{finitely generated
biclosed set} is a biclosed set which is equal to the closure of finitely many roots.

\begin{lemma}\label{fingencompt}
If $\mathscr{B}(\widetilde{\Phi}^+)$ is a lattice, a biclosed set in
$\widetilde{\Phi}^+$ is compact if it is finitely generated.
\end{lemma}

\begin{proof}
According to the classification of the biclosed sets in an affine positive system, $\mathscr{B}(\widetilde{\Phi}^+)$ is an infinite countable set.
Let $B=\overline{\{\alpha_1,\alpha_2,\dots,\alpha_n\}}$ be biclosed in
$\widetilde{\Phi}^+$. Suppose that $B\subset \bigvee_{i=1}^{\infty}B_i$. So
$\alpha_j\in \bigvee_{i=1}^{\infty}B_i, 1\leq j\leq n$. This forces that
$\alpha_j\in \bigvee_{i=1}^{t_j}B_{i}, 1\leq j\leq n$ for some $t_j$. Take $t=\max\{t_j|1\leq j\leq n\}$ and one sees that $B\subset\bigvee_{i=1}^{t}B_{i}$.
\end{proof}

Let $\Psi^+$ be a positive system in $\Phi$ and $\Pi$ be its simple
system. Let $M$ be a subset of $\Pi$.

\begin{theorem}\label{thm:classifycompact}
Given $L=\{\alpha\in \Pi|(\alpha, M)=0\}$ for $M\neq \emptyset$ (resp. $L=\Pi$ for $M=\emptyset$), a biclosed set $\Gamma$
with $I_{\Gamma}=\Psi_{L,M}^+$ is finitely generated. Given
$L\subsetneq\{\alpha\in \Pi|(\alpha, M)=0\}$ for $M\neq \emptyset$ (resp. $L\subsetneq\Pi$ for $M=\emptyset$), a biclosed set
$\Gamma$ with $I_{\Gamma}=\Psi_{L,M}^+$ is not finitely generated.
\end{theorem}

\begin{proof}
A biclosed set $\Gamma$ with $I_{\Gamma}=\Psi_{L,M}^+$ consists of the
following three disjoint subsets:

(1) For each $\alpha_i\in \Psi_L$, there
may be a $\delta$-string through $\alpha_i$ of finite length
beginning from $(\alpha_i)_0$. These form a finite set, denoted by
$\Gamma_L$.

(2) The set $\widehat{\Psi_{L\cup M,\emptyset}^+}$.

(3) For each $\gamma_i\in \Psi_M$, there is a
$\delta$-string through $\gamma_i$ of infinite length beginning from
$\gamma_i+n(\gamma_i)\delta$.

To see (1), (2) and (3), one notes that by Theorem \ref{classifybiclosedsetaffine} (2), $\Gamma=w\cdot \widehat{\Psi_{L,M}^+}, w\in W'$ where $W'$ is the reflection subgroup generated by $\{s_{\alpha}|\alpha\in \widehat{\Psi_{M\cup L}}\}$. In particular, we have the formula $w\cdot \widehat{\Psi^+_{L,M}}=\widehat{\Psi^+_{L,M}}+(\Phi_{W'}^+\cap w(-\Phi_{W'}^+))$ where $+$ denotes the set symmetric difference and $\Phi_{W'}^+\subset \widehat{\Phi}$ is the positive system of $W'$. The set $(\Phi_{W'}^+\cap w(-\Phi_{W'}^+))$ is finite and consists of the positive roots in $\Phi_{W'}^+$ made negative by $w^{-1}$. Then the description of (1) (2) and (3) follows from Lemma \ref{lem:startfrombottom} and the fact that $\Phi_{W'}^+=\widehat{\Psi_{M\cup L}}$. (Note that the elements in (1) are in $(\Phi_{W'}^+\cap w(-\Phi_{W'}^+))\backslash \widehat{\Psi^+_{L,M}}$ and the elements in (2) and (3) are in $\widehat{\Psi^+_{L,M}}\backslash (\Phi_{W'}^+\cap w(-\Phi_{W'}^+))$.)

The set $w\cdot \widehat{\Psi^+_{\Pi,\emptyset}}$ is finite and is therefore finitely generated.

Now suppose that $M\neq \emptyset$ and $L=\{\alpha\in \Pi|(\alpha, M)=0\}$. We show that
$$\Gamma_L\cup \{\alpha_0|\alpha\in\Psi_{L\bigcup
M,\emptyset}^+\}\cup \{\gamma_i+n(\gamma_i)\delta|\gamma_i\in
\Psi_M\}$$ is a finite generating set of $\Gamma.$

Since part (1) has been included in the set so we only need to show that the elements in (2) and (3) can be generated by this set.

We first show that part (3) can be generated.

Take $\gamma,-\gamma\in \Psi_M$. By assumption
$\gamma+n(\gamma)\delta,-\gamma+n(-\gamma)\delta$ are in the
generating set.
Then

$$\frac{n(\gamma)+n(-\gamma)+1}{n(\gamma)+n(-\gamma)}(\gamma+n(\gamma)\delta)+\frac{1}{n(\gamma)+n(-\gamma)}(-\gamma+n(-\gamma)\delta)=\gamma+(n(\gamma)+1)\delta.$$
Proceeding in this way one sees that all elements in part (3) can be generated.

Now we check that part (2) can be generated.

Take $\alpha\in \Psi^+\backslash \Psi_{L\cup
M}$. We can define $h$ function as
in Section \ref{sec:bijection}, i.e. $h(\alpha):=\sum_{\alpha_i\not\in L}k_i$ if $\alpha=\sum_{\alpha_i\in \Pi}k_i\alpha_i$. By Proposition \ref{prop:twofunctionequal} the root
$\alpha$ can be written as a sum of $h(\alpha)$ roots in
$\Psi^+\backslash \Psi_L$. We first show that
for $\alpha$ such that $h(\alpha)=1$, we have that $\widehat{\alpha}$ can
be generated. Suppose that $\alpha=\alpha_1+\alpha_2+\cdots+\alpha_p$
with $\alpha_i\in \Pi$ and $\alpha_j\in \Pi\backslash L$ for one $j$
 and the rest $\alpha_i$'s are in $L$. In fact we have that $\alpha_j\in\Pi\backslash(L\cup M)$. So we can find $\beta\in
M$ such that $(\beta,\alpha_j)<0$. Then
$$s_{\beta}(\alpha)=\alpha-\frac{2(\alpha_j,\beta)}{(\beta,\beta)}\beta\in \Psi^+.$$
We note that
$\alpha-\frac{2(\alpha_j,\beta)}{(\beta,\beta)}\beta$ is still in $\Psi^+\backslash \Psi_{L\cup
M}$ and therefore  $(\alpha-\frac{2(\alpha_j,\beta)}{(\beta,\beta)}\beta)_0$ is in the generating set.
Suppose that $\alpha-\frac{2(\alpha_j,\beta)}{(\beta,\beta)}\beta\in
\Phi^+$. Take $k\geq n(-\beta)$,
$$\alpha-\frac{2(\alpha_j,\beta)}{(\beta,\beta)}\beta+(-\frac{2(\alpha_j,\beta)}{(\beta,\beta)})(-\beta+k\delta)$$
$$=\alpha-k\frac{2(\alpha_j,\beta)}{(\beta,\beta)}\delta.$$
Then $(\alpha)_0,
\alpha-\frac{2n(-\beta)(\alpha_j,\beta)}{(\beta,\beta)}\delta,\alpha-\frac{2(n(-\beta)+1)(\alpha_j,\beta)}{(\beta,\beta)}\delta,
\alpha-\frac{2(n(-\beta)+2)(\alpha_j,\beta)}{(\beta,\beta)}\delta,
\dots$ are all generated.

Otherwise suppose that
$\alpha-\frac{2(\alpha_j,\beta)}{(\beta,\beta)}\beta\in
\Phi^-$. Take $k\geq n(-\beta)$,
$$\alpha-\frac{2(\alpha_j,\beta)}{(\beta,\beta)}\beta+\delta+(-\frac{2(\alpha_j,\beta)}{(\beta,\beta)})(-\beta+k\delta)$$
$$=\alpha-k\frac{2(\alpha_j,\beta)}{(\beta,\beta)}\delta+\delta.$$
Then $(\alpha)_0,
\alpha-\frac{2n(-\beta)(\alpha_j,\beta)}{(\beta,\beta)}\delta+\delta,\alpha-\frac{2(n(-\beta)+1)(\alpha_j,\beta)}{(\beta,\beta)}\delta+\delta,
\alpha-\frac{2(n(-\beta)+2)(\alpha_j,\beta)}{(\beta,\beta)}\delta+\delta,
\dots$ are all generated.

The ``gaps" in those sequences can be filled. Consider
$\lambda+m\delta$ and $\lambda+n\delta$  both in a closed set
with $n>m$, then
$$\lambda+(m+1)\delta=\frac{n-m-1}{n-m}(\lambda+m\delta)+\frac{1}{n-m}(\lambda+n\delta).$$
So we conclude that for $\alpha$ with $h(\alpha)=1$ the set $\widehat{\alpha}$ can be
generated.

For $h(\alpha)=k>1$ we can assume that
$\alpha=\gamma_1+\gamma_2+\cdots+\gamma_k$ with $\gamma_i\in
\Psi^+\backslash \Psi_L$ and
$h(\gamma_i)=1$ and $\gamma_1+\gamma_2+\cdots+\gamma_{k-1}\in
\Psi^+$ by Chapter VI \S 1 Proposition 19 in \cite{Bourbaki}. We have shown either that $\widehat{\gamma_k}$ can be generated or that there is an infinite $\delta$-chain through $\gamma_k$ beginning from $n(\gamma_k)$ can be generated (depending on whether $\gamma_k$ lies in $\Phi_M$).
So we can suppose that there exists $N$ when $t\geq N$, $\gamma_k+t\delta$ can be generated.
Also there must be some $l$ such that $\gamma_1+\gamma_2+\cdots+\gamma_{k-1}+l\delta$ is in the generating set since $\gamma_1+\gamma_2+\cdots+\gamma_{k-1}\in \Psi^+_{L,\emptyset}$. Note that
$$\alpha+(N+l)\delta=(\gamma_1+\gamma_2+\cdots+\gamma_{k-1}+l\delta)+\gamma_k+N\delta.$$
Hence $\alpha_0, \alpha+(l+N)\delta, \alpha+(l+N+1)\delta,\dots$ are all generated.
Again by the previous ``gap-filling" method we see that $\widehat{\alpha}$ can be generated.

 If $M\neq \emptyset$ and $L\subsetneq\{\alpha\in
\Pi|(\alpha, M)=0\}$ (resp. $L\subsetneq \Pi$ and $M=\emptyset$), assume that $\{\alpha\in \Pi\backslash L|(\alpha,
M)=0\}=\{\alpha_1,\alpha_2,\dots,\alpha_k\}\neq \emptyset$ (if $M=\emptyset$, let this set  be $\Pi\backslash L$). Also
assume that $\Gamma$ is generated by a finite set $\Lambda.$ Suppose that
$t_i=\max\{p|\alpha_i+p\delta\in \Lambda\}$. Then we show that
$\alpha_i+(t_i+1)\delta$ cannot be generated by $\Lambda$ (which is a
contradiction): if there are two roots $\beta_1,\beta_2$ in
$\Psi_{\emptyset,M}^+$ and $k_1,k_2\geq 0$ such that
$k_1\beta_1+k_2\beta_2=\alpha_i,$ the only possibility is that
$\beta_1=\beta_2=\alpha_i$ and $k_1+k_2=1.$
\end{proof}

Now let $\widetilde{W}$ be of type $\widetilde{A}_2,
\widetilde{B}_2$ or $\widetilde{G}_2$.

\begin{corollary}
The complete lattice $\mathscr{B}(\widetilde{\Phi}^+)$ is algebraic.
\end{corollary}

\begin{proof}
By Theorem \ref{classifybiclosedsetaffine}, for every biclosed $B$
in $\widetilde{\Phi}^+$, $I_B$ is a biclosed set in $\Phi$. Now the
rank 2 nature of $\Phi$, Corollary \ref{cor:infinitelongform} and
Theorem \ref{thm:classifycompact} ensures that the only non-finitely
generated biclosed sets are $\Phi_w, w\in \widetilde{W}_l$. But they
are the union of an ascending chain of finite biclosed sets. The other biclosed sets in $\widetilde{\Phi}^+$ are finitely generated, and by Lemma \ref{fingencompt} and Theorem \ref{mainthmlattice} they are themselves compact.
\end{proof}

\section{Braid Operation On The Reflection Orders}\label{sec:braid}

First let $(W,S)$ be an arbitrary Coxeter system and $\Phi,\Phi^+$ be its root system and a positive system. If $s$ is a simple reflection, we denote by $\alpha_s$ the corresponding simple root. Dyer made the following conjecture (\cite{DyerPrivate}) on generalizing the braid relations on $W$ to the reflection orders.

Let $R$ be a finite set contained in $\Phi^+$. We construct a finite undirected graph as follows. The vertices are given by pairs $(R,\leq)$, where $\leq$ is a total order on $R$ obtained by restricting some reflection order to $R$.
Note that  there can be different reflection orders $\leq_1,\leq_2$ such that when restricted to $R$ they are identical. In that case $(R,\leq_1),(R,\leq_2)$ are regarded as the same element (vertex).

To define the edges of the graph, we first define a notion of dihedral sub-string of $(R,\leq)$. Suppose that $(R,\leq)$ is given by
$$\beta_1<\beta_2<\cdots<\beta_n.$$
$\beta_{p+1}<\beta_{p+2}<\cdots<\beta_q$ is called a dihedral substring of $(R,\leq)$ if there exists a maximal dihedral reflection subgroup $W'$ such that $R\cap \Phi_{W'}^+=\{\beta_{p+1},\beta_{p+2},\dots,\beta_q\}.$

Now $(R,\leq_1)$ and $(R,\leq_2)$ have an edge connecting them if and only if $(R,\leq_2)$ is obtained by reversing a dihedral substring of $(R,\leq_1)$. We call the graph thus defined the ``braid graph''.

\begin{conjecture}(Dyer)
The braid graph is connected for arbitrary $R$ and $(W,S)$.
\end{conjecture}

For $W$ finite the conjecture is reduced to the fact that any two reduced expressions of an element can be connected by braid moves.

We call $(R,\leq_1,\leq_2,\Phi_w)$  a \emph{finite quadruple} if

(i) $R\subset \Phi^+$ and $w\in W$;

(ii) $\leq_1,\leq_2$ are two reflection orders such that when restricted on $R$, one obtains
$$\beta_1<_1\beta_2<_1\cdots<_1\beta_m, \gamma_1<_2\gamma_2<_2\cdots<_2\gamma_m$$
where $R=\{\beta_1,\beta_2,\dots,\beta_m\}=\{\gamma_1,\gamma_2,\dots,\gamma_m\}$;

(iii) there exists a reduced expression $s_1s_2\cdots s_k$ of $w$ such that
$$\beta_i=s_1s_2\cdots s_{j_{i}-1}(\alpha_{s_{j_{i}}}), j_1<j_2<\cdots<j_m$$
and there exists a reduced expression $r_1r_2\cdots r_k$ of $w$ such that
$$\gamma_i=r_1r_2\cdots r_{l_{i}-1}(\alpha_{r_{l_{i}}}), l_1<l_2<\cdots<l_m.$$

We call $(R,\leq_1,\leq_2,\Phi_w')$  a cofinite quadruple if (i) and (ii) as above hold and

(iii)$'$ there exists a reduced expression $s_1s_2\cdots s_k$ of $w$ such that
$$\beta_{m+1-i}=s_1s_2\cdots s_{j_{i}-1}(\alpha_{s_{j_{i}}}), j_1<j_2<\cdots<j_m$$
and there exists a reduced expression $r_1r_2\cdots r_k$ of $w$ such that
$$\gamma_{m+1-i}=r_1r_2\cdots r_{l_{i}-1}(\alpha_{r_{l_{i}}}), l_1<l_2<\cdots<l_m.$$

Let $\rho$ be a root in $\Phi^+$ and $\leq$ be a reflection order. Define $I_{\leq \rho}:=\{\gamma\in \Phi^+|\gamma\leq \rho\}$.

\begin{lemma}\label{lem:finitecofinitereflorder}
Let $W$ be an affine Weyl group
and  $(R,\leq_1,\leq_2,\Phi_w)$ be a finite quadruple (resp. $(R,\leq_1,\leq_2,\Phi_w')$ be a cofinite quadruple).
Suppose that $I$ is an initial section of  both $\leq_1$ and $\leq_2$ with $I_{\leq_1\beta_m}, I_{\leq_2\gamma_m}, \Phi_w\subset I$ (resp. $I_{\leq_1\beta_1}, I_{\leq_2\gamma_1}, \Phi_w'\supset I$).
Let $R'$ be another subset of $\Phi^+$ such that $R'\subset \Phi^+\backslash I$ (resp. $R'\subset I$) and $\leq_1|_{R'}=\leq_2|_{R'}$.   Then $(R\uplus R',\leq_1)$ and $(R\uplus R',\leq_2)$ are in the same connected component of the braid graph.
\end{lemma}

\begin{proof}
Suppose that $(R,\leq_1,\leq_2,\Phi_w)$ is a finite quadruple. Since $W$ is affine, the reflection orders are in bijection with the maximal totally ordered subsets of $\mathscr{B}(\widehat{\Phi})$ by Theorem \ref{reflectionorderaffine}.

Therefore we can find a reflection order $\leq_1'$ whose initial sections include $$\Phi_{s_1},\Phi_{s_1s_2},\dots,\Phi_{s_1s_2\cdots s_k}=\Phi_w,I.$$ When restricted to $R$, $\leq_1=\leq_1'$. Since $I$ is an initial section of both $\leq_1$ and $\leq_1'$, by \cite{dyerhecke} Remark 2.10 there exists a unique reflection order $\leq_1''$ such that $\leq_1''|_{I}=\leq_1'|_{I}$ and $\leq_1''|_{\Phi^+\backslash I}=\leq_1|_{\Phi^+\backslash I}$. Note that $R'\subset (\Phi^+\backslash I)$ and $I_{\leq_1\beta_m}\subset I$ (so $\alpha<_1\beta$ for any $\alpha\in R, \beta\in R'$). Therefore one has that $(R\uplus R',\leq_1)=(R\uplus R', \leq_1'')$.

Now we apply to $\leq_1''$ the reversals of dihedral strings  corresponding to the braid moves converting the reduced expression $s_1s_2\cdots s_k$ to $r_1r_2\cdots r_k$  and obtain a new reflection order $\leq_2'.$ Each braid move is equivalent to reversing the substring of the positive roots of a (possibly non-standard) maximal dihedral parabolic subgroup.  Therefore one sees that $(R\uplus R', \leq_1'')$ is connected to
$(R\uplus R', \leq_2')$ in the braid graph.

It is clear that $\leq_2'|_R=\leq_2|_R.$ Since $R'\subset\Phi^+\backslash I\subset \Phi^+\backslash \Phi_w$, $\leq_2'|_{R'}=\leq_1''|_{R'}=\leq_1|_{R'}$. Hence $\leq_2'|_{R'}=\leq_2|_{R'}$. Finally for any $\alpha\in R$ and $\beta\in R'$,
one has $\alpha<_2\beta$ and $\alpha<_2'\beta$. So we have that $(R\uplus R', \leq_2')=(R\uplus R', \leq_2)$. Therefore we are done.

The case where $(R,\leq_1,\leq_2,\Phi_w')$ is a cofinite quadruple is proved in the essentially same way and is omitted.
\end{proof}

\begin{lemma}\label{findanelement}
Let $\leq$ be a reflection order and $\gamma_1\leq \gamma_2\leq \dots \leq \gamma_t$ be a finite chain of positive roots under $\leq.$ Suppose that $I_{\leq \gamma_t}=\Phi_w$ for some $w\in \overline{W}$ (resp. $I_{\leq \gamma_1}=\Phi_w'$ for some $w\in \overline{W}$). Then there exists $x\in W$ such that $x\leq w$ and $x$ has a reduced expression $s_1s_2\cdots s_l$ with $\gamma_i=s_1s_2\cdots s_{j_{i}-1}(\alpha_{s_{j_{i}}}), 1\leq i\leq t, j_1<j_2<\dots<j_t$
(resp. $\gamma_{t+1-i}=s_1s_2\cdots s_{j_{i}-1}(\alpha_{s_{j_{i}}}), 1\leq i\leq t, j_1<j_2<\dots<j_t$).
\end{lemma}

\begin{proof}
Let $m$ be maximal such that $I_{\leq\gamma_m}$ is finite. Thus $I_{\leq\gamma_m}$ is equal to $\Phi_v, v\in W$. (In case $I_{\leq\gamma_i}$ are all infinite for $1\leq i\leq t$, let $m=0$ and $v=e.$) We prove the lemma by reverse induction.

Suppose that $m=t$.
Then $I_{\leq\gamma_t}$ is finite biclosed and is equal to $\Phi_w$ for some $w\in W$. We show that $w$ itself can be chosen (i.e. $x=w$).
Let this initial section $I_{\leq\gamma_t}$ (equipped with the reflection order $\leq$) be $\beta_0<\beta_1<\dots<\beta_{l(w)}$ with $\beta_{i_j}=\gamma_j,1\leq j\leq t.$ (In particular $i_t=l(w)$.)
This gives rise to a reduced expression of $s_1s_2\cdots s_{l(w)}$  such that
$$\beta_i=s_1s_2\cdots s_{i-1}(\alpha_{s_i}),1\leq i\leq l(w)$$
and therefore
$$\gamma_i=s_1s_2\cdots s_{j_{i}-1}(\alpha_{s_{j_{i}}}), 1\leq i\leq t, j_1<j_2<\dots<j_t.$$

Now assume that the lemma holds for $m\geq k$. Let $m=k-1$.

Therefore  $I_{\leq\gamma_{k}}=\Phi_u$ where $u$ is an infinite reduced word. Since $\gamma_{k-1}<\gamma_{k}$, $\Phi_v:=I_{\leq\gamma_{k-1}}\subset I_{\leq\gamma_{k}}=\Phi_u$ and $v$ is a finite prefix of $u$ such that $\gamma_{k}\not\in\Phi_v$.
So $u$ has a finite prefix $v'$ such that $v$ is a prefix of $v'$ and $\gamma_{k}\in \Phi_{v'}$. Note that $\gamma_{k+1},\dots,\gamma_{t}\not\in \Phi_{u}$ and therefore they are not contained in $\Phi_{v'}$.
Also note that $\Phi_{v'}\subset \Phi_u$ and therefore $\Phi_{v'}\subset I_{\leq \gamma_i}, i\geq k.$
Then we have a finite chain of biclosed sets:
$$I_{\leq \gamma_1}\subset I_{\leq \gamma_2}\subset \cdots \subset I_{\leq\gamma_{k-1}}(=\Phi_v)\subset \Phi_{v'}\subset I_{\leq \gamma_{k+1}}\subset \cdots \subset I_{\gamma_t}.$$
Extend such a finite chain to a maximal chain of biclosed sets. This maximal chain of biclosed sets gives a reflection order $\leq'$ by Theorem \ref{reflectionorderaffine} (1).
When restricted to $\{\gamma_1,\gamma_2,\dots,\gamma_t\}$, $\leq'$ is the same as $\leq$. However $I_{\leq'\gamma_{k}}$ is finite and $I_{\leq'\gamma_{k+1}}$ is infinite.
(In particular if $m=t-1$, then $\Phi_{v'}\subset \Phi_u=I_{\leq\gamma_t}=\Phi_w$. Therefore $I_{\leq'\gamma_t}\subset \Phi_{v'}\subset \Phi_w$.)
Therefore the assertion follows by induction.

The other dual assertion of the lemma can be proved in the same manner and is omitted.
\end{proof}

Now let $\widetilde{W}$ be a rank 3 affine Weyl group with $\Phi,\widehat{\Phi}$ defined as in the previous sections.

\begin{lemma}\label{lem:distancenot1}
Let $u,v\in \overline{\widetilde{W}}$ and $\Phi_v'\supset \Phi_u$. Then $|\Phi_v'\backslash\Phi_u|\neq 1.$
\end{lemma}

\begin{proof}
The lemma is an easy consequence of Theorem \ref{thm:bijection}. For an infinite reduced word $w$, $I_{\Phi_w}=\Psi^+_{L,\emptyset}$ for some positive system $\Psi^+$ in $\Phi$ and a subset  $L$ of its simple system. In particular, if $\alpha+m\delta\in \Phi_w$, then $\widehat{-\alpha}\cap \Phi_w=\emptyset$.
However if $L\neq \emptyset$, for some $\beta\in \Phi$, the set $\Phi_w'(=\widehat{\Psi^-_{\emptyset, -L}})$ must contain infinite $\delta$ chains through both $\beta$ and $-\beta$. Therefore removing one root from $\Phi_w'$ will not make it an inversion set. If $L=\emptyset$, then $\Phi_w'=\widehat{\Psi^-}$. We show that removing one root from $\Phi_w'$ in this case never creates a coclosed set. Suppose that $\alpha\in \Psi^-$ and the root $\alpha+p\delta$ is removed. Let $B=\widetilde{\Phi}^+\backslash (\widehat{\Psi^-}\backslash \{\alpha+p\delta\})$. Then $\widehat{-\alpha}\cup \{\alpha+p\delta\}\subset B$ and $\widehat{\alpha}\cap B=\{\alpha+p\delta\}$. The computation $(M-1)(-\alpha+q\delta)+M(\alpha+p\delta)=\alpha+(Mp+(M-1)q)\delta, M>0$ shows that $B$ is not closed.
\end{proof}

\begin{lemma}\label{initialsectionstru}
Let $\leq$ be a reflection order on $\widehat{\Phi}$. Then there exists some positive system $\Psi^+$ of $\Phi$ such that $\widehat{\Psi^+}$ is an initial section of $\leq$.
\end{lemma}

\begin{proof}
Under  inclusion the initial sections of $\leq$ form a maximal chain $\{B_j\}_{j\in J}$ of biclosed sets which is also a maximal chain of the poset $(\mathcal{P}(\widehat{\Phi}),\subset)$ (Theorem \ref{reflectionorderaffine} (1),(2)). We show that (regarded as a poset itself) this chain is chain complete (i.e. any subchain of it has a least upper bound and a greatest lower bound in $\{B_j\}_{j\in J}$). Suppose that to the contrary a collection of biclosed sets $\{B_i\}_{i\in I}, I\subset J$ of this chain has no least upper bound in this chain. By Lemma \ref{chainunionintersection}, the union $\cup_i B_i$ is biclosed and thus is not in this chain. (If it is in the chain, it has to be the least upper bound of $\{B_i\}_{i\in I}$.) Take another biclosed set $B$ in this chain. If it contains all $B_i, i\in I$ then it contains $\cup_i B_i.$ If it does not contain some $B_i, i\in I$, it has to be contained in $B_i$ and thus contained in $\cup_iB_i$. Hence we can add $\cup_iB_i$ to $\{B_j\}_{j\in J}$ to create a chain of biclosed sets with one more biclosed set and this contradicts the maximality of this chain. Similarly we can show that any subchain has a greatest lower bound.

Now we consider the subchain $\mathcal{C}=\{B_i\}_{i\in I}$ consisting of those $B_i$ contained in some $\widehat{\Psi^+}$ where  $\Psi^+$ is a positive system in $\Phi.$ (For different $B_i$ in this subchain, such $\widehat{\Psi^+}$ could be different though eventually we will see that a same $\widehat{\Psi^+}$ can be chosen.) Every biclosed set in this subchain is of the form $\Phi_u$ for some $u\in \overline{\widetilde{W}}$ (since they are bounded by some $\widehat{\Psi^+}$) and by the rank 2 nature of $\Phi$, the biclosed sets in $\{B_j\}_{j\in J}$ but not in $\mathcal{C}$ are of the form $\Phi_v'$ for some $v\in \overline{\widetilde{W}}$. Because of Lemma \ref{lem:upisnotinbottom}, take any biclosed set $B_i$ from this subchain $\mathcal{C}$, and any biclosed set $B_k$ from $\{B_j\}_{j\in J}$ but outside  this subchain $\mathcal{C}$, we have that $B_i\subset B_k.$

Take $B=\cup_{i\in I}B_i$. This is a biclosed set by Lemma \ref{chainunionintersection}. Thanks to the chain completeness of $\{B_j\}_{j\in J}$, $B\in \{B_j\}_{j\in J}.$ Denote by $w_{B_i}$ the group element or the infinite word such that $\Phi_{w_{B_i}}=B_i$. Thanks to Lemma \ref{chaininwbar}, $\vee_{i\in I}w_{B_i}$ exists and $\Phi_{\vee_{i\in I}w_{B_i}}=B.$ Therefore by definition $B\in \mathcal{C}$.

Now we consider the chain $\{\widehat{\Phi}\backslash B_j\}_{j\in J\backslash I}$ under containment. Since each $B_j, j\in J\backslash I$ is the complement of an inversion set, such a chain consists of inversion sets. Same argument as in the last paragraph shows that $\cup_{j\in J\backslash I}(\widehat{\Phi}\backslash B_j)=\widehat{\Phi}\backslash(\cap_{j\in J\backslash I}B_j)$ is an inversion set itself. Therefore $B':=\cap_{j\in J\backslash I}B_j$ is the complement of an inversion set. By chain completeness, $B'\in \{B_j\}_{j\in J}$.

Since $B_i\subset B_j$ for any $i\in I$ and $j\in J\backslash I$, we have that $B\subset B'.$ We show that there cannot be a biclosed set $B''$ such that $B\varsubsetneqq B''\varsubsetneqq B'.$ If such $B''$ exists, it is in the chain $\{B_{j}\}_{i\in J}$ by maximality. If it is an inversion set it has to be contained in $B$. If it is not an inversion set it has to contain $B'$. This is a contradiction.

On the other hand one notes that $\{B_{j}\}_{i\in J}$ is also a maximal chain in the poset $(\mathcal{P}(\widehat{\Phi}),\subset)$. Therefore either $B=B'$ or $B=B'\backslash \{\alpha\}$. But $B=\Phi_u$ and $B'=\Phi_v'$ for some $u,v\in \overline{\widetilde{W}}$, the latter is impossible by Lemma \ref{lem:distancenot1}. Hence we have the first situation and $B=B'$ is both an inversion set and the complement of an inversion set. This forces it to be $\widehat{\Psi^+}$ for some positive system $\Psi^+$ in $\Phi.$

Since $B$ is in the chain $\{B_j\}_{j\in J}$ and  every biclosed set in this chain is an initial section of $\leq$, $B=\widehat{\Psi^+}$ is an initial section of $\leq$.
\end{proof}

\begin{theorem}\label{finalmainthm}
For $\widetilde{W}$ and any finite $R\subset \widehat{\Phi}$, the braid graph is connected.
\end{theorem}

\begin{proof}
Now let $(R,\leq_1)$ and $(R,\leq_2)$ be two vertices of the braid graph.  By Lemma \ref{initialsectionstru} we can suppose that $\widehat{\Psi_1^+}$ is an initial section of $\leq_1$ and $\widehat{\Psi_2^+}$ is an initial section of $\leq_2$  where $\Psi_1^+,\Psi^+_2$ are two positive systems of $\Phi$.

We prove the theorem by reverse induction on $|\Psi^+_1\cap\Psi^+_2|$.
First suppose that $\Psi_1^+=\Psi_2^+.$
Let $\Psi^+=\Psi_1^+=\Psi_2^+.$

Let $R\cap \widehat{\Psi^+}=\{\beta_1,\beta_2,\dots,\beta_t\}$ with $\beta_1<_1\beta_2<_1\cdots<_1\beta_t$ (resp. $R\cap \widehat{\Psi^+}=\{\gamma_1,\gamma_2,\dots,\gamma_t\}$ with $\gamma_1<_2\gamma_2<_2\cdots<_2\gamma_t$). Since $\leq_{1,\beta_t}\subset \widehat{\Psi^+}$ (resp. $\leq_{2,\gamma_t}\subset \widehat{\Psi^+}$), it is an inversion set. By Lemma \ref{findanelement} one can find $x\in \widetilde{W}$ (resp. $y\in \widetilde{W}$) such that there exists a reduced expression $s_1s_2\cdots s_l=x$ (resp. $r_1r_2\cdots r_{l'}=y$) and $\beta_i=s_1s_2\cdots s_{j_{i}-1}(\alpha_{s_{j_{i}}}), j_1<j_2<\dots<j_t$ (resp. $\gamma_i=r_1r_2\cdots r_{k_{i}-1}(\alpha_{r_{k_{i}}}), k_1<k_2<\dots<k_t$). Furthermore $\Phi_x\subset \widehat{\Psi^+}$ (resp. $\Phi_y\subset \widehat{\Psi^+}$).

By \cite{dyerhecke} Remark 2.10 there exists a unique reflection order $\leq_3$ on $\widehat{\Phi}$ such that  $\leq_3|_{R\cap\widehat{\Psi^+}}=\leq_2|_{R\cap\widehat{\Psi^+}}$ and $\leq_3|_{R\cap\widehat{\Psi^-}}=\leq_1|_{R\cap\widehat{\Psi^-}}$.
Since $\Phi_{x},\Phi_{y}\subset \widehat{\Psi^+}$, they admit a join $w$ in $\widetilde{W}$. Then one sees that
$(R\cap \widehat{\Psi^+},\leq_1,\leq_3,w)$ is a finite quadruple. Also note that $\leq_1\beta_t, \leq_3\gamma_t, \Phi_w\subset \widehat{\Psi^+}$. By Lemma \ref{lem:finitecofinitereflorder} we see that $(R,\leq_1)$ and $(R,\leq_3)$ are in the same connected component of the braid graph.

Let $R\cap \widehat{\Psi^-}=\{\beta_1',\beta_2',\dots,\beta_q'\}$ with $\beta_1'<_2\beta_2'<_2\cdots<_2\beta_q'$ (resp. $R\cap \widehat{\Psi^-}=\{\gamma_1',\gamma_2',\dots,\gamma_q'\}$ with $\gamma_1'<_3\gamma_2'<_3\cdots<_3\gamma_q'$). Since $\leq_{2,\beta_1'}\supset \widehat{\Psi^+}$ (resp. $\leq_{3,\gamma_1'}\supset \widehat{\Psi^+}$), it is the complement of an inversion set. By Lemma \ref{findanelement} one can find $x'\in \widetilde{W}$ (resp. $y'\in \widetilde{W}$) such that there exists a reduced expression $s_1s_2\cdots s_m=x'$ (resp. $r_1r_2\cdots r_{m'}=y'$) and $\beta_{q+1-i}'=s_1s_2\cdots s_{j_{i}-1}(\alpha_{s_{j_{i}}}), j_1<j_2<\dots<j_q$ (resp. $\gamma_{q+1-i}'=r_1r_2\cdots r_{k_{i}-1}(\alpha_{r_{k_{i}}}), k_1<k_2<\dots<k_q$). Furthermore $\Phi_{x'}\subset \widehat{\Psi^-}$ (resp. $\Phi_{y'}\subset \widehat{\Psi^-}$).

Since $\Phi_{x'},\Phi_{y'}\subset \widehat{\Psi^-}$ by Lemma \ref{findanelement}, they admit a join $w'$ in $\widetilde{W}$. Then
$(R\cap \widehat{\Psi^-},\leq_2,\leq_3,w')$ is a cofinite quadruple. Also note that $\leq_2\beta_1', \leq_3\gamma_1', \Phi_{w'}'\supset \widehat{\Psi^+}$. By Lemma \ref{lem:finitecofinitereflorder}   we see that $(R,\leq_2)$ and $(R,\leq_3)$ are in the same connected component of the braid graph and we are done with this case.

Now suppose that $h=|\Psi_1^+\cap\Psi_2^+|<\frac{|\Phi|}{2}$ and that the theorem holds for $(R,\leq_1)$ and $(R,\leq_2)$ with $|\Psi_1^+\cap\Psi_2^+|>h$.

Assume that $\Psi_1^+=\{\alpha_1,\alpha_2,\dots,\alpha_n\}$ where $\alpha_1<\alpha_2<\dots<\alpha_n$ and $<$ is one of the two reflection orders on $\Psi_1^+$. So $\alpha_1,\alpha_n$ are simple and $\alpha_2=s_{\alpha_1}(\alpha_n),\alpha_3=s_{\alpha_1}s_{\alpha_n}(\alpha_1),\dots$. Assume that $\Psi_2^+=g\Psi_1^+, g\in W$. Without loss of generality assume that $s_{\alpha_1}$ is a left prefix of $g$. (Otherwise we can choose the other reflection order $\alpha_n<\alpha_{n-1}<\dots<\alpha_1$ on $\Psi^+_1$ and the same argument below works.) Assume that $g=\underbrace{s_{\alpha_1}s_{\alpha_n}s_{\alpha_1}\cdots}_{\text{length}\,k}$ is a reduced expression.
Then $\Psi_2^+=\{-\alpha_1,-\alpha_2,\dots,-\alpha_k,\alpha_{k+1},\dots,\alpha_n\}$ for some $1\leq k\leq n$.

Assume that $R\cap \widehat{\Psi_1^+}=\{\beta_1,\beta_2,\dots,\beta_t\}$ with $\beta_1<_1\beta_2<_1\cdots<_1\beta_t.$ By Lemma \ref{findanelement} we can find some $w\in \widetilde{W}$ such that $s_1s_2\cdots s_l$ is a reduced expression of $w$ and
$\beta_i=s_1s_2\cdots s_{j_{i}-1}(\alpha_{s_{j_{i}}}), j_1<j_2<\dots<j_t$. Furthermore $\Phi_w\subset \widehat{\Psi_1^+}$.
Note that $\{\beta_1,\beta_2,\dots,\beta_t\}\backslash \widehat{\alpha_1}\subset \widehat{(\Psi^+_{1})_{\{\alpha_1\},\emptyset}}=\Phi_{u}$ where $u$ is an infinite reduced word.
Hence $\{\beta_1,\beta_2,\dots,\beta_t\}\backslash \widehat{\alpha_1}$ is
contained in some finite prefix $y$ of $u$.
Since $\Phi_w,\Phi_y\subset \widehat{\Psi_1^+}$, the elements $w,y$ admit a join $z\in \widetilde{W}$.
Extend the chain of biclosed sets $\Phi_y\subset \Phi_z\subset \widehat{\Psi_1^+}$ to a maximal chain of biclosed sets.
Such a maximal chain gives rise to
a reflection order $\leq_3$ which has $\widehat{\Psi_1^+},\Phi_z$ and $\Phi_{y}$ as initial sections. (Assume that under $\leq_3$, $\beta_{\sigma(1)}<_3 \beta_{\sigma(2)} <_3 \cdots <_3\beta_{\sigma(t)}$ where $\sigma$ is a permutation of $\{1,2,\dots, t\}$.) By \cite{dyerhecke} Remark 2.10 there exists a unique reflection order $\leq_4$ such that $\leq_{4}|_{\widehat{\Psi_1^+}}=\leq_3|_{\widehat{\Psi_1^+}},$ $\leq_4|_{\widehat{\Psi_1^-}}=\leq_1|_{\widehat{\Psi_1^-}}$. Therefore $(R\cap\widehat{\Psi_1^+},\leq_1,\leq_4,\Phi_z)$ is a finite quadruple. Also note that $\leq_1\beta_t,\leq_4\beta_{\sigma(t)},\Phi_z\subset \widehat{\Psi_1^+}$.  So by Lemma \ref{lem:finitecofinitereflorder}  $(R,\leq_1)$ and $(R,\leq_4)$ are in the same connected component of the braid graph. Comparing $(R,\leq_1)$ and $(R,\leq_4)$,
one sees that $\widehat{\alpha_1}\cap (R\cap \widehat{\Psi_1^+})$ is moved to the end of  $\leq_4|_{R\cap \widehat{\Psi_1^+}}$ (while the total order on $R\cap \widehat{\Psi_1^-}$ remains unchanged and $(\alpha_1)_0<_4(\alpha_1)_0+\delta<_4(\alpha_1)_0+2\delta<_4\dots$).

Now we can apply the above techniques to $R\cap \widehat{\Psi_1^-}$. Assume that $R\cap \widehat{\Psi_1^-}=\{\gamma_1,\gamma_2,\dots,\gamma_q\}$ with $\gamma_q<_1\gamma_{q-1}<_1\cdots <_1\gamma_1$ (also $\gamma_q<_4\gamma_{q-1}<_4\cdots <_4\gamma_1$ since $\leq_4$ coincides with $\leq_1$ on $\widehat{\Psi^-_1}$).
By Lemma \ref{findanelement} can find an element $w'\in \widetilde{W}$ such that $s_1's_2'\cdots s_m'$ is a reduced expression of $w'$ and
$\gamma_i=s_1's_2'\cdots s_{k_i-1}'(\alpha_{s_{k_i}'}),k_1<k_2<\dots <k_q.$
One notes that $\{\gamma_1,\gamma_2,\dots,\gamma_q\}\backslash\widehat{-\alpha_1}\subset \widehat{(\Psi^-_1)_{\{-\alpha_1\},\emptyset}}=\Phi_{u'}$ where $u'$ is an infinite reduced word.
So the finite set $\{\gamma_1,\gamma_2,\dots,\gamma_q\}\backslash\widehat{-\alpha_1}$ is  contained in some $\Phi_{y'}$ where $y'$ is a finite prefix of $u'$. Since $\Phi_{w'},\Phi_{y'}$ are both contained in $\widehat{\Psi^-_1}$, the elements $w',y'$ admit a join $z'$ in $\widetilde{W}$.
Extend the chain of biclosed sets $\widehat{\Psi_1^+}\subset \Phi_{z'}'\subset \Phi_{y'}'$ to a maximal chain of biclosed sets.
 This maximal chain gives rise to a reflection order $\leq_5$ having $\widehat{\Psi^-_1},\Phi_{z'}$ and $\Phi_{y'}$ as final sections. (Assume that under $\leq_5$, $\gamma_{\sigma(q)}<_5 \gamma_{\sigma(q-1)}<_5 \cdots <_5 \gamma_{\sigma(1)}$ where $\sigma$ is a permutation of $\{1,2,\dots,q\}$.) By \cite{dyerhecke} Remark 2.10 there exists a unique reflection order  $\leq_6$ with $\leq_6|_{\widehat{\Psi_1^+}}=\leq_4|_{\widehat{\Psi_1^+}}$ and $\leq_6|_{\widehat{\Psi^-_1}}=\leq_5|_{\widehat{\Psi^-_1}}$. Hence $(R\cap\widehat{\Psi_1^-},\leq_4,\leq_6,\Phi_{z'}')$ is a cofinite quadruple. Also note that $\leq_4\gamma_q, \leq_6\gamma_{\sigma(q)}, \Phi_{z'}'\supset \widehat{\Psi_1^+}$. Therefore by Lemma \ref{lem:finitecofinitereflorder}, $(R,\leq_6)$ and $(R,\leq_4)$ are in the same connected component of the braid graph. Consequently $(R,\leq_6)$ and $(R,\leq_1)$ are in the same connected component of the braid graph. Comparing $(R,\leq_1)$ and $(R,\leq_6)$, one sees that

(1) $\widehat{\alpha_1}\cap (R\cap \widehat{\Psi_1^+})$ is moved to the end of  $\leq_6|_{R\cap \widehat{\Psi_1^+}}$ (while $(\alpha_1)_0<_6(\alpha_1)_0+\delta<_6(\alpha_1)_0+2\delta<_6\dots$);

(2) $\widehat{-\alpha_1}\cap (R\cap \widehat{\Psi^-})$ is moved to the beginning of $\leq_6|_{R\cap \widehat{\Psi_1^-}}$ (while
$\cdots<_6(-\alpha_1)_0+2\delta<_6(-\alpha_1)_0+\delta<_6(-\alpha_1)_0$).

Note that the infinite dihedral reflection subgroup $W'$ generated by $s_{(\alpha_1)_0},s_{(-\alpha_1)_0}$ has the positive system $\widehat{\{\alpha_1,-\alpha_1\}}$ and one of the two reflection orders is given by
$$(\alpha_1)_0<_6(\alpha_1)_0+\delta<_6(\alpha_1)_0+2\delta<_6\cdots<_6(-\alpha_1)_0+2\delta<_6(-\alpha_1)_0+\delta<_6(-\alpha_1)_0. $$

Construct a reflection order $\leq_7$ from $\leq_6$ by simply reversing the above subchain of $\leq_6$. This indeed gives a reflection order as one can verify that the initial sections of $\leq_7$ form a maximal chain of biclosed sets.  Such reflection order $\leq_7$ has an initial section $\widehat{\{-\alpha_1,\alpha_2,\dots,\alpha_n\}}$.
The positive system $\Psi_3:=\{-\alpha_1,\alpha_2,\dots,\alpha_n\}$ of $\Phi$ has simple roots $-\alpha_1,\alpha_2$. One of its reflection order is given by
$\alpha_2<\alpha_3<\dots<\alpha_n<-\alpha_1.$
Now reversing the dihedral substring of $(R,\leq_6)$ corresponding to $W'$, we see that $(R,\leq_6)$ and $(R,\leq_7)$ are connected in the braid graph.

Since $|\Psi_3^+\cap \Psi_2^+|=|\Psi^+_1\cap \Psi^+_2|+1$, by induction $(R,\leq_7)$ and $(R,\leq_2)$ are connected in the braid graph and we are done.
\end{proof}

\begin{example*}
We give an example illustrating the process described in  the above proof. Let $\Phi$ be of type $A_2$ and $\alpha,\beta$ be two simple roots of $\Phi^+.$
Consider the maximal chain of biclosed sets in $\widehat{\Phi}$:
$$\emptyset\subset \Phi_{s_{\alpha}}\subset \Phi_{s_{\alpha}s_{\beta}}\subset \Phi_{s_{\alpha}s_{\beta}s_{\alpha}}\subset \Phi_{s_{\alpha}s_{\beta}s_{\alpha}s_{\delta-\alpha-\beta}}\subset \Phi_{s_{\alpha}s_{\beta}s_{\alpha}s_{\delta-\alpha-\beta}s_{\alpha}}\subset \dots$$
$$\subset\Phi_{s_{\alpha}s_{\beta}s_{\alpha}s_{\delta-\alpha-\beta}s_{\alpha}s_{\beta}s_{\alpha}s_{\delta-\alpha-\beta}\cdots}(=\Phi_{s_{\delta-\alpha-\beta}s_{\alpha}s_{\beta}s_{\alpha}s_{\delta-\alpha-\beta}s_{\alpha}s_{\beta}s_{\alpha}\cdots}')\subset$$
$$\cdots \subset \Phi_{s_{\delta-\alpha-\beta}s_{\alpha}s_{\beta}s_{\alpha}}'\subset\Phi_{s_{\delta-\alpha-\beta}s_{\alpha}s_{\beta}}'\subset\Phi_{s_{\delta-\alpha-\beta}s_{\alpha}}'\subset\Phi_{s_{\delta-\alpha-\beta}}'\subset \widehat{\Phi}.$$
This gives rise to a reflection order $\leq_1$:
$$\alpha<_1\alpha+\beta<_1\beta<_1\alpha+\beta+\delta<_1\alpha+\delta<_1\alpha+\beta+2\delta<_1\beta+\delta<_1\alpha+\beta+3\delta<_1\cdots$$
$$<_1-\beta+2\delta<_1-\alpha-\beta+3\delta<_1-\alpha+\delta<_1-\alpha-\beta+2\delta<_1-\beta+\delta<_1-\alpha-\beta+\delta.$$
Let $\Psi^+_1=\{\alpha,\alpha+\beta,\beta\}$ be a positive system of $\Phi$.
It is useful to note that $\widehat{\Psi_1^+}(=\Phi_{s_{\alpha}s_{\beta}s_{\alpha}s_{\delta-\alpha-\beta}s_{\alpha}s_{\beta}s_{\alpha}s_{\delta-\alpha-\beta}\cdots})$ is an initial section of $\leq_1.$

Now consider another maximal chain of biclosed sets in $\widehat{\Phi}$:
$$\emptyset\subset \Phi_{s_{\beta}}\subset \Phi_{s_{\beta}s_{\alpha}}\subset \Phi_{s_{\beta}s_{\alpha}s_{\delta-\alpha-\beta}}\subset \Phi_{s_{\beta}s_{\alpha}s_{\delta-\alpha-\beta}s_{\alpha}}\subset\cdots$$
$$\subset\Phi_{s_{\beta}s_{\alpha}s_{\delta-\alpha-\beta}s_{\alpha}s_{\beta}s_{\alpha}s_{\delta-\alpha-\beta}s_{\alpha}\cdots}(=\Phi_{s_{\alpha}s_{\delta-\alpha-\beta}s_{\alpha}s_{\beta}s_{\alpha}s_{\delta-\alpha-\beta}s_{\alpha}s_{\beta}\cdots}')\subset$$
$$\cdots\subset\Phi_{s_{\alpha}s_{\delta-\alpha-\beta}s_{\alpha}s_{\beta}}' \subset \Phi_{s_{\alpha}s_{\delta-\alpha-\beta}s_{\alpha}}'\subset\Phi_{s_{\alpha}s_{\delta-\alpha-\beta}}'\subset \Phi_{s_{\alpha}}'\subset\widehat{\Phi}$$
This gives rise to a reflection order $\leq_2$:
$$\beta<_2\alpha+\beta<_2\beta+\delta<_2-\alpha+\delta<_2\beta+2\delta<_2\alpha+\beta+\delta<_2\beta+3\delta<_2-\alpha+2\delta\cdots$$
$$<_2-\beta+3\delta<_2\alpha+\delta<_2-\beta+2\delta<_2-\alpha-\beta+\delta<_2-\beta+\delta<_2\alpha.$$
Let $\Psi^+_2=\{-\alpha,\alpha+\beta,\beta\}$ be a positive system of $\Phi$.
$\widehat{\Psi_2^+}(=\Phi_{s_{\beta}s_{\alpha}s_{\delta-\alpha-\beta}s_{\alpha}s_{\beta}s_{\alpha}s_{\delta-\alpha-\beta}s_{\alpha}\cdots})$ is an initial section of $\leq_2.$

Let $R=\{\alpha,\beta,\alpha+\beta+\delta,\alpha+\delta,-\alpha+\delta,-\beta+\delta,-\alpha-\beta+\delta\}$.
Restricting $\leq_1$ and $\leq_2$ to $R$ we obtain
$$\alpha<_1\beta<_1\alpha+\beta+\delta<_1\alpha+\delta<_1-\alpha+\delta<_1-\beta+\delta<_1-\alpha-\beta+\delta,$$
$$\beta<_2-\alpha+\delta<_2\alpha+\beta+\delta<_2\alpha+\delta<_2-\alpha-\beta+\delta<_2-\beta+\delta<_2\alpha.$$

We illustrate how to convert $(R,\leq_1)$ to $(R,\leq_2)$ by successively reversing the dihedral substrings.
Note that $R\cap \widehat{\Psi^+_1}=\{\alpha,\beta,\alpha+\beta+\delta,\alpha+\delta\}$ and we have that $\alpha<_1\beta<_1\alpha+\beta+\delta<_1\alpha+\delta$.
These roots are contained in the inversion set $\Phi_{s_{\alpha}s_{\beta}s_{\alpha}s_{\delta-\alpha-\beta}s_{\alpha}}$. Set $w=s_{\alpha}s_{\beta}s_{\alpha}s_{\delta-\alpha-\beta}s_{\alpha}$ and $y=s_{\beta}s_{\alpha}s_{\delta-\alpha-\beta}s_{\beta}$. Then $(R\cap \widehat{\Psi^+_1})\backslash \widehat{\alpha}=\{\beta,\alpha+\beta+\delta\}\subset \Phi_y=\{\beta,\alpha+\beta,\beta+\delta,\alpha+\beta+\delta\}$.
The elements $w$ and $y$ are both bounded by $s_{\alpha}s_{\beta}s_{\alpha}s_{\delta-\alpha-\beta}s_{\alpha}s_{\beta}s_{\alpha}s_{\delta-\alpha-\beta}\cdots$.
Take their join $z=s_{\alpha}s_{\beta}s_{\alpha}s_{\delta-\alpha-\beta}s_{\alpha}s_{\beta}s_{\alpha}=s_{\beta}s_{\alpha}s_{\delta-\alpha-\beta}s_{\beta}s_{\delta-\alpha-\beta}s_{\alpha}s_{\beta}.$

Now we consider another maximal chain of biclosed sets in $\widehat{\Phi}$ (containing  $\Phi_z, \Phi_y$ and $\widehat{\Psi_1^+}$):
$$\emptyset\subset \Phi_{s_{\beta}}\subset \Phi_{s_{\beta}s_{\alpha}}\subset \Phi_{s_{\beta}s_{\alpha}s_{\delta-\alpha-\beta}}\subset \Phi_{s_{\beta}s_{\alpha}s_{\delta-\alpha-\beta}s_{\beta}}\subset \Phi_{s_{\beta}s_{\alpha}s_{\delta-\alpha-\beta}s_{\beta}s_{\delta-\alpha-\beta}}$$$$\subset \Phi_{s_{\beta}s_{\alpha}s_{\delta-\alpha-\beta}s_{\beta}s_{\delta-\alpha-\beta}s_{\alpha}}\subset \Phi_{s_{\beta}s_{\alpha}s_{\delta-\alpha-\beta}s_{\beta}s_{\delta-\alpha-\beta}s_{\alpha}s_{\beta}}\subset\cdots$$
$$\subset\Phi_{s_{\alpha}s_{\beta}s_{\alpha}s_{\delta-\alpha-\beta}s_{\alpha}s_{\beta}s_{\alpha}s_{\delta-\alpha-\beta}\cdots}(=\Phi_{s_{\delta-\alpha-\beta}s_{\alpha}s_{\beta}s_{\alpha}s_{\delta-\alpha-\beta}s_{\alpha}s_{\beta}s_{\alpha}\cdots}')\subset$$
$$\cdots \subset \Phi_{s_{\delta-\alpha-\beta}s_{\alpha}s_{\beta}s_{\alpha}}'\subset\Phi_{s_{\delta-\alpha-\beta}s_{\alpha}s_{\beta}}'\subset\Phi_{s_{\delta-\alpha-\beta}s_{\alpha}}'\subset\Phi_{s_{\delta-\alpha-\beta}}'\subset \widehat{\Phi}.$$
(Note that one has more than one way to fill in the missing part of the above chain. One just picks and fixes one choice.) This gives rise to a reflection order $\leq_3$:
$$\beta<_3 \alpha+\beta<_3 \beta+\delta<_3 \alpha+\beta+\delta<_3 \alpha<_3 \alpha+\beta+2\delta<_3 \alpha+\delta<_3\cdots$$
$$<_3-\beta+2\delta<_3-\alpha-\beta+3\delta<_3-\alpha+\delta<_3-\alpha-\beta+2\delta<_3-\beta+\delta<_3-\alpha-\beta+\delta.$$
Restricting $\leq_3$ to $R$ we obtain:
$$\beta<_3 \alpha+\beta+\delta<_3 \alpha<_3 \alpha+\delta<_3-\alpha+\delta<_3-\beta+\delta<_3-\alpha-\beta+\delta. $$
$(R,\leq_1)$ and $(R,\leq_3)$ are connected in the braid graph as we can convert the former to the latter by successively reversing the dihedral substring of finite parabolic groups. (These reversals correspond to braid moves of reduced expressions.)

The infinite dihedral reflection subgroup generated by $s_{\alpha},s_{\delta-\alpha}$ has the positive system $\widehat{\{\alpha,-\alpha\}}$ and the following two reflection orders:
$$\alpha<\alpha+\delta<\alpha+2\delta<\dots<-\alpha+3\delta<-\alpha+2\delta<-\alpha+\delta,$$
$$-\alpha+\delta<-\alpha+2\delta<-\alpha+3\delta<\dots<\alpha+2\delta<\alpha+\delta<\alpha.$$

Then by reversing the dihedral substring $\alpha<_3 \alpha+\delta<_3-\alpha+\delta$ (corresponding to the above dihedral reflection subgroup) we obtain
$(R,\leq_4)$ with $\beta<_4 \alpha+\beta+\delta<_4 -\alpha+\delta<_4 \alpha+\delta<_4\alpha<_4-\beta+\delta<_4-\alpha-\beta+\delta$ where $\leq_4$ is the reflection order
$$\beta\leq_4\alpha+\beta<_4 \beta+\delta<_4 \alpha+\beta+\delta <_4 \beta+2\delta<_4 -\alpha+\delta<_4 \beta+3\delta\cdots$$
$$<_4\alpha+\delta<_4-\beta+2\delta<_4 \alpha<_4 -\beta+\delta<_4 -\alpha-\beta+\delta$$
which comes from the maximal chain of biclosed sets in $\widehat{\Phi}$ (Again there are more than one way to fill in the missing part of the chain below. One just picks and fixes one choice):
$$\emptyset\subset \Phi_{s_{\beta}}\subset \Phi_{s_{\beta}s_{\alpha}}\subset \Phi_{s_{\beta}s_{\alpha}s_{\delta-\alpha-\beta}}\subset \Phi_{s_{\beta}s_{\alpha}s_{\delta-\alpha-\beta}s_{\beta}}\subset \Phi_{s_{\beta}s_{\alpha}s_{\delta-\alpha-\beta}s_{\beta}s_{\alpha}}$$$$\subset \Phi_{s_{\beta}s_{\alpha}s_{\delta-\alpha-\beta}s_{\beta}s_{\alpha}s_{\beta}}\subset \Phi_{s_{\beta}s_{\alpha}s_{\delta-\alpha-\beta}s_{\beta}s_{\alpha}s_{\beta}s_{\delta-\alpha-\beta}}\subset\cdots$$
$$\subset\Phi_{s_{\beta}s_{\alpha}s_{\delta-\alpha-\beta}s_{\alpha}s_{\beta}s_{\alpha}s_{\delta-\alpha-\beta}s_{\alpha}\cdots}(=\Phi_{s_{\alpha}s_{\delta-\alpha-\beta}s_{\alpha}s_{\beta}s_{\alpha}s_{\delta-\alpha-\beta}s_{\alpha}s_{\beta}\cdots}')\subset$$
$$\cdots\subset  \Phi_{s_{\delta-\alpha-\beta}s_{\alpha}s_{\delta-\alpha-\beta}s_{\beta}s_{\alpha}}'\subset \Phi_{s_{\delta-\alpha-\beta}s_{\alpha}s_{\delta-\alpha-\beta}s_{\beta}}'\subset \Phi_{s_{\delta-\alpha-\beta}s_{\alpha}s_{\delta-\alpha-\beta}}'$$$$\subset \Phi_{s_{\delta-\alpha-\beta}s_{\alpha}}'\subset \Phi_{s_{\delta-\alpha-\beta}}'\subset \widehat{\Phi}$$

Now to convert $(R,\leq_4)$ to $(R,\leq_2)$ one only needs to reverse dihedral substrings of finite parabolic subgroups. Note that $\beta,\alpha+\beta+\delta$ and $-\alpha+\delta$ are contained in the inversion set $\Phi_u$ where $u=s_{\beta}s_{\alpha}s_{\delta-\alpha-\beta}s_{\beta}s_{\alpha}s_{\beta}$. By applying braid moves, this element $u$ is equal to
$s_{\beta}s_{\delta-\alpha-\beta}s_{\alpha}s_{\delta-\alpha-\beta}s_{\beta}s_{\alpha}$. Also note that $-\alpha-\beta+\delta,-\beta+\delta,\alpha$ and $\alpha+\delta$ are contained in the inversion set  $\Phi_v$ where $v=s_{\delta-\alpha-\beta}s_{\alpha}s_{\delta-\alpha-\beta}s_{\beta}s_{\alpha}$. By applying braid moves, this element $v$ is equal to $s_{\alpha}s_{\delta-\alpha-\beta}s_{\alpha}s_{\beta}s_{\alpha}$.

Therefore we have a reflection order $\leq_5$:
$$\beta<_5-\alpha+\delta<_5\beta+\delta<_5\alpha+\beta<_5\beta+2\delta<_5\alpha+\beta+\delta<_5\cdots$$
$$<_5\alpha+\delta<_5-\beta+2\delta<_5-\alpha-\beta+\delta<_5-\beta+\delta<_5\alpha.$$

This reflection order arises from the following maximal chain of biclosed sets in $\widehat{\Phi}$ (For the missing part, choose and fix a choice as above.):
$$\emptyset\subset \Phi_{s_{\beta}}\subset \Phi_{s_{\beta}s_{\delta-\alpha-\beta}}\subset \Phi_{s_{\beta}s_{\delta-\alpha-\beta}s_{\alpha}}\subset$$ $$\Phi_{s_{\beta}s_{\delta-\alpha-\beta}s_{\alpha}s_{\delta-\alpha-\beta}}\subset \Phi_{s_{\beta}s_{\delta-\alpha-\beta}s_{\alpha}s_{\delta-\alpha-\beta}s_{\beta}}\subset \Phi_{s_{\beta}s_{\delta-\alpha-\beta}s_{\alpha}s_{\delta-\alpha-\beta}s_{\beta}s_{\alpha}}\subset$$
$$\cdots\subset \Phi_{s_{\beta}s_{\alpha}s_{\delta-\alpha-\beta}s_{\alpha}s_{\beta}s_{\alpha}s_{\delta-\alpha-\beta}s_{\alpha}\cdots}(=\Phi_{s_{\alpha}s_{\delta-\alpha-\beta}s_{\alpha}s_{\beta}s_{\alpha}s_{\delta-\alpha-\beta}s_{\alpha}s_{\beta}\cdots}')\subset \cdots$$
$$\subset\Phi_{s_{\alpha}s_{\delta-\alpha-\beta}s_{\alpha}s_{\beta}s_{\alpha}}'\subset\Phi_{s_{\alpha}s_{\delta-\alpha-\beta}s_{\alpha}s_{\beta}}'\subset\Phi_{s_{\alpha}s_{\delta-\alpha-\beta}s_{\alpha}}'\subset\Phi_{s_{\alpha}s_{\delta-\alpha-\beta}}'\subset\Phi_{s_{\alpha}}'\subset\widehat{\Phi}$$

Therefore $(R,\leq_1)$ and $(R,\leq_5)$ are connected in the braid graph. But $(R,\leq_5)=(R,\leq_2)$ and so we finish the process.
\end{example*}

\section{Acknowledgements}

The author acknowledges the support from Guangdong  Natural Science Foundation  Project 2018A030313581. The paper is based on part of author's dissertation.
 The author wishes to thank his advisor Dr. Matthew Dyer for his guidance, many helpful discussions and for reading the manuscript. The author thanks Jiefang Xu, Meili Miao, Xiao Xu, Yiqing Zou and Zhisheng Wang for their encouragement and help. The author thanks the anonymous referee for his/her  useful comments and some important corrections.

\end{document}